\documentclass[11pt]{article}

\usepackage[margin=1in]{geometry}
\usepackage{amsmath,amssymb,amsthm}
\usepackage{xcolor}
\usepackage{setspace}
\usepackage{subcaption}
\usepackage{graphicx}
\usepackage{url}
\usepackage{hyperref}

\RequirePackage{bm}
\RequirePackage{endnotes}

\usepackage{algorithm}
\usepackage{algpseudocode}
\usepackage{tikz}

\newcommand{\p}{\mathcal{P}}

\newcommand{\ttt}{\mathcal{T}}
\newcommand{\LL}{\mathcal{L}}

\newcommand{\M}{\mathcal{M}}

\newcommand{\T}{^{\top}}

\usepackage{natbib}
 \bibpunct[, ]{(}{)}{,}{a}{}{,}%
\newtheorem{theorem}{Theorem}
\newtheorem{lemma}{Lemma}

\theoremstyle{definition}
\newtheorem{definition}{Definition}
\newtheorem{assumption}{Assumption}
\theoremstyle{remark}
\newtheorem{remark}{Remark}

\renewenvironment{proof}[1][Proof]{\par\noindent\textit{#1.}\ }{\hfill$\blacksquare$\par\medskip}

\title{Anderson Accelerated Primal-Dual Hybrid Gradient for solving LP}
\author{
Yingxin Zhou \quad Cipolla Stefano \quad Phan T. Vuong \\[4pt]
\small School of Mathematical Sciences, University of Southampton \\
\small \texttt{yz11u24@soton.ac.uk}, \texttt{S.Cipolla@soton.ac.uk}, \texttt{T.V.Phan@soton.ac.uk}
}
\date{}

\begin{document}

\maketitle

\begin{abstract}
 {We present the Anderson Accelerated Primal--Dual Hybrid Gradient (AA-PDHG), a fixed-point-based framework that integrates Anderson Acceleration into the PDHG method for solving linear programming (LP) problems. A central motivation is to investigate whether Anderson Acceleration, which systematically exploits multi-step historical information, can serve as a viable alternative to the restart strategy for PDHG. We establish the global convergence of AA-PDHG under a safeguard condition and propose a filtered variant (FAA-PDHG) that enforces the uniform boundedness of the coefficient matrix through angle and length filtering, thereby providing a rigorous convergence guarantee.
Numerical experiments on LP instances derived from MIPLIB 2017 demonstrate that both AA-PDHG and FAA-PDHG deliver significant speedups over vanilla PDHG. On pre-solved MIPLIB instances, AA-PDHG is the fastest method on about $70$\% of the benchmark when neither method uses primal-weight updates, and remains competitive when both AA-PDHG and restart PDHG use their respective primal-weight update strategies, establishing Anderson Acceleration as a competitive alternative to the restart mechanism.}
\end{abstract}



\section{Introduction}\label{se:intro}
In recent years, first-order methods have become increasingly popular for solving large-scale convex optimisation problems, particularly due to their low per-iteration cost and favourable scalability. Among them, the Primal Dual Hybrid Gradient (PDHG) method has received considerable attention. If we consider a min-max problem in the form:
\[\min_{x} \max_{z} \; f(x) + z^{\T} K x - g^*(z),
\]
where \( f \) and \( g \) are proper, closed, and convex functions, \( g^* \) denotes the conjugate function of \( g \), and \( K \) is a linear operator, then the corresponding PDHG scheme can be written as
\begin{equation}\label{alg:PDHG_intro}
\left\{
\begin{aligned}
x^{k+1} &= \arg\min_{x} \left\{ f(x) + \frac{1}{2\tau} \left\| x - \left( x^k - \tau K^{\top} z^{k} \right) \right\|^2 \right\},\\
z^{k+1} &= \arg\min_{z} \left\{ g^*(z) + \frac{1}{2\sigma} \left\| z - \left( z^k + \sigma K \left(2 x^{k+1} - x^{k} \right) \right) \right\|^2 \right\},
\end{aligned}
\right.
\end{equation}
being $\tau,\;\sigma$ step-sizes. This algorithm is also commonly referred to as the  Chambolle–Pock algorithm \cite{ChambollePDHG}. When the updates in \eqref{alg:PDHG_intro} are specialised for the solution of linear programs (LP), see \eqref{eq:pdhg}, it is possible to observe a major advantage of PDHG when compared to other first or second-order optimisation methods, that is, it requires only matrix-vector multiplications at each iteration, thereby avoiding solution of linear systems and related expensive matrix factorisations. Such a characteristic is often referred to as matrix-free in the literature, as also discussed in~\cite{ChambollePDHG}. This characteristic, combined with the fact that the matrix-vector product is a highly parallelisable task, makes PDHG particularly attractive for large-scale applications. 

 {The use of PDHG for solving LP problems has attracted considerable attention in recent years}, see, e.g., \cite{PGLP87,SDLP89,AFOM-SPECIAL-LP11,ADMMLP17,AFOM-SPECIAL-LP19,fasterPDHG-23,google2022p,PDLPnew2025}, which have demonstrated that PDHG  {can} offer a more scalable alternative to classical LP methods, such as interior-point or simplex methods. In addition, \cite{LUGPU25} reported engineering results for restarted PDHG for linear programming from the GPU perspective. Related developments also include the work of {\cite{LiuCross24}}, who
studied the geometric behavior of PDHG for LP and proposed a crossover algorithm based on the
spiral trajectory of PDHG. { A related geometric analysis of PDHG was also provided by \cite{GEOMETRYPDHGLP-24}.} We also note that recent progress on large-scale LP is not limited to PDHG-based methods. For example, \cite{KCHPRLP26} proposed HPR-LP, which is based on a Halpern Peaceman Rachford splitting scheme with semi-proximal terms.

Nonetheless, when applied to LP problems, the PDHG method might still exhibit relatively slow convergence in practice. In particular, it often suffers from stagnation in the later stages of optimisation, especially when approaching optimal solutions.  This is because linear programming problems lack strong convexity, making PDHG achieve only sub-linear convergence rates \cite{first-order-beck}. As a result, obtaining  {meaningful} solutions within a reasonable computational time remains difficult, especially in large-scale or ill-conditioned scenarios.

 {Motivated by the above discussion, in this work, we explore the integration of Anderson Acceleration (AA) into the PDHG framework for solving LP problems more efficiently. A central motivation is the observation that the restart mechanism proposed by \cite{fasterPDHG-23}, while effective, represents only one possible strategy for accelerating PDHG. We investigate whether AA, which systematically exploits multi-step historical information from past iterates, can serve as a viable alternative acceleration mechanism, with the particular aim of achieving acceleration not only in the early stages of the iteration but also in the later stages.}

It is important to note that in some works it has been shown that LP problems possessing special properties, e.g., sharpness, see \cite{fasterPDHG-23}, or error bounds, see \cite{QEB-FE2023}, can support faster convergence rates for first-order methods. These properties can be leveraged to theoretically establish faster convergence rates for first-order methods applied to LPs.   {And indeed, much before the above-mentioned specialised LP approaches}, numerous accelerated variants of PDHG have been developed. Notable acceleration strategies include inertial techniques~\cite{IPDHG-16}, line search strategies~\cite{LINE-PDHG18}, coordinate updates~\cite{ACPDHG21}.  {As a result of the above-mentioned piece of literature, more in general, refined primal-dual schemes for LP problems have been proposed~\cite{google2022p,fasterPDHG-23,PDLPnew2025}. Among these, the restart strategy proposed by \cite{fasterPDHG-23} has emerged as a key ingredient and particularly effective approach to accelerate convergence. The restart mechanism periodically reinitialises the algorithm and leverages the sharpness structure of LP problems to achieve linear convergence from each restart epoch. This strategy has since been adopted as a key ingredient in practical first-order LP solvers \cite{google2022p,PDLPnew2025}.}

 {In contrast to the restart strategy, which discards historical information and reinitialises the iterates at prescribed intervals, Anderson Acceleration builds a low-rank model from multiple past iterates to compute improved updates. This fundamental difference in how historical information is utilised motivates the present investigation: we adopt a fixed-point perspective and study the use of AA to enhance the convergence behaviour of PDHG, with the aim of assessing whether such a history-based acceleration mechanism can serve as a viable alternative to the restart strategy of \cite{fasterPDHG-23}.}

Anderson Acceleration (AA) is a technique originally introduced in~\cite{AAoriginal} to accelerate fixed-point iterations. Since then, it has been widely adopted as an effective tool in a variety of applications. For a broader overview of its theoretical development and practical uses, we refer the interested reader to~\cite{AA-FPI-NI,AA-CON-TAK,AA-IMPROVE-PROOF-ECPRX,MR3841161,MR4926317}.

The core idea behind AA is to compute the next iterate as an affine combination of several previous iterates, thereby incorporating historical information to improve convergence. We  {summarise} the standard version of the AA method in Algorithm~\ref{al:AA-TYPE-II}.

\begin{algorithm}[!t]
\caption{Standard Anderson Acceleration with memory $m$ for solving $x=F(x)$}\label{al:AA-TYPE-II}
\begin{algorithmic}[1] 
\State  \textbf{Given} $F$: $\mathbb{R}^n\to \mathbb{R}^n$,  {an} initial point $x^0\in \mathbb{R}^n$, a damping parameter $\beta$, and a positive integer $m$. Set $x^1= F(x^0)$ and $G^0 = F(x^0)-x^0.$
\For{$k = 0, 1, \dots$}
    \State Set $m_k = \min \{m, k\}.$
    \State Set $\hat{G}^k=(G^{k-m_k},\ldots,G^{k})$, where $G^k=F(x^k)-x^k$.
    \State Solve the problem:\[
        \alpha^k:=\arg \min_{\alpha} \|\hat{G}^k\alpha\|,
        \text{\quad s.t. } \sum_{i=0}^{m_k} \alpha_i=1.
    \]
    \State Compute $x^{k+1}=(1-\beta)\sum_{i=0}^{m_k}\alpha^{k}_ix^{k-m_k+i}+\beta\sum_{i=0}^{m_k}\alpha^{k}_iF(x^{k-m_k+i}).$
\EndFor
\end{algorithmic}
\end{algorithm}

In Algorithm~\ref{al:AA-TYPE-II}, the scalar $0<\beta\leq1$ is a damping (mixing) factor, which is often set as~1 in most cases.  {  {Several works} have shown} that  {a smaller} choice of $\beta$  {can} help improve the  {numerical} stability and algorithm convergence in practice \cite{AACONTRACTIVE-21,AAMIXING19,dynamicBETAAA-25}. 

Anderson Acceleration (AA) has broad applicability. For instance, \cite{AA-FPI-NI} demonstrates that, for linear problems, AA is essentially equivalent to the generalized minimal residual (GMRES) method; see also \cite{MR4711308} for connections among Krylov methods and AA. Furthermore, \cite{AA-QUASI-NEWTON} and \cite{AA-NONLINEAR-21} interpret AA as a quasi-Newton method for nonlinear problems, thereby offering a theoretical basis for its acceleration properties.

 {Although} AA has demonstrated significant acceleration capability in practical applications, the theoretical analysis of its convergence has only been gradually developed in recent years. In \cite{AA-FPI-NI} it was observed that when the fixed-point iteration is governed by a contractive operator, the algorithm can achieve local linear convergence rates. \cite{EDIIS-19} further considers the so-called EDIIS algorithm as a special case of AA requiring $\alpha_{(k)} \geq 0,$ and proves its global convergence property under the case $F$ is contractive on a convex set. In \cite{AA-IMPROVE-PROOF-ECPRX}, the authors extended the result in \cite{AA-FPI-NI} for AA with damping factors, and observed that AA is more effective for linearly convergent fixed-point methods, whereas for algorithms with quadratic convergence, it may actually reduce the acceleration effect. In \cite{anderson-type-I} it was recently proposed a global convergence scheme for AA-I assuming the fixed-point mapping is non-expansive. Based on this, the scheme was further applied in \cite{anderson-accelerated-operatorsplitting}, where it served as a global convergence framework and was applied to splitting optimisation methods.   {Building on these developments}, we note that, as a member of the class of splitting methods, the PDHG algorithm can also be formulated as a fixed-point iteration. In particular, \cite{degenerate2021} investigated the convergence properties of PDHG in fixed-point form under degenerate conditions. 

\subsection{ {Contributions}}

 {While AA has been applied to various first-order algorithms -- including proximal gradient methods \cite{AA-PG-21}, extragradient methods \cite{EG-AA-24}, and general operator splitting \cite{anderson-accelerated-operatorsplitting} -- PDHG scheme combined with AA was proposed in \cite{RBNolinear23}.}   {However, the framework in \cite{RBNolinear23} addresses a general nonlinear setting and does not account for the specific challenges posed by LP. In particular, applying AA to PDHG for LP requires three key adaptations that are absent from prior work: (a)~an explicit projection step to restore feasibility after the AA update, since affine combinations of past iterates may violate the box and cone constraints inherent in LP; (b)~a diagonal preconditioning matrix $\widehat{D}$ within the quasi-Newton update, which exploits the structure of the constraint matrix to improve conditioning; and (c)~an analysis of the uniform boundedness of the AA coefficient matrix $H^k$, a condition that is assumed but not enforced or even discussed in \cite{RBNolinear23}. We address (c) through the FAA-PDHG variant, which algorithmically guarantees boundedness via angle and length filtering. Finally, no prior work has systematically compared AA-accelerated PDHG against the restart strategy on a large-scale LP benchmark.} The main contributions of this paper are summarised as follows.
\begin{enumerate}
    \item[\textbf{(i)}] We integrate Anderson Acceleration into the PDHG framework for LP and establish its global convergence under a safeguard condition (Theorem~\ref{th:con}).  {Moreover, as already mentioned, to} enhance numerical stability, we incorporate a diagonal correction $\widehat{D}$ into the AA update and, to ensure that the accelerated iterates remain feasible, we introduce an explicit projection step that addresses the key challenge arising when affine combinations of past iterates leave the feasible set.
    \item[\textbf{(ii)}] We propose a filtered variant (FAA-PDHG) that enforces the uniform boundedness of the coefficient matrix $H^k$ through angle and length filtering procedures (Theorem~\ref{th:FAA-convergence}), thereby providing a rigorous and algorithmically verifiable convergence guarantee.   {FAA-PDHG serves primarily as a theoretical device, demonstrating that the boundedness assumption required by the convergence analysis can be enforced at the algorithmic level. 
   In the numerical experiments, we employ the unfiltered AA-PDHG variant with moderate memory $m_A$, which was observed to be numerically stable while avoiding the overhead of explicit filtering.}
    \item[\textbf{(iii)}] From a theoretical standpoint, we highlight a key structural difference
    between the convergence guarantees of AA-PDHG and restart PDHG (rPDHG) \cite{fasterPDHG-23}. The convergence of rPDHG relies on problem-dependent quantities such as the sharpness
    constant, which is generally unknown and difficult to estimate in practice.
      {The global convergence of AA-PDHG is established under
    a safeguard condition and the assumption that the matrices $H^k$ are uniformly bounded.
    We note that both approaches rely on structural properties: rPDHG on sharpness (a problem-dependent quantity), and AA-PDHG on bounded $H^k$ (an algorithm-dependent quantity). 
    This gives a complementary type of convergence guarantee: whereas rPDHG exploits problem geometry through sharpness, whereas FAA-PDHG enforces an algorithmic boundedness condition on the Anderson coefficient matrices.}
    \item[\textbf{(iv)}] We conduct an extensive numerical comparison of AA-PDHG against rPDHG \cite{fasterPDHG-23} on   {$381$ LP relaxations from the} MIPLIB dataset.  {
    On the pre-solved instances, AA-PDHG achieves the fastest running time on approximately $70$\% of the benchmark when neither method uses primal-weight updates, and remains competitive when primal-weight updates are included. These results provide evidence that history-based acceleration via AA is a viable alternative to the restart mechanism for first-order methods applied to LP.}
\end{enumerate}

 {The remainder of this paper is organised as follows.}  {Section~\ref{se:Prelimi} recalls the necessary preliminary material.} In Section~\ref{se:problem-satement}, we present the formulation of the primal problem addressed in \cite{google2022p} and explain how its corresponding dual form is derived. In Section \ref{se:AA-PDHG}, we provide the specific formulation of the primal-dual problem in the form of a fixed-point iteration tailored to the given problem.  {The global convergence analysis is presented in Section~\ref{se:converge}.} In Section~\ref{se:FAA-PDHG}, we  {introduce FAA-PDHG, a filtered variant that ensures the uniform boundedness required for convergence.} Numerical  {experiments demonstrating the performance of the proposed algorithms} are presented in Section~\ref{se:NE}.

\section{Preliminaries}\label{se:Prelimi}

Throughout this work, $\mathbb{R}^n$ is the $n$-dimensional space whereas $\mathbb{R}^n_+ := \{ x \in \mathbb{R}^n \mid x_i \geq 0, \forall i = 1, \dots, n \}$ and $x^{\top}$ will denote the transpose of a vector $x$. Moreover, we define the distance of a point $x
$ from a {closed} set $\Omega$ as $\text{dist}(x, \Omega) := \min \{\, \| x - z\| ,\; z \in \Omega \,\}$. 
$\|\cdot\|$ will be used to denote the ${\ell}_2$-norm with the inner product $\langle\cdot,\cdot\rangle$ and the inner product induced by the  {quadratic form associated to the} operator $\mathcal{M}$ is defined as $\langle u,v\rangle_{\mathcal{M}}=\langle \mathcal{M}u,v\rangle$, where $\M$ is a bounded, positive semi-definite operator. The metric-induced norm $\|u\|_{\mathcal{M}}=\sqrt{\langle \mathcal{M}u,u\rangle}=\sqrt{\langle u,u\rangle_\mathcal{M}}$. Note that if $\M$ is only positive semi-definite, it is possible that $\|u\|_\mathcal{M}=0$ while $u \neq 0$. On the other hand, we will restrict our attention to the case where $\M$ is positive definite. Besides, we call $ f:\mathbb{R}^n\to(-\infty,\infty) $ a proper function if $f(x) < \infty $ for at least one $ x \in \mathbb{R}^n $, and $ f(x) > -\infty $ for all $ x \in \mathbb{R}^n $.
We denote by $\mathcal{I}$ the identity operator.

Let \( \mathcal{C} \subseteq \mathbb{R}^n \) be a nonempty, closed, and convex set. The indicator function $\mathcal{I}_{\mathcal{C}}$ of set \( \mathcal{C} \)  {is defined} as: \[
\mathcal{I}_{\mathcal{C}}(x) =
\begin{cases}
0, & x \in \mathcal{C}, \\
+\infty, & x \notin \mathcal{C}.
\end{cases}
\]

\begin{definition}[Projection Operator \cite{first-order-beck}]\label{de:projection}
    \( \p_{\mathcal{C}}: \mathbb{R}^n \to \mathcal{C} \) is a mapping that assigns to each \( x \in \mathbb{R}^n \) the closest point in \( \mathcal{C} \), given by:
\[
\p_\mathcal{C}(x) := \arg\min_{y \in \mathcal{C}} \| x - y \|.
\]
\end{definition}
This means that \( \p_\mathcal{C}(x) \) is the unique point in \( \mathcal{C}\) that minimizes the Euclidean distance to \( x \).  It is important to note that \( \p_\mathcal{C} \) is a \textit{non-expansive} mapping that satisfies:
   \[
   \| \p_\mathcal{C}(x) - \p_\mathcal{C}(y) \| \leq \| x - y \|, \quad \forall x,\; y \in \mathbb{R}^n.
   \]

\begin{definition}[Proximal Operator \cite{first-order-beck}]
Given a function \( f : \mathbb{R}^n \to \mathbb{R} \cup \{+\infty\} \), the proximal mapping of \( f \) is the operator given by
\[
{\text{Prox}}_f(x) := \arg\min_{u \in \mathbb{R}^n} \left\{ f(u) + \frac{1}{2} \|u - x\|^2 \right\},
\text{ for any x} \in \mathbb{R}^n .\]
\end{definition}
\begin{definition}[Subgradient and Subdifferential \cite{first-order-beck}]\label{de:subgradient}
    Let a function \( f : \mathbb{R}^n \to (-\infty, \infty] \) be a proper function and let \( x \in \operatorname{dom}(f) \). A vector \( g \in \mathbb{R}^n \) is called a {subgradient} of \( f \) at \( x \) if the following inequality holds:
\[
f(y) \geq f(x) + \langle g, y - x \rangle, \quad \text{for all } y \in \mathbb{R}^n.
\]
    The set of all subgradients of a function \( f \) at a point \( x \) is called the {subdifferential}  of \( f \) at \( x \), and is denoted by \( \partial f(x) \), i.e.,
\[
\partial f(x) = \left\{ g \in \mathbb{R}^n \;\middle|\; f(y) \geq f(x) + \langle g, y - x \rangle, \quad \text{for all } y \in \mathbb{R}^n \right\}.
\]
\end{definition}
Notably, the subdifferential of the indicator function $\mathcal{I}_\mathcal{C}$ is precisely the normal cone of \( \mathcal{C} \) at \( x \) \cite[Cor. 12.18]{variational-Rock}, where the \textit{normal cone} to \( \mathcal{C} \) at \( x \)  {is defined} as:
\[
 \mathcal{N}_\mathcal{C}(x) = \{ v \in \mathbb{R}^n \mid \langle v, z - x \rangle \leq 0, \quad \forall z \in \mathcal{C} \}=\partial \mathcal{I}_\mathcal{C}(x).
\]
If the set $\mathcal{C}=\mathbb{R}^n_+$, then for $v\in  \mathcal{N}_\mathcal{C}(x)$, we have $v^{\T}x =0$ for any $x \in \mathcal{C}$, see \cite[Example 6.10]{variational-Rock}.

\begin{lemma}[\cite{first-order-beck} Th. 6.39]\label{le:subgrad-prox}
    Given a function \( f : \mathbb{R}^n \to \mathbb{R} \cup \{+\infty\} \), let $x,\;u\in \mathbb{R}^n$. Then $u=\text{Prox}_f(x)$ if and only if $x-u\in\partial f(u).$
\end{lemma}
\begin{definition}[Conjugate Function \cite{first-order-beck}]\label{de:conjugate-function}
    Let a function \( f: \mathbb{R}^n \to \mathbb{R}\cup \{+\infty\} \) be a proper function.  
The {conjugate function} \( f^*: \mathbb{R}^n \to \mathbb{R} \cup \{+\infty\} \),  {is defined} as:
\[
f^*(y) = \sup_{x \in \mathbb{R}^n} \big( \langle y, x \rangle - f(x) \big), \quad y\in\mathbb{R}^n.
\]
\end{definition}
If $f$ is the indicator function \(\mathcal{I}_{\mathcal{C}}(x)\), then $f^*(y)=\sup_{x \in \mathcal{C} }\langle y, x \rangle$.

\begin{definition}[Quasi-Fejér Monotone \cite{quasi-Fejer}]\label{def:quasi-fejer}
The sequence \( \{z^k\} \subset \mathbb{R}^{n} \) is quasi-Fejér monotone with respect to a non-empty target set \( C \subset \mathbb{R}^{n} \), if for any \( z \in C \), there exists a nonnegative and summable sequence \( \varepsilon^k \), such that for any \( k \geq 0 \), we have
\[
\|z^{k+1} - z\|^2 \leq \|z^k - z\|^2 + \varepsilon^k.
\]
The definition remains valid when \( \|\cdot\| \) is replaced by a metric-induced norm \( \|\cdot\|_{\mathcal{M}} \), see \cite[Definition 3.1]{quasi-Fejer}.
\end{definition} 

\begin{lemma}[\cite{quasi-Fejer} Th. 3.3]\label{le:strong-con}
    Let \( \{z^k\} \) be a quasi-Fejér monotone sequence with respect to a non-empty target set \( C \subset \mathbb{R}^{n} \). Then \( \{z^k\} \) converges to a point in \( C \) if and only if for any limit point \( z \) of \( \{z^k\} \), we have \( z \in C \).
\end{lemma} 

\subsection{Problem Statement}\label{se:problem-satement}

 {In this section, we present the specific form of the linear programming (LP) 
	problem considered in this work, see \eqref{LP-O}, and derive the corresponding dual formulation, see \eqref{DUAL-2}.  
	Beyond setting up the notation, this derivation provides the theoretical foundation for the KKT-based termination criteria introduced in  {Remark~\ref{re:constrcution-termi-con}.}
    In particular, as in \cite{google2022p}, the primal form of the LP problem here considered is: } 
\begin{equation}\label{LP-O}
\begin{aligned}
\min\quad & c^{\top}x \\
\text{s.t.}\quad & Gx \geq h, \\
& Ax = b, \\
& l \leq x \leq u.
\end{aligned}
\end{equation}
where $
G \in \mathbb{R}^{m_1 \times n}, \; A \in \mathbb{R}^{m_2 \times n}, \; c \in \mathbb{R}^{n}, \; h \in \mathbb{R}^{m_1}, \; b \in \mathbb{R}^{m_2},\;
l \in (\mathbb{R} \cup \{-\infty\})^n, \; u \in (\mathbb{R} \cup \{\infty\})^n.
$ In order to get its dual form, we first write down the Lagrangian of problem \eqref{LP-O} as:
\begin{equation}\label{LAG}
\begin{array}{l}
\LL(x,y_{1:m_1},y_{m_1+1:m},\lambda^-,\lambda^+) := \\
\quad c^{\top}x + y_{1:m_1}^{\top}(-Gx+h) + y_{m_1+1:m}^{\top}(b-Ax)
+ (\lambda^-)^{\top}(x-u) + (\lambda^+)^{\top}(l-x)
\end{array}
\end{equation}
where $m=m_1+m_2$ and where the dual variables are given by $ y_{1:m_1}\in\mathbb{R}_+^{m_1},\; y_{m_1+1:m}\in {\mathbb{R}^{m_2}},\; \;\lambda^-\in \mathbb{R}_+^n,\; \lambda^+\in \mathbb{R}_+^n$. The corresponding dual function is then
\[
\begin{aligned}
&\mathcal{K}(y_{1:m_1}, y_{m_1+1:m}, \lambda^-, \lambda^+) 
:= \inf_{x} \LL(x,y_{1:m_1},y_{m_1+1:m},\lambda^-,\lambda^+) \\
&=  { \inf_{x} \{c^{\top}x + y_{1:m_1}^{\top}(-Gx+h) + y_{m_1+1:m}^{\top}(b-Ax)
+ (\lambda^-)^{\top}(x-u) + (\lambda^+)^{\top}(l-x) \} }\\
&= \inf_{x} \{ (c - G^{\top}y_{1:m_1} - A^{\top}y_{m_1+1:m} + \lambda^- - \lambda^+)^{\top}x +  h^{\top}y_{1:m_1} + b^{\top}y_{m_1+1:m} + l^{\top}\lambda^+ - u^{\top}\lambda^- \}.
\end{aligned}
\]
Clearly, when  $c-G^{\top}y_{1:m_1}-A^{\top}y_{m_1+1:m}+\lambda^--\lambda^+=0$, then ${\mathcal{K}(y_{1:m_1}, y_{m_1+1:m}, \lambda^-, \lambda^+)}$ attains a finite value. Substituting this condition into \eqref{LAG}, we get
\begin{equation}\label{DUAL-1}
\begin{aligned}
\mathcal{K}(y_{1:m_1}, y_{m_1+1:m}, \lambda^-, \lambda^+) = h^{\top}y_{1:m_1} + b^{\top}y_{m_1+1:m} + l^{\top}\lambda^+ - u^{\top}\lambda^-,
\end{aligned}
\end{equation}
when $c-G^{\top}y_{1:m_1}-A^{\top}y_{m_1+1:m}+\lambda^--\lambda^+=0$. Now, letting 
\begin{equation}
\begin{aligned}\label{lambda-case}
K^{\top} &= \begin{bmatrix} G^{\top}, A^{\top} \end{bmatrix}, \quad q^{\top} := \begin{bmatrix} h^{\top}, b^{\top} \end{bmatrix},\; y^{\top} = \begin{bmatrix} y_{1:m_1}^{\top}, y_{m_1+1:m}^{\top} \end{bmatrix}, \; \lambda = \lambda^+-\lambda^-,\\
\Lambda &= \Lambda_1 \times \cdots \times \Lambda_n, \quad
\Lambda_i =
\begin{cases}
\{0\}, & l_i = -\infty, \quad u_i = \infty, \\
\mathbb{R}_-, & l_i = -\infty, \quad u_i \in \mathbb{R}, \\
\mathbb{R}_+, & l_i \in \mathbb{R}, \quad u_i = \infty, \\
\mathbb{R}, & l_i \in \mathbb{R}, \quad u_i \in \mathbb{R},
\end{cases}
\end{aligned}
\end{equation}
and using such definitions in \eqref{DUAL-1}, the dual problem of \eqref{LP-O} can be formulated
\begin{equation}
\begin{aligned}\label{DUAL-2}
    \max_{y,\lambda} \;&q^{\top}y+l^{\top}\lambda^+-u^{\top}\lambda^-\\
    \text{s.t.\;}& c-K^{\top}y = \lambda,\\
    & y_{1:m_1} \geq 0,\\
    & \lambda \in \Lambda.
\end{aligned}
\end{equation}
For more details about the derivation of the constraint set $\Lambda$, see Supplementary Material Sec. \ref{appen:Lambda}.


\section{Anderson Accelerated PDHG}\label{se:AA-PDHG} 
As discussed in Section \ref{se:intro}, the goal of this work is to accelerate the convergence of PDHG using AA. On the other hand, in general, AA is designed to accelerate fixed-point iterations;  for this reason, in this section, we first show that the PDHG algorithm can be written in a fixed-point fashion. Afterwards, in Algorithm \ref{al-PD-safe}, we present the pseudo-code of the computational framework considered in this work. 

We begin by presenting the specific form of the PDHG method when applied to problem \eqref{LP-O}. To this aim, as in \cite{google2022p}, we use min-max formulation of such problem:
\[
\min_{x\in X}\max_{y \in Y} L(x, y) := c^\top x - y^\top (Kx - q)
\]
{where  $X:=\{x\mid l\leq x \leq u\}$ and $Y:=\mathbb{R}_+^{m_1} \times \mathbb{R}^{m_2}$}. Through the use of indicator functions, it can be equivalently reformulated as:
\begin{equation}
\begin{aligned}
\min_{x\in X}\max_{y \in Y} L(x, y)&=\min_{x\in\mathbb{R}^n}\max_{y\in\mathbb{R}^m}c^{\T}x+y^{\T}(q-Kx)+\mathcal{I}_{\{l\leq x\}}(x)+\mathcal{I}_{\{x\leq u\}}(x)-\mathcal{I}_{Y}(y)
 \\
&=\min_{x\in\mathbb{R}^n}\max_{y\in\mathbb{R}^m} c^\top x - y^\top (Kx - q) + \mathcal{I}_{{X}}(x)-\mathcal{I}_{{Y}}(y),\label{lag-con}
\end{aligned}
\end{equation}
where the first equality holds because we performed a decomposition on the constraint $X$, i.e., $X=\{x\mid l\leq x\}\cap \{x\mid x\leq u\}$. From \eqref{lag-con}, we can write the PDHG iteration as:
\begin{subequations} \label{eq:pdhg}
\begin{align}
x^{k+1} &= \p_X\left(x^k - \tau (c - K^\top y^k)\right), \label{pdlp-x-update} \\
y^{k+1} &= \p_Y\left(y^k + \sigma (q - K(2x^{k+1} - x^k))\right). \label{pdlp-y-update}
\end{align}
\end{subequations}

We will show, in the following, that the iteration \eqref{eq:pdhg} can be rewritten in fixed-point form. This reformulation is inspired by the approach presented in \cite[Eq. (3.3)]{degenerate2021} where the authors show, indeed, that the PDHG method can be reformulated in a fixed-point fashion. However, in their analysis, the reformulation is stated for problems having the composite form 
\begin{equation}
\begin{aligned}\label{problem:f+g}
\min_{x\in \mathbb{R}^n} f(x)+g(Kx),
\end{aligned}
\end{equation}
where $K$ is a generic bounded linear operator.
Following this idea, we first rewrite \eqref{lag-con} as a problem having the form \eqref{problem:f+g}. Indeed, \eqref{lag-con} can be equivalently written as:
\begin{equation}
\begin{aligned}
&\min_{x\in\mathbb{R}^n}\max_{y\in\mathbb{R}^m} c^\top x - y^\top (Kx - q) + \mathcal{I}_{{X}}(x)-\mathcal{I}_{{Y}}(y)\\
 =&\min_{x\in\mathbb{R}^n}c^\top x+\mathcal{I}_{{X}}(x)+\max_{y\in\mathbb{R}^m}\Big\{\langle -Kx+q,y\rangle -\mathcal{I}_{Y}(y)\Big\}\\
 =&\min_{x\in\mathbb{R}^n}c^\top x+\mathcal{I}_{{X}}(x)+\mathcal{I}^*_{Y}(-Kx+q)\\
=&\min_{x\in\mathbb{R}^n}c^\top x+\mathcal{I}_{{X}}(x)+\mathcal{I}_{Q}(Kx)\label{lag-mini},
\end{aligned}
\end{equation}
where $\mathcal{I}^*_{Y}(y)$ is the conjugate function of $\mathcal{I}_{Y}(y)$, see Section \ref{se:Prelimi}, and where $Q$ is the set 
$$Q:=\{Kx\mid (Kx)_i=q_i \text{ for all } i = m_1+1,\dots,m_1+m_2,\;(Kx)_i \geq  q_i \text{ for } i = 1, \dots, m_1 \}.$$
The full details about the last equality are provided in Supplementary Material Sec. \ref{appen:IQKX}. Hence, setting, {$f(x)=c^{\T}x+\mathcal{I}_{X}(x)$ and $g(Kx)=\mathcal{I}_Q(Kx)$}, we finally proved the equivalence of the formulation \eqref{problem:f+g}
and the formulation \eqref{lag-mini}.
 
Following \cite{degenerate2021}, we can finally write PDHG in the fixed-point form:
\begin{equation}
\begin{aligned} \label{new-fixed-point}
    u^{k+1}=(\mathcal{M}+\mathcal{A})^{-1}\mathcal{M} u^k=(\mathcal{I}+\M^{-1}\mathcal{A})^{-1}u^k=\ttt (u^{k}),
\end{aligned}
\end{equation}
where \[
u = \begin{bmatrix} x \\ y \end{bmatrix},
\quad \mathcal{A} := 
\begin{bmatrix}
\tau \bar{A} &\; -\tau K^{\top} \\
\sigma K &\; \sigma \bar{B}^{-1}
\end{bmatrix}, 
\quad
\mathcal{M} := 
\begin{bmatrix}
 I & \tau K^{\top} \\
 \sigma K &I
\end{bmatrix},\quad \ttt = (\mathcal{I}+\M^{-1}\mathcal{A})^{-1}
,\] $\bar{A}=c+ \partial \mathcal{I}_{X}$ and $\bar{B}=\partial \mathcal{I}_{Q}$. Parameters $\sigma,\;\tau$ are the same as in \eqref{pdlp-x-update} and \eqref{pdlp-y-update}.

\begin{assumption}
In the remainder of this work, we assume that the step sizes $\tau$ and $\sigma$ are such that $\sigma\tau\|K\|^2<1$. This ensures the positive definiteness of $\M$ and that $\ttt$ is an $\M$-firmly non-expansive operator. Moreover, in this case, the induced semi-norm $\|\cdot \|_{\M}$ introduced in Section~\ref{se:Prelimi} becomes a norm. 
\end{assumption}

\begin{remark}\label{remark-lambda}
  The dual variable $\lambda$ introduced in \eqref{lambda-case} can be characterised,  { at optimality, }as an element of the subgradient set of the indicator function $\mathcal{I}_{X}$ in \eqref{lag-mini} . A simple explanation is shown in Supplementary Material Sec. \ref{appen:Lambda}. Please note that we will freely use this result in subsequent proofs. 
\end{remark}

Before introducing our computational framework, we note that, in general, AA relies on affine combinations of multiple historical iterates, see Algorithm \ref{al:AA-TYPE-II}. When applied to the fixed-point iterations  {defined} in \eqref{new-fixed-point}, such combinations do not necessarily preserve feasibility w.r.t. $X$ and $Y$,  whereas PDHG 
inherently produces feasible iterates -- see also equations \eqref{pdlp-x-update} and \eqref{pdlp-y-update}. In the context of PDHG, this issue is particularly undesirable. To address this problem and preserve feasibility in the AA-accelerated scheme, considering $$\p:\mathbb{R}^n\times \mathbb{R}^m \to X\times Y $$ the projection operator onto the feasible set. Specifically, this projection can be decomposed into:
\[
\mathcal{P} =
\begin{bmatrix}
\p_X & 0 \\
0 & \p_Y
\end{bmatrix},
\]
{where $\p_X:\mathbb{R}^n\to X$ and $\p_Y:\mathbb{R}^m\to Y$.}

We are now ready to present, see Algorithm \ref{al-PD-safe}, our proposed Anderson Accelerated Primal Dual Hybrid Gradient, denoted by \textit{AA-PDHG} in the following.  The specific structure of Algorithm~\ref{al-PD-safe} incorporates several key mechanisms to ensure soundness and convergence, see Section \ref{se:converge}.

The core of such a mechanism is the safeguard step included at Line~\ref{algline:safeguard}. This step adopts the strategy introduced in \cite{ANDERSON-DRS-Fu_2020,anderson-accelerated-operatorsplitting,anderson-type-I} to guarantee global convergence. In particular, when the new fixed-point residual norm $\|g^{k}\|$ does not satisfy the safeguard condition, the algorithm defaults to a classical PDHG iteration to compute the next approximation, as detailed in Line~\ref{al-update-PDHG-uk+1}. To provide insight into the interplay between acceleration and fallback strategies, the algorithm tracks the number of times Anderson Acceleration (AA) and vanilla PDHG are used. Specifically, if up to iteration \( k \), the AA step (see Line~\ref{al-update-AA-uk+1}) has been invoked \( i^{k} \) times, and the PDHG step (see Line~\ref{al-update-PDHG-uk+1}) has been applied \( j^{k} \) times, then at iteration \( k+1 \), we increment either \( i^{k+1} = i^{k} + 1 \) or \( j^{k+1} = j^{k} + 1 \).

Concerning the embedded AA procedure, we observe that the AA memory, see Line \ref{algline:AA_mem} is updated also when the safeguard step is not used in order to ensure that any acceleration step is performed using the most updated information.

 {It is important to note that} Line~\ref{algline:AA_update} does not represent a classical AA update, as the algorithm uses a regularisation technique originally proposed in \cite{AA-NONLINEAR-21} to update the accelerated iterate \( u^{k+1}_{\mathrm{AA}} \). This step is particularly important when the matrix \( (\Delta \mathcal{G}^{k-m_k})^T \Delta \mathcal{G}^{k-m_k} \) is non-invertible or poorly conditioned.  {A diagonal matrix \( \widehat{D} \) is also incorporated in the update. This matrix plays a role analogous to the preconditioning matrix discussed in \cite{diagonalPDHG-11}, helping to moderate the effect of scaling across different components of the update. The specific construction of $\widehat{D}$ used in our experiments is described in Section~\ref{se:Dhat_construction}.}

\begin{algorithm}[!t]
\caption{Anderson Accelerated Primal-Dual Hybrid Gradient (AA-PDHG)}\label{al-PD-safe}
\begin{algorithmic}[1]

\State \text{Given} initial feasible point \(u^0:=(x^0,y^0)\), a projection operator
\(\mathcal{P}:\mathbb{R}^n\times\mathbb{R}^m\to X\times Y\), and the fixed-point map
\(\mathcal{T}\)  {defined} in \eqref{new-fixed-point}. Select \(D>0\) and a diagonal matrix \(\widehat{D}\).
Let  {\(m_A\ge 1\)}, \(\varepsilon>0\), \(tol>0\). Set \(i^1=0\), \(j^1=0\).
Set \(u^{1}=\mathcal{T}u^{0}\), \(g^{0}:=u^{1}-u^{0}\).

\For{\(k=1,2,\ldots\)}
  \State \(\widehat{u}^{k+1}= \mathcal{T}u^{k}\)
  \State \(g^{k}= \widehat{u}^{k+1}-u^{k}\) \label{algline:fixedpointres}

  \If{\(\|g^{k}\|<tol\)}
    \State \text{break}
  \EndIf

  \State \(m_k = \min\{ {m_A},k\}\)
\State For $j=1,\ldots,m_k$ \label{algline:AA_mem} compute:
\[
\begin{array}{ll}
\Delta u^{k-j} := u^{k-j+1}-u^{k-j}, &\quad \Delta g^{k-j} := g^{k-j+1}-g^{k-j} \\
\Delta\mathcal{U}^{k-m_k} := [\Delta u^{k-m_k},\ldots,\Delta u^{k-1}], &\quad \Delta\mathcal{G}^{k-m_k} := [\Delta g^{k-m_k},\ldots,\Delta g^{k-1}]
\end{array}
\]
  \If{\(\|g^{k}\|\le D\|g^{0}\|(i^{k}+1)^{-(1+\varepsilon)}\)}\label{algline:safeguard} 
    \State Compute \label{algline:AA_update}: 
    \begin{equation*}
    \begin{aligned}
    &u^{k+1}_{AA}:=u^{k}-H^{k}g^{k},\\ 
    &\text{where}\;
      H^{k}=-\beta \widehat{D}
      +\big(\Delta\mathcal{U}^{k-m_k}+\beta \widehat{D}\,\Delta\mathcal{G}^{k-m_k}\big)
      \big((\Delta\mathcal{G}^{k-m_k})^{\top}\Delta\mathcal{G}^{k-m_k}+\eta I\big)^{-1}
      (\Delta\mathcal{G}^{k-m_k})^{\top}.
    \end{aligned}
\end{equation*}   
    \State Set \label{al-update-AA-uk+1} $\begin{aligned}[t]
u^{k+1} &= \mathcal{P}u_{AA}^{k+1}, \\
i^{k+1} &= i^k + 1.
\end{aligned}$
\Else
\State Set \label{al-update-PDHG-uk+1} $\begin{aligned}[t]
u^{k+1} &= \widehat{u}^{k+1}, \\
j^{k+1} &= j^k + 1.
\end{aligned}$
  \EndIf
\EndFor
\end{algorithmic}
\end{algorithm}

 {Having presented  AA-PDHG (Algorithm \ref{al-PD-safe}),  two remarks are in order:}

 {\begin{remark}[Feasibility preservation]\label{rem:convex_AA}
		An alternative approach to preserving feasibility would be to constrain the AA coefficients to be nonnegative and sum to one, thereby computing a convex combination of past iterates that automatically lies in the feasible set. However, this would significantly restrict the expressiveness of the AA update, as the ability to form general affine combinations is essential for the acceleration properties of AA. Indeed, the connection between AA and GMRES for linear problems \cite{AA-FPI-NI} relies precisely on the affine (not convex) structure of the combination. We therefore adopt the projection-based approach, which preserves the full acceleration capability of AA while ensuring feasibility through an inexpensive additional step.
\end{remark}}

 {\begin{remark}[Parameter selection]\label{rem:param_selection}
Algorithm~\ref{al-PD-safe} involves several parameters whose selection we briefly discuss.
The safeguard parameter $D > 0$ controls how frequently the AA step is accepted: larger values of $D$ relax the safeguard and allow more frequent AA updates, while smaller values enforce stricter control and more frequent fallback to vanilla PDHG. The exponent $\varepsilon > 0$ governs the asymptotic decay rate of the safeguard threshold. The damping parameter $\beta \in (0,1]$ is fixed at $\beta = 1$ (no damping), which is the standard choice in most AA implementations. The Tikhonov regularisation parameter $\eta > 0$ addresses potential ill-conditioning in the AA least-squares subproblem . The AA memory size $m_A$ determines the   {maximum} number of past iterates used in the acceleration step. Values in the range $5$ to $20$ are typical in the literature \cite{AA-CON-TAK,AA-FPI-NI,AA-IMPROVE-PROOF-ECPRX}. A detailed experimental performance analysis of these parameters is provided in Section~\ref{appen:hyper} of the Supplementary Material.
\end{remark}}

\subsection{Convergence analysis}\label{se:converge}
In this section, we present the proof of global convergence of Algorithm~\ref{al-PD-safe}. To this aim, for the sake of completeness,  we prove that the fixed-point residual generated by the vanilla iteration \eqref{new-fixed-point}, converges to zero, see also \cite{degenerate2021}.

\begin{theorem}\label{th:pdhg-residual-0}
    Let the sequence \( \{u^k\} \) be generated by  $u^{k+1} = \mathcal{T} u^k$. Then, the fixed-point residual converges to zero, i.e.,
\begin{equation}
\begin{aligned}
    \lim_{k \to \infty} \| \mathcal{T} u^k -  u^k \| = 0.
\end{aligned}
\end{equation}
\end{theorem}
\begin{proof}{Proof}
The proof is provided in Supplementary Material Sec. \ref{appen:fix-res}.
\end{proof}

To ensure the convergence of the iteration \eqref{new-fixed-point} when coupled with AA, we follow the strategy of incorporating a safeguard check that controls the behaviour and activation of AA. With this design, we now claim that the safeguard condition will be satisfied infinitely many times, i.e.,  {Line~\ref{algline:AA_update}-\ref{al-update-AA-uk+1}} will be used infinitely many times in Algorithm \ref{al-PD-safe}.

\begin{lemma}\label{le-AA-use-inifite}
In Algorithm \ref{al-PD-safe}, the safeguard check is passed infinitely many times, i.e., the Anderson Acceleration step at Line \ref{al-update-AA-uk+1} is used infinitely often.
\end{lemma}
\begin{proof}{Proof}
The proof is provided in Supplementary Material Sec. \ref{appen:AA-infi}.
\end{proof}

Let us introduce some necessary notation useful in the following.  Regardless of how the algorithm alternates between AA and PDHG steps, for the sake of clarity, we divide the index set \( \mathbb{N}=\{0,1,\ldots\} \) of iterations into two subsets: $K_{aa}:=\{a_0, a_1,a_2,\ldots\} $ and $K_{pd}=\{p_0, p_1,\ldots\}$. So, $K_{aa}$ represents the iteration indices where the safeguard condition is satisfied and Line \ref{al-update-AA-uk+1} is applied, whereas
$K_{pd}$ represents the iteration indices where Line \ref{al-update-PDHG-uk+1} is used. It is important to note by Lemma \ref{le-AA-use-inifite}  we know that $K_{aa}$ is an infinite set.

In the next lemma, we prove that the sequence $\{u^k\}$ generated by Algorithm~\ref{al-PD-safe} is a quasi-Fejér monotone sequence with respect to non-empty set $Fix\ttt $.

\begin{lemma}\label{le-quasi-fejer}
Assume that matrix $H^k$ generated by Algorithm \ref{al-PD-safe} is uniformly bounded, i.e., $\|H^k\|\leq M$ for all $k$. Consider $\{u^k\}$  generated by Algorithm~\ref{al-PD-safe} and assume the point $u^*\in Fix\ttt$. Then $\{u^k\}$ is a quasi-Fejér monotone sequence with respect to the set $Fix\ttt $ {under the norm induced by $\mathcal{M}$}.
\end{lemma}

\begin{proof}{Proof}
The proof is provided in Supplementary Material Sec. \ref{appen:AA-fejer}.
\end{proof}

  {\begin{remark}[On the boundedness assumption]\label{rem:boundedness_assumption}
The following theorem requires the uniform boundedness of the matrices $H^k$, i.e., $\|H^k\| \leq M$ for all $k$. We emphasise that this is an \emph{assumption} on the iterates of Algorithm~\ref{al-PD-safe}, not a property guaranteed by the algorithm itself. The Tikhonov regularisation parameter $\eta$ ensures the well-posedness of each individual AA update but does not, by itself, imply a uniform bound on $\|H^k\|$ across all iterations (see the discussion in Section~\ref{se:rpdhg_comparison}). In Section~\ref{se:FAA-PDHG}, we introduce the FAA-PDHG variant, which enforces this bound algorithmically through ad-hoc procedures. 
\end{remark}}

We are finally able to show that the sequence generated by Algorithm \ref{al-PD-safe} converges globally.

\begin{theorem}\label{th:con}
Assume that the sequence \( \{u^k\} \) is generated by Algorithm~\ref{al-PD-safe}, where \( H^k \) satisfies \( \|H^k\| \leq M \). Then \( \{u^k\} \) converges globally to a fixed-point $u^*\in Fix\ttt$.
\end{theorem}

\begin{proof}{Proof}
  {Let $\lambda_{\max}{(\M)}$ be the biggest eigenvalue of matrix $\M$.}
 {Let us define the constant
\[
E:=\| u^{0}- u^*\|_\M + M\sqrt{\lambda_{\max}(\M) }D\|g^0\|\sum_{j=0}^{\infty}(j+1)^{-(1+\varepsilon)},
\]
which provides a uniform upper bound on $\|u^k - u^*\|_\M$ for all $k \geq 0$, as established in the proof of Lemma~\ref{le-quasi-fejer} (see Supplementary Material Sec.~\ref{appen:AA-fejer}).}
Let us recall, that by \eqref{uk-sequence-diff-bound}, for $k=a_{l}\in K_{aa}$, we have
\begin{equation}
\begin{aligned}\label{eq:ki-aa-2}
&\|u^{k+1}-u^*\|_\M^2-\|u^{k}-u^*\|_\M^2\\
 & \leq M^2D^2{  { {\lambda_{\max}(\M)}}}\|g^0\|^2(i^{k}+1)^{-2(1+\varepsilon)}+2EM{  { \sqrt{\lambda_{\max}(\M)}}} D\|g^0\|(i^{k}+1)^{-(1+\varepsilon)}.
\end{aligned}
\end{equation}
Similarly, by \eqref{eq:kpd_PDHG-used-1}, for $k=p_{l}\in K_{pd}$, we have
\begin{equation}
\begin{aligned}\label{pdhg-residual-kj-pdhg-2}
    \|u^{k+1}-u^*\|_\M^2-\|u^{k}-u^*\|_\M^2 \leq  - \|g^{k}\|^2_\M.
\end{aligned}
\end{equation}

Let us considered the sequences \( \{g^{p_{l}}\} \) and \( \{g^{a_{l}}\} \), extracted from the full sequence \( \{g^k\} \).
Observing that AA is used infinitely many times in Algorithm~\ref{al-PD-safe}, see Lemma \ref{le-AA-use-inifite}, if we can prove that the subsequences satisfy one of the following two cases:
\begin{align}
&\lim_{l \to \infty} \|g^{p_{l}}\| = 0 = \lim_{l \to \infty} \|g^{a_{l}}\|,\label{eq:lim-gaa-gpd}\\
\text{or\quad}
&\quad\lim_{l \to \infty} \|g^{a_{l}}\|=0, \text{ and the number of elements in $K_{pd}$ is finite,} \label{eq:lim-gaa-gpd-1}
\end{align}
then, we can conclude that $\lim_{k \to \infty} \|g^k\| = 0.$
Next, we prove, indeed, \eqref{eq:lim-gaa-gpd} and \eqref{eq:lim-gaa-gpd-1} hold. 
Using \eqref{eq:ki-aa-2} and \eqref{pdhg-residual-kj-pdhg-2} over all $j \leq k$, we have,
\begin{equation}
\begin{aligned}\label{eq:sum-gpl}
  &\|u^{k+1}-u^*\|^2_\M-\|u^{0}-u^*\|^2_\M\\
  &=\sum_{j \in K_{pd}, \; j \leq k}\Big(\|u^{j+1}-u^*\|_\M^2-\|u^{j}-u^*\|_\M^2 \Big)
   + \sum_{j \in K_{aa}, \; j \leq k} \Big(\|u^{j+1}-u^*\|_\M^2-\|u^{j}-u^*\|_\M^2\Big)\\
    &\leq - \sum_{j \in K_{pd}, \; j \leq k} \|g^{j}\|_\M^2 +M^2D^2{\lambda_{\max}(\M) }\|g^0\|^2\sum_{j \in K_{aa}, \; j \leq k}  (i^{j}+1)^{-2(1+\varepsilon)}
\\
&
\quad+2EM{{  { \sqrt{\lambda_{\max}(\M)}}} }D\|g^0\|\sum_{j \in K_{aa}, \; j \leq k}(i^{j}+1)^{-(1+\varepsilon)}.
\end{aligned}
\end{equation}

From the above, we obtain
\begin{equation*}
\begin{aligned}
&\sum_{j \in K_{pd}, \; j \leq k} \|g^{j}\|_\M^2 \\
&\leq M^2D^2\lambda_{\max}(\M)\|g^0\|^2
      \sum_{j \in K_{aa}, \; j \leq k} (i^{j}+1)^{-2(1+\varepsilon)} \\
&\quad + 2EM{  { \sqrt{\lambda_{\max}(\M)}}}D\|g^0\|
      \sum_{j \in K_{aa}, \; j \leq k}(i^{j}+1)^{-(1+\varepsilon)}
      - \|u^{k+1}-u^*\|_\M^2 + \|u^{0}-u^*\|_\M^2 \\
&\leq M^2D^2\lambda_{\max}(\M)\|g^0\|^2
      \sum_{j \in K_{aa}, \; j \leq k}(i^{j}+1)^{-2(1+\varepsilon)} \\
&\quad + 2EM{  { \sqrt{\lambda_{\max}(\M)}}}D\|g^0\|
      \sum_{j \in K_{aa}, \; j \leq k}(i^{j}+1)^{-(1+\varepsilon)}
      + \|u^{0}-u^*\|_\M^2
\end{aligned}
\end{equation*}

Now letting $k \to \infty$, due to the summability of the series  
$$M^2D^2{{  { {\lambda_{\max}(\M)}}} }\|g^0\|^2\sum_{j=0}^{\infty}(i^{j}+1)^{-2(1+\varepsilon)}+2EM{{  { \sqrt{\lambda_{\max}(\M)}}} }D\|g^0\|\sum_{j=0}^{\infty}(i^{j}+1)^{-(1+\varepsilon)}+\|u^0-u^*\|^2_\M,$$ we obtain that $\sum_{j \in K_{pd}} \|g^{j}\|_\M^2 < + \infty$, i.e., either $K_{pd}$ is finite, or $\lim_{\substack{j \to \infty \\ j \in K_{pd}}}\|g^{j}\|=0$. 

Moreover, from safeguard condition, we have for any $j \in K_{aa}$,
$$0\leq \lim_{\substack{j \to \infty \\ j \in K_{aa}}}\|g^{j}\|\leq D \|g^0\|\lim_{j\to\infty}(j+1)^{-(1+\varepsilon)}=0.$$
Hence, we conclude that $\lim_{k\to\infty}\|g^k\|=0$.

By the proof of Lemma \ref{le-quasi-fejer}, we get that the sequence \( \{ \|u^k - u^*\| \}_k \) is bounded, and hence, we can conclude that  \( \{u^k\} \) is bounded. Then for any convergent subsequence of $\{u
 ^k\}$, i.e., \( u^{k_n} \to  {u^*} \), combined with \( \lim_{n \to \infty} g^{k_n} =  \lim_{n \to \infty} \ttt u^{k_n}-u^{k_n}= 0 \), we know that \(\ttt  {u^*}= \lim_{n \to \infty} \ttt u^{k_n} = \lim_{n \to \infty}  u^{k_n} = {u^*}, \) and hence \(   {u^*} \in Fix \mathcal{T} \).
We can then conclude that the limit point of any convergent subsequence of \( \{u_k\} \) is a fixed-point of \( \ttt \).

Moreover, we have already shown that \( \{u^k\} \) is a \( \mathcal{M} \)-quasi-Fejér monotone sequence with respect to \( Fix\, \mathcal{T} \) in Lemma~\ref{le-quasi-fejer}, i.e., $
\|u^{k+1} - u^*\|_{\mathcal{M}} \leq \|u^k - u^*\|_{\mathcal{M}} + \varepsilon^k,$
and hence, from Lemma~\ref{le:strong-con}, we conclude that \( \{u^k\} \) converges to some \( u^* \in Fix\ttt\).
\end{proof}

\section{Filtered Anderson Accelerated PDHG (FAA-PDHG)}\label{se:FAA-PDHG}
In this section, we will address the issue concerning the uniform boundedness of the matrices $\|H^k\|$. The results here presented are mainly inspired by the developments in \cite{PRFAA-2023}, and extend the techniques there presented to the safeguard framework in Algorithm \ref{al-PD-safe}. To this aim, we will consider the following Algorithm \ref{al-PD-safe_filtered1}, a modification of Algorithm \ref{al-PD-safe},  that enforces algorithmically the uniform boundedness of the matrices $H^k$ using the procedure described in Algorithms~\ref{al:angle-filtering1} (\textit{AngleFilter}) and~\ref{al:length-filtering1} (\textit{LengthFilter}), see Lines~\ref{algline:a_filter1} and~\ref{algline:l_filter1} Algorithm \ref{al-PD-safe_filtered1}.  {We explicitly note, more in detail, that the main differences w.r.t. Algorithm \ref{al-PD-safe}, are represented by Lines~\ref{algline:reverse1} - \ref{algline:filtered_update}.}  The auxiliary function \textit{reverse} simply swaps the columns of the input matrices, as illustrated in Algorithm~\ref{al:reverse}. Its use is motivated by the choice to leverage existing results from \cite{PRFAA-2023}, where the most recent information is stored to the left of the matrices $\Delta\mathcal{U}^{k-m_k}$ and $\Delta\mathcal{G}^{k-m_k}$, in contrast to our convention, where the most updated columns are placed at the right.   {In this regard, it is important to note, that the reverse function in Algorithm \ref{al:reverse}  is a notational convenience that aligns our indexing convention -- where the most recent columns appear on the right -- with that of \cite{PRFAA-2023}, where they appear on the left.  This alignment allows us to invoke   Lemma~\ref{eq:lemma1.51}  and \ref{le:RK-1bound1} directly without reformulation.}

 {We remark, moreover, that Algorithm~\ref{al-PD-safe_filtered1} inherits the global convergence guarantees established in Section~\ref{se:converge}. Indeed, the convergence analysis of Theorem~\ref{th:con} relies solely on the uniform boundedness assumption $\|H^k\| \leq M$ and does not depend on the particular form of the matrix $H^k$. Since the angle and length filtering procedures enforce this bound algorithmically (see Theorem~\ref{th:FAA-convergence}), the convergence of Algorithm~\ref{al-PD-safe_filtered1} follows directly. Furthermore, we note that Algorithm~\ref{al-PD-safe_filtered1} does not employ the Tikhonov regularisation parameter $\eta$ present in Algorithm~\ref{al-PD-safe}. This is because the \textit{AngleFilter} and \textit{LengthFilter} procedures guarantee the well-conditioning of the QR factorisation of $\Delta\mathcal{G}^{k-m_k}$, thereby ensuring the invertibility of $R_k$ and the well-definedness of the Anderson update without the need for explicit regularisation.}

\begin{algorithm}
\caption{Filtered Anderson Accelerated PDHG (FAA-PDHG) }\label{al-PD-safe_filtered1}
\vspace{-0.1cm}
\begin{algorithmic}[1] 
\State \textbf{Given} initial feasible point $u^0 := (x^0, y^0)$, a projection operator $\mathcal{P} : \mathbb{R}^n \times \mathbb{R}^m \to X \times Y$, and the fixed-point iteration $\mathcal{T}$ defined in \eqref{new-fixed-point}. Select a scalar $D>0$ and a diagonal matrix $\widehat{D}$. Let $m_A \geq 1$, $\varepsilon>0$, $tol>0$. Set $i^1 = 0$, $j^1 = 0$.  Set  $u^1=\mathcal{T} u^{0}$, $g^{0}:=u^{1}  - u^{0}$. 
\For{$k=1,\dots$}
    \State Set $\widehat{u}^{k+1} = \mathcal{T} u^{k} $
    \State Set $g^{k}=  \widehat{u}^{k+1} -  u^{k}$ \label{algline-fixedpointres}

    \If{$\|g^{k}\|<tol$} 
        \State \textbf{break}
    \EndIf
    
    \State Set $m_k=\min \{k, m_A\}$

\State For $j = 1, \dots, m_k$ compute \label{algline:AA_mem_filter}
\[
\begin{array}{ll}
\Delta u^{k-j} := u^{k-j+1}-u^{k-j}, &\quad \Delta g^{k-j} = g^{k-j+1}-g^{k-j}\\
\Delta \mathcal{U}^{k-m_k}=[\Delta u^{k-m_k},\ldots,\Delta u^{k-1}], &\quad \Delta \mathcal{G}^{k-m_k} =[\Delta g^{k-m_k},\ldots,\Delta g^{k-1}]
\end{array}
\]
    \If{$\|g^{k}\| \leq D \|g^0\| (i^{k}+1)^{-(1+\varepsilon)}$} \label{algline:safeguard_filter}
        \State $(\Delta\mathcal{U}^{k-m_k},\Delta\mathcal{G}^{k-m_k})=Reverse(\Delta\mathcal{U}^{k-m_k},\Delta\mathcal{G}^{k-m_k})$ \label{algline:reverse1}
        \State Angle Filter Step: \label{algline:a_filter1}\;
$\begin{aligned}
(\Delta\mathcal{U}^{k-m_k},\Delta\mathcal{G}^{k-m_k}) = {} &
\textit{AngleFilter}(\Delta\mathcal{U}^{k-m_k},\Delta\mathcal{G}^{k-m_k},m_k,c_s)
\end{aligned}$
        \State Length Filter Step: \label{algline:l_filter1}\;
        $\begin{aligned}
        (\Delta\mathcal{U}^{k-m_k},\Delta\mathcal{G}^{k-m_k})=LengthFilter(\Delta\mathcal{U}^{k-m_k},\Delta\mathcal{G}^{k-m_k},m_k,c_s,\bar{\kappa})
        \end{aligned}$
        \State $(\Delta\mathcal{U}^{k-m_k},\Delta\mathcal{G}^{k-m_k})=Reverse(\Delta\mathcal{U}^{k-m_k},\Delta\mathcal{G}^{k-m_k})$ \label{algline:reverse2}
        \State Compute $\Delta \mathcal{G}^{k-m_k}=Q_kR_k$ \label{algline:filtered_QR}
    \State Compute:  \label{algline:filtered_update}
\begin{equation}\label{eq:AA-update_filter1}
\begin{aligned}
&u_{AA}^{k+1} := u^{k} - H^{k} g^{k},\\  
&\text{where } H^{k} = -\beta \widehat{D} + (\Delta \mathcal{U}^{k-m_k} + \beta \widehat{D} \Delta \mathcal{G}^{k-m_k}) (R_k)^{-1}(Q_k)^{\top}.
\end{aligned}
\end{equation}     
\State Set $\begin{aligned}[t]
u^{k+1} &= \mathcal{P}u_{AA}^{k+1}, \\
i^{k+1} &= i^k + 1.
\end{aligned}$
\Else
\State Set $\begin{aligned}[t]
u^{k+1} &= \widehat{u}^{k+1}, \\
j^{k+1} &= j^k + 1.
\end{aligned}$
\EndIf
\EndFor
\end{algorithmic}
\vspace{-0.1cm}
\end{algorithm}

\begin{algorithm}[t]
\caption{Angle Filtering $(E,\;F,\; {m},\;c_s)$}\label{al:angle-filtering1}
\begin{algorithmic}[1]
\State \textbf{Require} $E,\;F$: current matrices, $ {m}$: number of columns, $c_s$: threshold
\State Compute the economy QR decomposition $F = QR$
\For{$i = 2$ , \ldots, $ {m}$}
    \State Compute $\sigma_i = |r_{ii}| / \|f_i\|_2$, where $f_i$ is column $i$ of $F$, and $r_{ii}$ is the corresponding diagonal entry of $R$
    \If{$\sigma_i < c_s$}
        \State Remove column $i$ from $E$ and $F$
    \EndIf
\EndFor
\State \textbf{Return} $(E,\;F)$
\end{algorithmic}
\end{algorithm}

\begin{algorithm}[t]
\caption{Length Filtering $(E,\;F,\; {m},\;c_s, \bar{\kappa})$}\label{al:length-filtering1}
\begin{algorithmic}[1]
\State \textbf{Require} $E,\;F$: current matrices,  {$m$}: current number of columns, $c_s$: threshold, $\bar{\kappa}$: tolerance number
\State $c_t \gets \sqrt{1 - c_s^2}$; for $j = 1, \ldots,  {m}$, compute column norms $e_j$ of $E$ and upper bounds $b_j$ for the squared norm of the $j$-th column of $R^{-1}$.
\For{$k =  {m}$ , \ldots, $1$}
    \State $C \gets \left( \sum_{j=1}^{k} \|e_j\|^2 \right) \left( \sum_{j=1}^{k} b_j \right)$
    \If{$C \leq \bar{\kappa}^2$}
        \State \textbf{break}
    \EndIf
\EndFor
\State $\bar{m} \gets k$
\State $E \gets$ first $\bar{m}$ columns of $E$
\State $F \gets$ first $\bar{m}$ columns of $F$
\State \textbf{Return} $(E,F)$
\end{algorithmic}
\end{algorithm}

\begin{algorithm}[t]
\caption{Reverse $(E,\;F)$}\label{al:reverse}
\begin{algorithmic}[1]
\State \textbf{Require} $E,\;F$ with the same number of columns
\State $ {m} \gets$ number of columns of $E$
\State  {$\widehat{E} \gets E$, $\widehat{F} \gets F$}
\For{$k = 1$ , \ldots, $ {m}$}
  \State   {$e_k \gets \hat{e}_{m-k+1}$; \quad $f_k \gets \hat{f}_{m-k+1}$}
\EndFor
\State \textbf{Return} $(E,F)$
\end{algorithmic}
\end{algorithm}

 {Having presented the \textit{AngleFilter} (Algorithm~\ref{al:angle-filtering1}) and \textit{LengthFilter} (Algorithm~\ref{al:length-filtering1}) procedures, we present the  theoretical results supporting them:}

\begin{lemma}\cite[Lemma 1.1]{PRFAA-2023}\label{eq:lemma1.51}
    Let $F\in \mathbb{R}^{n\times m}$ and let $F = QR$ be the economy QR decomposition. For $p=1,\dots, m$, let $F_p = \operatorname{span}\{f_1,\dots,f_p\}$, the subspace spanned by the first $p$ columns of $F$. Suppose there is a constant $0 < c_s \leq 1$ such that
\[
\sin(f_i, F_{i-1}) \geq c_s, \quad \text{for } i = 2,\dots,m
\]
and, defining $c_t = \sqrt{1 - c_s^2}$, suppose 
\(
\cos(f_i, f_k) \leq c_t, \;\text{for } k = 1, \ldots, i.
\)
Denote $R^{-1} = (s_{ij})$. Then it holds that
\begin{equation*}
\begin{aligned}
    &(s_{11}) = \frac{1}{\|f_1\|},  & |(s_{1j})| \leq \frac{c_t (c_t + c_s)^{j-2}}{\|f_1\| \, c_s^{j-1}}, \; 2 \leq j \leq m, \label{eq:lemma1.4} \\
    &(s_{ii}) \leq \frac{1}{\|f_i\| c_s}, \; 2 \leq i \leq m,  & |(s_{ij})| \leq \frac{c_t (c_t + c_s)^{j-i-1}}{\|f_i\| \, c_s^{j-i+1}}, \; i+1 \leq j \leq m.
\end{aligned}
\end{equation*}
\end{lemma}
\begin{lemma}\cite[Proposition 1.2]{PRFAA-2023}\label{le:RK-1bound1}
Suppose the hypotheses of Lemma~\ref{eq:lemma1.51} hold. Let $(s_i)$ denote column $i$ of $R^{-1}$. The following bounds hold:
\begin{equation*}
\begin{aligned}
   & \|(s_1)\|^2 = \frac{1}{\|f_1\|^2} =: b_1, \quad \quad \|(s_2)\|^2 \leq \frac{1}{c_s^2} \left( \frac{c_t^2}{\|f_1\|^2} + \frac{1}{\|f_2\|^2} \right) =: b_2, \\
   & \|(s_j)\|^2 \leq \frac{1}{c_s^2} \left[
        \frac{c_t^2 (c_t + c_s)^{2(j-2)}}{\|f_1\|^2 \, c_s^{2(j-2)}} 
        + \sum_{i=2}^{j-1} \frac{c_t^2 (c_t + c_s)^{2(j-i-1)}}{\|f_i\|^2 \, c_s^{2(j-i)}}  + \frac{1}{\|f_j\|^2}
    \right] =: b_j,  \;  3 \leq j \leq  {m}. 
\end{aligned}
\end{equation*}
\end{lemma}

In particular, the \textit{AngleFilter} procedure used at Line \ref{algline:a_filter1} of Algorithm~\ref{al-PD-safe_filtered1} ensures that the upper bounds used in \textit{LengthFilter} at Line \ref{algline:l_filter1} hold. The next lemma clarifies the properties of the output of Algorithm~\ref{al-PD-safe_filtered1}.

\begin{lemma}\label{lem:conditioning_after_stab}
   Let us assume $(E,F,m)$ to be the input of Algorithm~\ref{al:length-filtering1} with $F$ satisfying hypotheses of Lemma~\ref{eq:lemma1.51} and consider the corresponding output $(\bar{E},\bar{F},\bar{m})$, with $\bar{F}=QR$.  Then it holds  
    \begin{equation}
     \|\bar{E}\|^2\|R^{-1}\|^2 \leq \|\bar{E}\|_F^2\|R^{-1}\|_F^2=  (\sum_{i=1}^m \|\bar{e}_i\|^2)(\sum_{i=1}^m \|s_i\|^2) \leq    {\Big(}\sum_{i=1}^m \|\bar{e}_i\|^2  {\Big)\Big(}\sum_{i=1}^m b_i  {\Big)} \leq  \bar{\kappa}^2
    \end{equation}
    where $b_i$ are the upper bounds in Lemma {\ref{le:RK-1bound1}}.
\end{lemma}

\begin{theorem}\label{th:FAA-convergence}
 {Let the sequence $\{u^k\}$ be generated by Algorithm~\ref{al-PD-safe_filtered1} with angle filtering parameter $c_s \in (0,1]$ and length filtering parameter $\bar{\kappa} > 0$. Then there exists a constant
\[
M = \beta\|\widehat{D}\| + \bar{\kappa} + \dfrac{2c_2}{c_1}\beta\|\widehat{D}\|\bar{\kappa},
\]
independent of $k$ such that $\|H^k\| \leq M$ for all $k$. Consequently, by Theorem~\ref{th:con}, the sequence $\{u^k\}$ converges globally to a fixed point $u^* \in \operatorname{Fix}\ttt$.}
\end{theorem}
\begin{proof}{Proof}
Before starting the proof, we observe that in Algorithm~\ref{al-PD-safe_filtered1}, the computation of the QR decomposition at Line 15, is motivated by observing that if $\Delta\mathcal{G}^{k-m_k}=Q_kR_k$, we have that
  \begin{equation*}
    ((\Delta\mathcal{G}^{k-m_k})^T\Delta\mathcal{G}^{k-m_k})^{-1}(\Delta\mathcal{G}^{k-m_k})^T= R_k^{-1} Q_k^T,
\end{equation*}  
and hence the particular Quasi-Newton  updated in \eqref{eq:AA-update_filter1}.
We have, moreover,
\begin{equation}
\begin{aligned}\label{ineq:Hk-bound1}
        \|H^{k}\|
       & \leq \beta\|\widehat{D}\|
    + \|\Delta \mathcal{U}^{k-m_k}\|\|R_k^{-1}\| + \beta \|\widehat{D} \|\|  \Delta \mathcal{G}^{k-m_k}\| \|R_k^{-1}\|.
    \end{aligned}
\end{equation}
{
Due to the equivalence of norms and the $\M$-firm non expansiveness of $\ttt$, there exists $c_1,~\;c_2~>~0$ such that
\begin{equation}
\begin{aligned}\label{ineq:deltagk}
        c_1\|\Delta g^{k-j-1}\|
&\leq\|g^{k-j}-g^{k-j-1}\|_{\M} =\|\ttt u^{k-j}-u^{k-j}-\ttt u^{k-j-1}+u^{k-j-1}\|_{\M}\\
&\leq 2\|u^{k-j}-u^{k-j-1}\|_{\M}\leq2c_2\|\Delta u^{k-j-1}\|,
\end{aligned}
\end{equation}
i.e., $\|\Delta \mathcal{G}^{k-m_k}\|_F \leq \dfrac{2c_2}{c_1} \|\Delta \mathcal{U}^{k-m_k}\|_F.$
Thus the claim follows by the inequalities
\begin{equation*}
\begin{aligned}
        \|H^{k}\| & \leq \beta\|\widehat{D}\|
    + \|\Delta \mathcal{U}^{k-m_k}\|\|R_k^{-1}\| + \beta \|\widehat{D} \|\|  \Delta \mathcal{G}^{k-m_k}\| \|R_k^{-1}\|  \\
    & \leq \beta\|\widehat{D}\|
    + \|\Delta \mathcal{U}^{k-m_k}\|_F\|R_k^{-1}\|_F + \beta \|\widehat{D} \|\|  \Delta \mathcal{G}^{k-m_k}\|_F \|R_k^{-1}\|_F \\
    & \leq \beta\|\widehat{D}\|
    +  {\bar{\kappa}} + \dfrac{2c_2}{c_1} \beta \|\widehat{D} \|   {\bar{\kappa}} =:M,
\end{aligned}
\end{equation*}
where the last inequality follows from Lemma~\ref{lem:conditioning_after_stab}.
}
\end{proof}

\section{Numerical  {Experiments}}\label{se:NE}
 {This section presents a comprehensive numerical evaluation of the proposed algorithms. We organise the experiments into two parts. In Section~\ref{se:vanilla_comparison}, we compare AA-PDHG (Algorithm \ref{al-PD-safe}), FAA-PDHG (Algorithm \ref{al-PD-safe_filtered1}), and vanilla PDHG on selected LP instances to demonstrate the effectiveness of Anderson Acceleration. In Section~\ref{se:rpdhg_comparison}, we address the central question of this work: whether AA can serve as a competitive alternative to the restart strategy proposed by \cite{fasterPDHG-23}. To this end, we compare AA-PDHG against rPDHG on the full MIPLIB dataset of LP relaxations.} All  {experiments are} carried out in \textit{Julia 1.6.5}  {on the} \textit{University of Southampton's Iridis 6 high-performance computing cluster} {. The} code is publicly available at \url{https://github.com/ZYXYJS/AA-PD}.

Before presenting the numerical results, we note that the structure of the AA update allows the associated matrix $\Delta\mathcal{G}^{k-m_k}$ to be updated incrementally when columns are added or removed \cite{OWAAre23}. Instead of recomputing the entire least-squares system at each iteration, we update the corresponding inner products, which significantly reduces the computational cost.  {We emphasize that the per-iteration overhead of the AA update is modest: the AA least-squares subproblem has one of the two dimensions equal   {at most} to $m_A\times (n+m)$ and its solution cost is negligible when $m_A$ is kept small. The additional memory requirement is $O(m_A(n+m))$ for storing the matrices $\Delta\mathcal{U}^{k-m_k}$ and $\Delta\mathcal{G}^{k-m_k}$, which is also small relative to the storage of the constraint matrix $K$.} Besides, the implemented version of Algorithm~\ref{al-PD-safe_filtered1}, employs a slightly different update strategy for the matrices $\Delta \mathcal{U}^{k-m_k}$ and $\Delta \mathcal{G}^{k-m_k}$ compared to the description in Section~\ref{se:FAA-PDHG}. Specifically, rather than fully recomputing these matrices as in Line~\ref{algline:AA_mem_filter}, the implementation appends the most recent vectors $\Delta u^{k-1}$ and $\Delta g^{k-1}$ to the matrices obtained after applying the filtering procedures, as returned in Lines~\ref{algline:a_filter1} and~\ref{algline:l_filter1}. To ensure that the number of columns does not exceed the prescribed memory size $m_A$, the matrices produced by the filtering procedures may discard the oldest information. This modification can be regarded as an adaptive memory selection mechanism, prioritising the most recent and relevant information gathered through the fixed-point iteration.

 {\begin{remark}\label{se:Dhat_construction}
The diagonal matrix $\widehat{D}$ in Algorithms~\ref{al-PD-safe} and~\ref{al-PD-safe_filtered1} is constructed from the constraint matrix $K$ as follows. Writing $\widehat{D} = \operatorname{diag}(\widehat{D}_1, \widehat{D}_2)$, where $\widehat{D}_1 \in \mathbb{R}^{n \times n}$ corresponds to the dual variables and $\widehat{D}_2 \in \mathbb{R}^{m \times m}$ corresponds to the primal variables, each diagonal entry is defined as
\[
(\widehat{D}_1)_{jj} = \sqrt{\max\Big\{\max_{i} |K_{ij}|,\; 10^{-10}\Big\}}, \quad j = 1, \dots, n.
\]
\[
(\widehat{D}_2)_{ii} = \sqrt{\max\Big\{\max_{j} |K_{ij}|,\; 10^{-10}\Big\}}, \quad i = 1, \dots, m,
\]
This construction captures the row and column scaling of the constraint matrix, thereby providing a diagonal preconditioning effect analogous to that discussed in \cite{diagonalPDHG-11}. The threshold $10^{-10}$ prevents division by zero for rows or columns of $K$ that are identically zero.
\end{remark}}

\begin{remark}(Termination conditions)\label{re:constrcution-termi-con} At Line~\ref{algline-fixedpointres} of Algorithms~\ref{al-PD-safe} and \ref{al-PD-safe_filtered1},  {instead of the termination conditions based on the fixed-point point residuals, we use} a stopping criterion based on the KKT residual as proposed in \cite{google2022p}. In particular, for $x\in X,\;y\in Y, \;\lambda \in\Lambda,$ such  {stopping } conditions are given as follows:
\begin{align}
| q^T y + l^T \lambda^+ - u^T \lambda^- - c^T x |&\leq \epsilon (1 + | q^T y + l^T \lambda^+ - u^T \lambda^- | + | c^T x |), \label{duality-gap} \\
\left\| \begin{pmatrix} b -Ax \\ (h - Gx)_+ \end{pmatrix} \right\|_2
&\leq \epsilon (1 + \| q \|_2),\label{primal-feasi} \\
\| c - K^T y - \lambda \|_2 &\leq \epsilon (1 + \| c \|_2) \label{dual-feasi}.
\end{align}
In general, equation \eqref{duality-gap} measures the relative duality gap $r_g$, \eqref{primal-feasi} the relative primal infeasibility $r_p$,  whereas \eqref{dual-feasi} is the relative dual infeasibility $r_d$, where $\epsilon$ controls the quality of the sought approximate solution.  {In the following experiments, the term  \textit{KKT residual} denotes the quantity}
\begin{align*}
    KKT= \max\{r_g,\; r_p,\; r_d\}.
\end{align*}
It is important to note that, since the PDHG algorithm does not explicitly include reduced-cost variable \( \lambda \), to evaluate  \eqref{dual-feasi},  $\lambda = \p_{\Lambda}(c - K^T y)$ is used. This is because $c-K^{\top}y=\lambda$ with $\lambda \in \Lambda$ represent the feasibility conditions of the dual formulation \eqref{DUAL-2}.
 {Throughout the experiments, unless otherwise stated, we use a target accuracy of $\epsilon = 10^{-4}$ for the KKT residual, which is the same tolerance used in comparable studies of first-order methods for LP \cite{fasterPDHG-23}. All algorithms are terminated when the KKT residual falls below this threshold or a maximum time limit of one hour is reached. The same termination criteria are applied uniformly to all methods in each comparison.}
\end{remark}

 {
\begin{remark}(Primal weight update)  {We note that in \cite{fasterPDHG-23}, rPDHG does not consider an additional primal weight update. Nevertheless, since primal weight updates are considered in \cite{google2022p}, we also conduct the experiments in which both methods are combined with their respective primal weight update strategies. In particular,} instead of using a fixed stepsize throughout, we consider an adaptive primal weight update in the experiments of Section~\ref{se:rpdhg_comparison}. 
The update procedure is similar to that in \cite{google2022p}, but differs in several aspects. Specifically, in Algorithm \ref{alg:smooting}, we present the details of the primal weighting strategy employed in the following. It is important to note that we propose the use of a periodic primal weight update scheme rather than updating it at every iteration. Frequent updates are indeed observed to destroy the stability of Anderson acceleration. Therefore, the update is performed only at fixed intervals of length {$U_P$}, and the accumulated AA history is reset after each update.
\begin{algorithm}[H] 
\caption{Primal weight update}\label{alg:smooting}
\begin{algorithmic}[1]
\State \textbf{Input:}  $x^{k},\;y^k,\; \omega^{k-1}$, smoothing parameter $\theta$,  $\epsilon_{\mathrm{zero}}>0$, $U_P \in \mathbb{N}$.
\If {$k>0$ and $k\equiv 0\; \text{(mod } U_P \text{)}$}
    \State \text{Compute: }  {$\Delta_x = \|x^{k} - x^{k-U_p}\|$,\; $\Delta_y = \|y^{k} - y^{k-U_p}\|$}
    \If{$\Delta_x > \epsilon_{\mathrm{zero}}$ \textbf{and} $\Delta_y > \epsilon_{\mathrm{zero}}$}
    \State\label{alstep:resetAA}  {Reset Anderson Memory (i.e., $\Delta \mathcal{U}^{k-m_k},\; \Delta \mathcal{G}^{k-m_k}$)}
        \State \Return $\omega^k:=\exp\left( \theta \log\left(\frac{\Delta_y}{\Delta_x}\right) + (1-\theta)\log(\omega^{k-1}) \right)$
    \Else
        \State \Return $\omega^{k-1}$
    \EndIf
\EndIf
\end{algorithmic}
\end{algorithm}
\end{remark}
}

 {
\begin{remark}\label{re:pwform}
After having computed the primal weight update weight $\omega$, in \eqref{pdlp-x-update} and \eqref{pdlp-y-update}, we choose
$$
\tau = \frac{\gamma}{\omega}, \quad \sigma = \gamma \omega,$$
where $\gamma$ denotes the step size.
Since this primal weight update can be viewed as a rescaling of the primal and dual step sizes, and the update of $\widehat{D}$ in Algorithm~\ref{al-PD-safe} also achieves such a rescaling via a diagonal matrix, we accordingly update $\widehat{D}$ when applying the primal weight update so that it remains consistent with the corresponding changes in the primal and dual step sizes. In particular, we update $\widehat{D}$ by using:
$$ \widehat D
    =
    (1-\alpha)D_1 + \alpha R,
$$
where $D_1=\operatorname{diag}(\dfrac{\textbf{1}_{n}}{\omega}, \omega\textbf{1}_{m}),\;R=\operatorname{diag}(\dfrac{\widehat{D}_1}{\omega}, \widehat{D}_2\omega)$, see Remark \ref{se:Dhat_construction} for definitions of $\widehat{D}_1$ and $\widehat{D}_2$.
In addition, the choice of $\alpha$ is adapted according to the relative
decrease of the fixed-point residual. Let
\[
    r_k = \|g_k\|,
    \qquad
    \rho_k =
    \frac{r_{k-1} - r_k}
         {\max(r_{k-1},\varepsilon)}.
\]
If the residual decreases sufficiently, we select $\alpha = 0.1$, otherwise, we use $\alpha = 0.9.$
The corresponding experiment results with this update are provided in Supplementary Material \ref{appen:hyper}.
\end{remark}
}

\subsection{Comparison with Vanilla PDHG}\label{se:vanilla_comparison}
 {We first assess the effectiveness of Anderson Acceleration by comparing AA-PDHG, FAA-PDHG, and vanilla PDHG on selected LP instances from the MIPLIB 2017 dataset \cite{MILP}.}

In this part, to preserve the original problem structure and ensure a fair comparison, we do not apply any pre-solve or preconditioning for the selected instances. Moreover, 
 For all algorithms, we apply the same step size, which is estimated using a power iteration, see \cite{power-iter-96} {.}

Concerning the  {safeguard parameter} $D$  {at} Line~\ref{algline:safeguard} of Algorithm~\ref{al-PD-safe}, we observed that  {its} choice has a substantial impact on the overall performance of the algorithm.  {Since}
$D$ controls how often Anderson Acceleration is accepted and applied throughout the iterations {,} a larger value of $D$ allows more AA steps, thereby increasing the frequency of acceleration.  {Figure \ref{fig:AAdifferent-D} illustrates the KKT residual history for a representative LP instance under different choices of $D$, showing that both excessively large and excessively small values lead to unstable residual decay.}  {This observation guided our choice of the search range for $D$ in the subsequent experiments. The regularisation parameter $\eta$ is fixed at $10^{-10}$ throughout, as this value proved sufficient across all test instances considered in this section.}

 {Regarding the memory size $m_A$, the existing literature \cite{AA-CON-TAK,AA-FPI-NI,AA-IMPROVE-PROOF-ECPRX} suggests that values in the range $5$ to $20$ yield the best performance. In general, a larger $m_A$ reduces the number of iterations at the cost of increased per-iteration computational overhead. We report results for $m_A=5$ and $m_A=10$ in Figure~\ref{fi:fixresidual-iteration} with tolerance $\texttt{tol=1e-4}$. We do not present an analogous comparison for FAA-PDHG, since its adaptive memory update makes the effect of $m_A$ on iteration count and runtime less directly interpretable.}

\begin{figure}[!t]
\centering
\includegraphics[width=\linewidth]{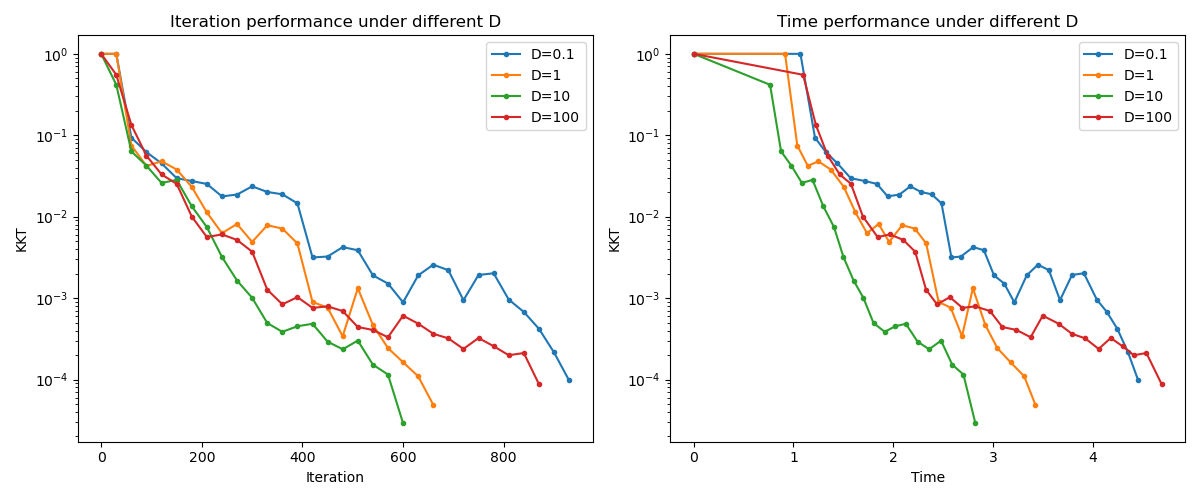}
\caption{ {KKT residual history of AA-PDHG with $m_A = 5$ for different values of the safeguard parameter $D$. Excessively large or small values of $D$ lead to unstable residual decay, as Anderson Acceleration is activated either too infrequently or too aggressively.}
}\label{fig:AAdifferent-D}
\end{figure}

 {Figure~\ref{fi:fixresidual-iteration} compares the performance of vanilla PDHG, FAA-PDHG ($m_A=5$), and AA-PDHG ($m_A=5,10$) in terms of both iteration count and computational time. For FAA-PDHG, we set $c_s=0.2$ and $\bar{\kappa}={10^8}$. To ensure a fair assessment, we fix $D=10$ and $\varepsilon=1$ for each instance.}

 {We observe from the left panels of Figure~\ref{fi:fixresidual-iteration} that, in the initial iterations, AA-PDHG and FAA-PDHG exhibit the same fluctuating behaviour as vanilla PDHG. However, once sufficient historical information is accumulated, both accelerated methods reduce the KKT residual at a substantially faster rate, in terms of both iterations and wall-clock time.
Moreover, setting $m_A = 10$ yields better iteration performance than $m_A = 5$, consistent with the general behaviour of limited-memory quasi-Newton methods.}

\begin{figure}[!t]
\centering
    \includegraphics[width=0.4\linewidth]{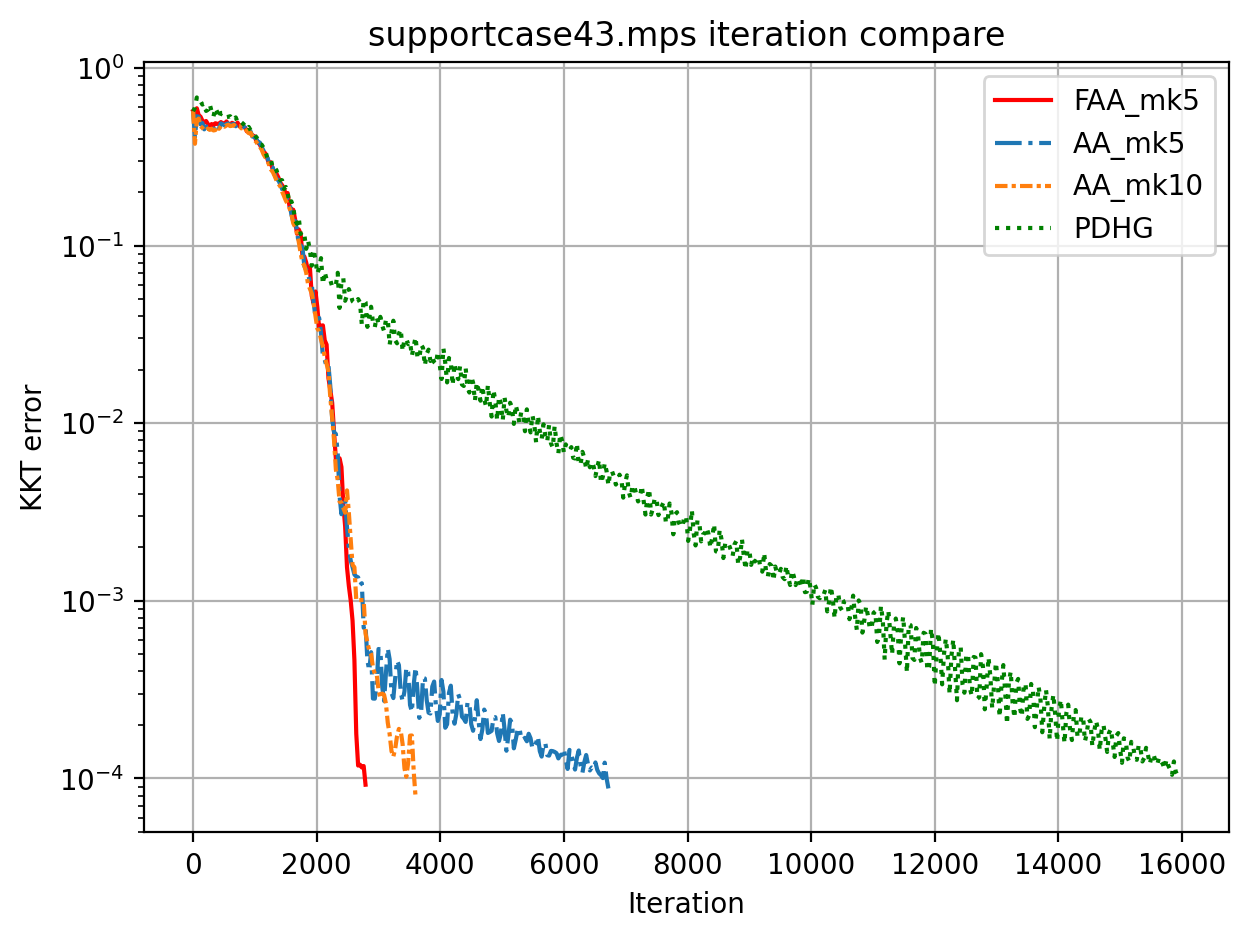}
    \includegraphics[width=0.4\linewidth]{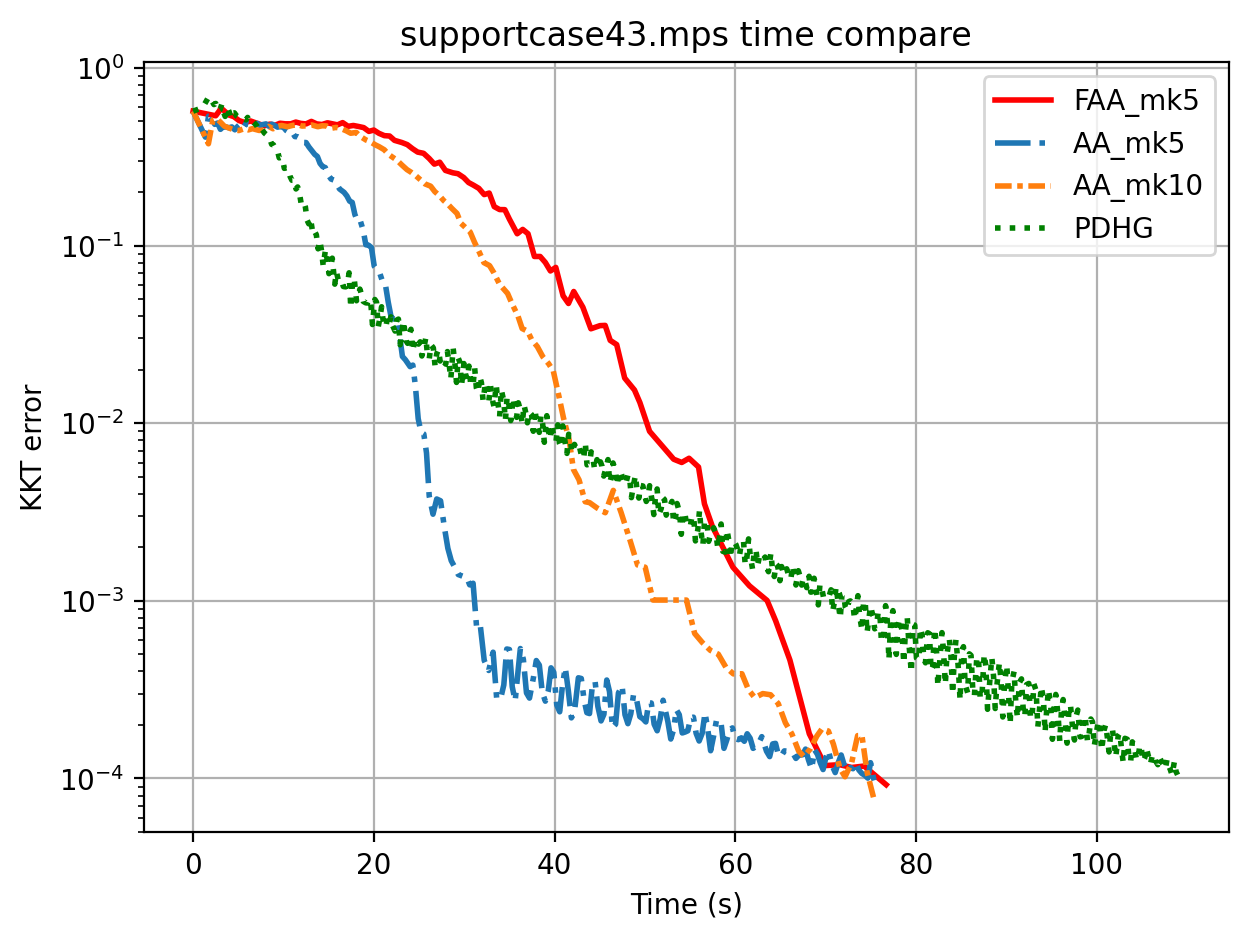}
    \includegraphics[width=0.4\linewidth]{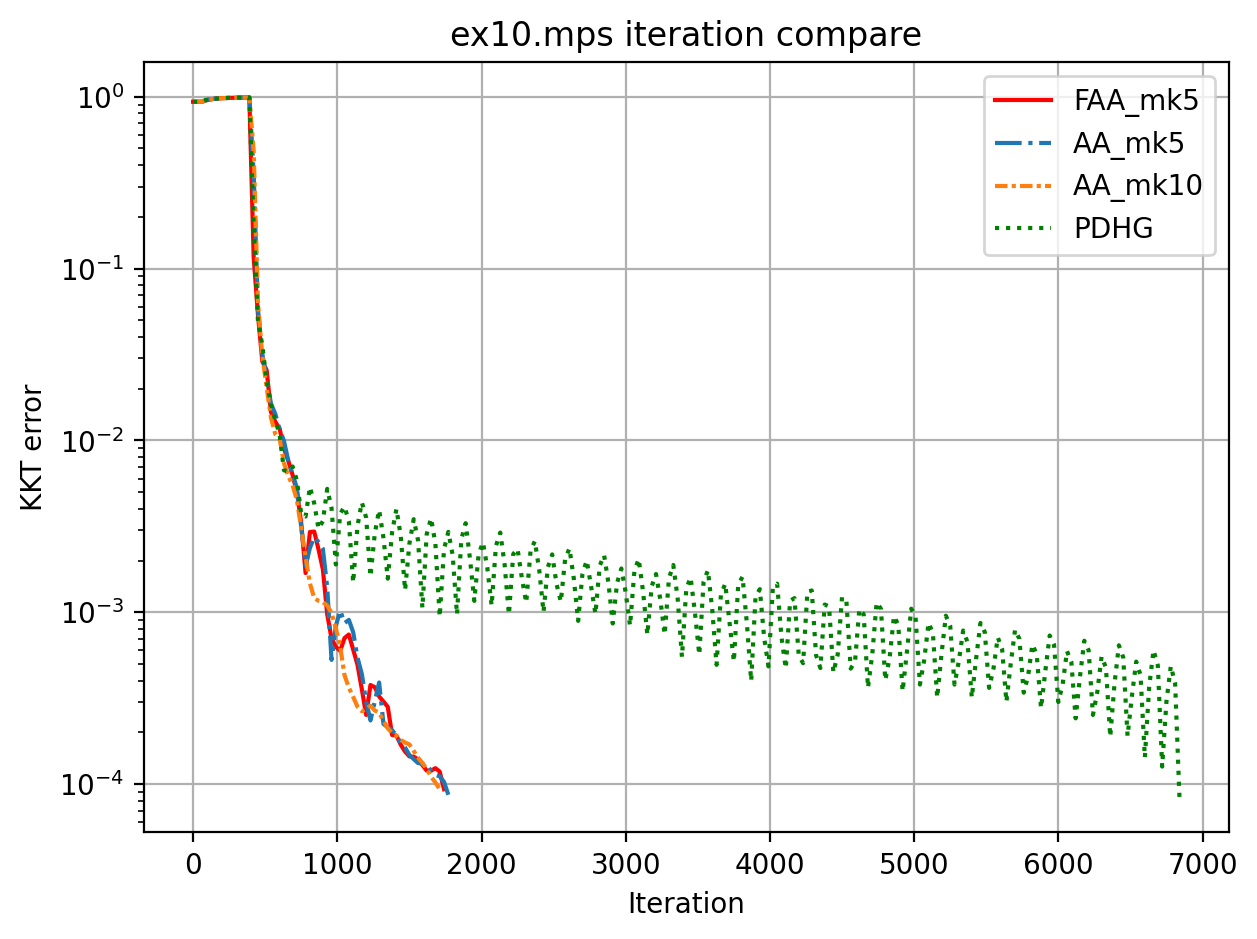}
    \includegraphics[width=0.4\linewidth]{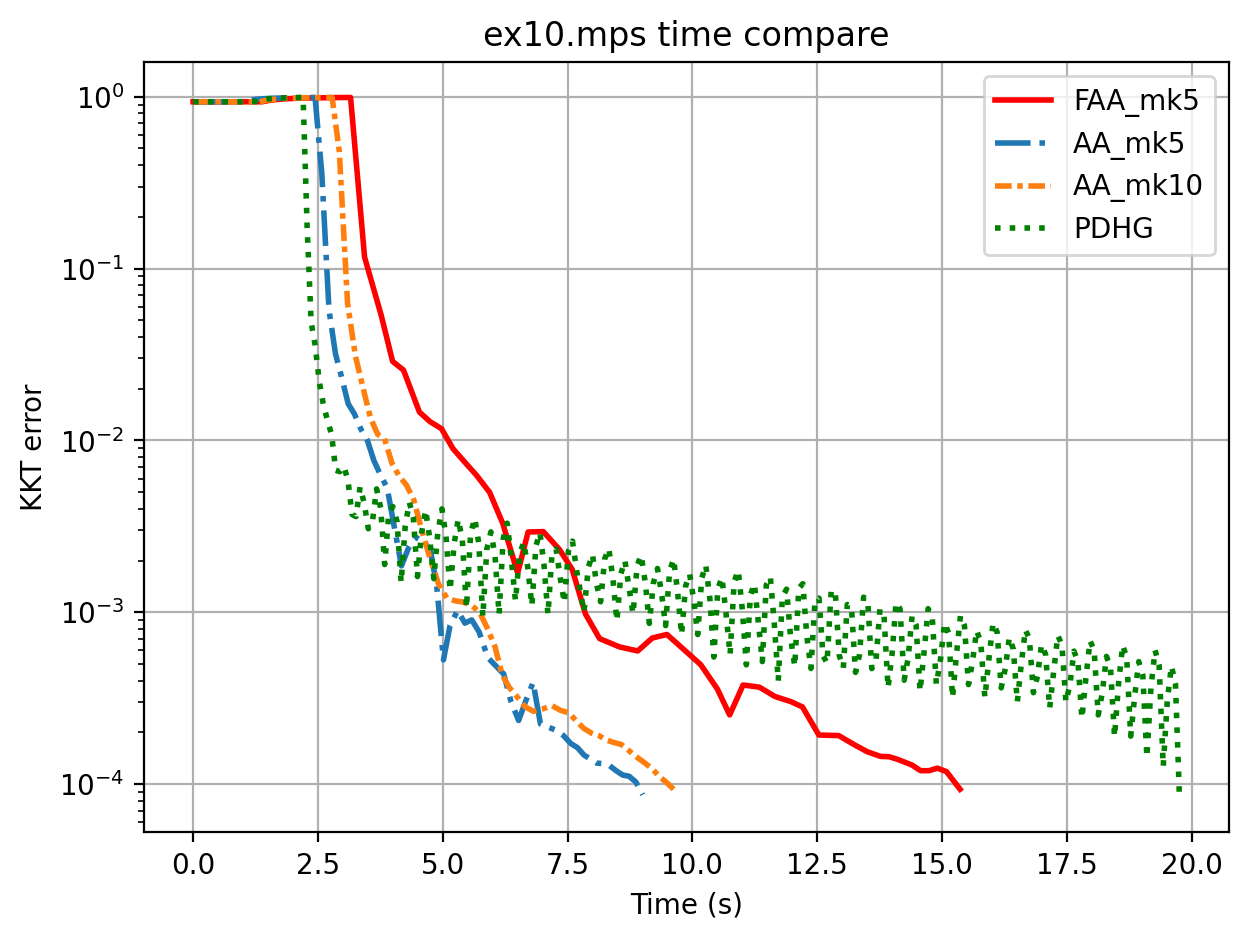}
    \includegraphics[width=0.4\linewidth]{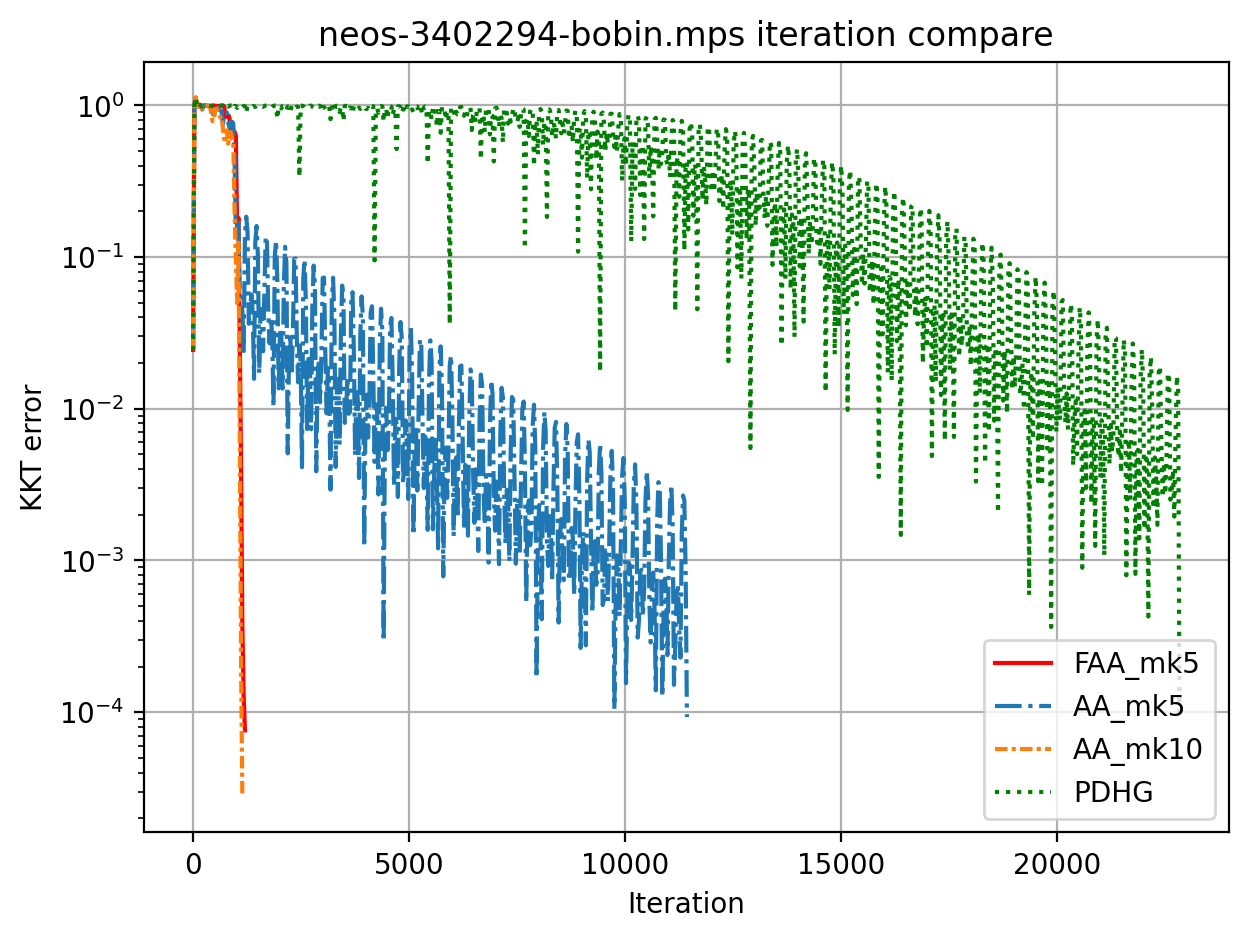}
    \includegraphics[width=0.4\linewidth]{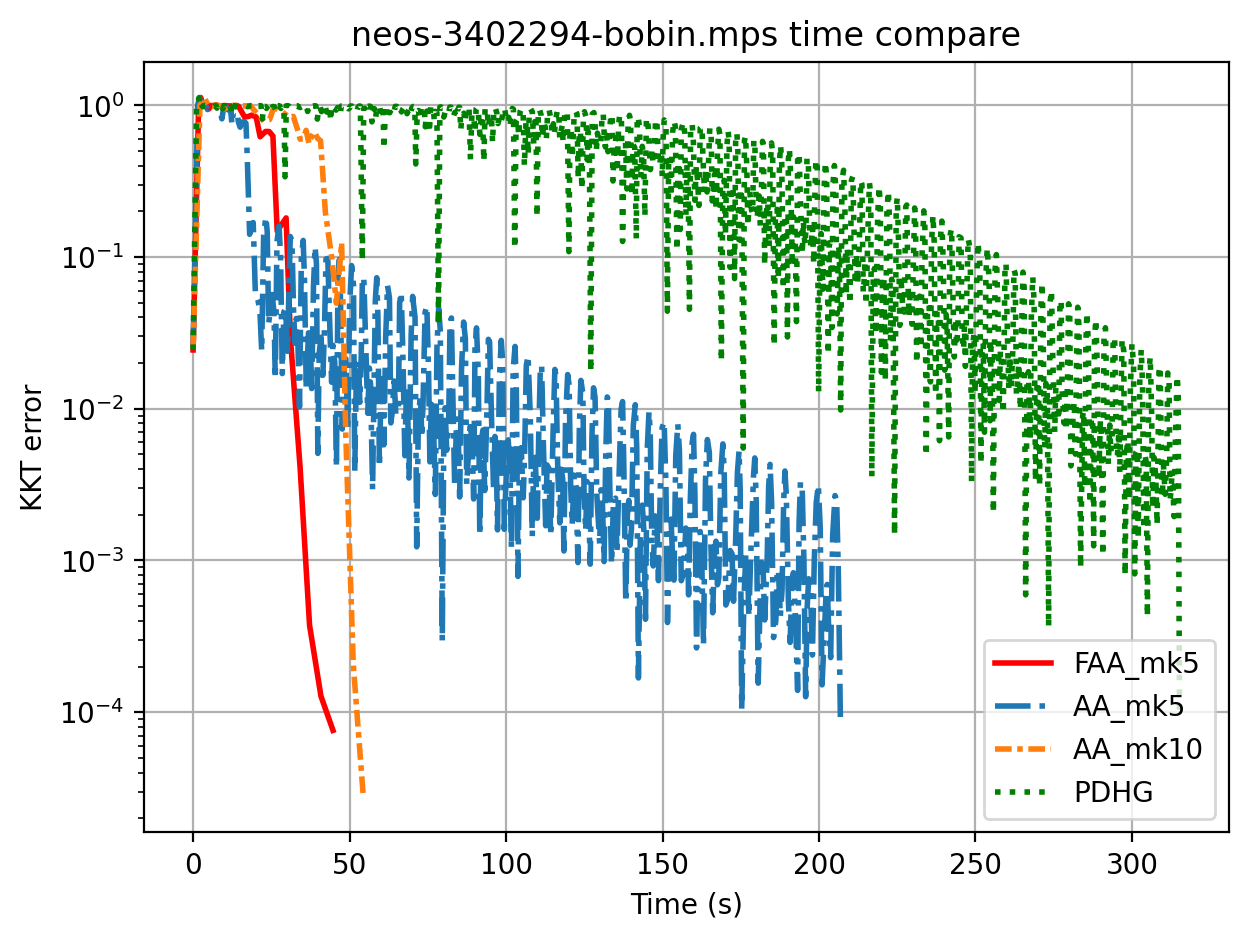}
\caption{The KKT residual trajectory comparison over iterations (left) and time (right).}
\label{fi:fixresidual-iteration}
\end{figure}

\subsection{Comparison with rPDHG}\label{se:rpdhg_comparison}
 {We now turn to the central question of this work: whether Anderson Acceleration can serve as a competitive alternative to the restart strategy for PDHG. To this end, we compare the performance of AA-PDHG against restart PDHG (rPDHG) as proposed by \cite{fasterPDHG-23}.}

 {We begin with a simple illustrative example. Figure \ref{fig:rPD-AA-cir} compares the iterate trajectories of AA-PDHG ($m_A=5$), fixed-restart PDHG, and vanilla PDHG for the following feasibility problem:}
\begin{equation}\label{eq:simple-feasibility}
\begin{aligned}
    &\min_x \; 0,\\
    &\text{s.t. } x=3,\; x\geq 0.
\end{aligned}
\end{equation}

\begin{figure}[!t]
\centering
    \includegraphics[width=0.5\linewidth]{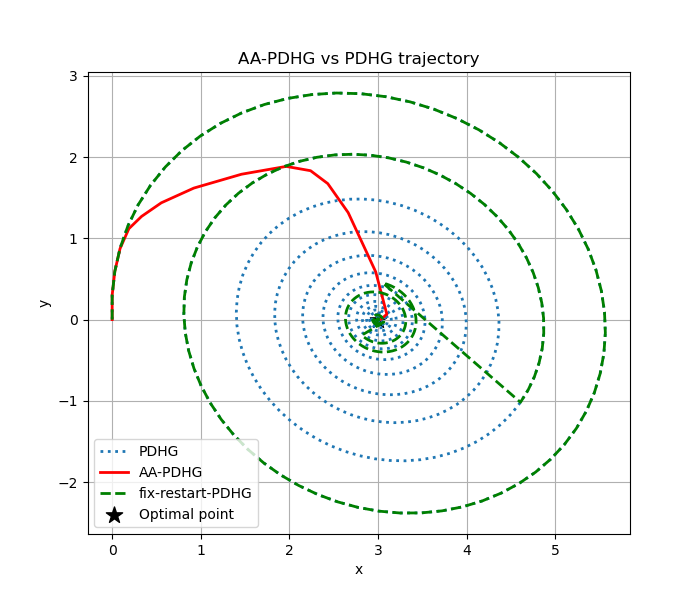}
    \caption{Comparison of the iterates produced by AA-PDHG, rPDHG and Vanilla PDHG. The optimal value is $(3,0)$. We use the same step-size $0.1$ for all. The restart length for rPDHG is 75.}
    \label{fig:rPD-AA-cir}
\end{figure}

 {From Figure \ref{fig:rPD-AA-cir}, we observe that AA-PDHG converges to the optimal solution in only a few iterations. In contrast, rPDHG initially exhibits a spiral trajectory, moves rapidly toward the solution, but then spirals again before converging. This qualitative difference illustrates the distinct acceleration mechanisms: while the restart strategy periodically reinitialises the iterates, AA continuously refines the search direction using accumulated historical information.}

 {For the comparison with rPDHG, the chosen instances are the same as in \cite{google2022p} and the complete list can be found in \texttt{mip\_relaxations\_instance\_list} of the related GitHub repository. Such instances have in-between $100K$-$10M$ non zeros, placing them in a medium-to-large-scale regime that is sufficient to reveal meaningful differences in algorithmic behaviour.}

 {Moreover, in the following,  \textit{pre-solve} and \textit{preconditioning/scaling}  will indicate that the problem instances are preprocessed using the same techniques as in \cite{google2022p}. Specifically, the preconditioning/scaling is performed using the Chambole-Pock technique combined with Ruiz scaling for $10$ iterations. 
It is important to note that in this setting, the filtering step used in FAA-PDHG incurs considerable computational overhead, making the method impractical. Since our objective is to assess whether AA can serve as a competitive alternative to the restart strategy, we focus exclusively on comparing AA-PDHG against rPDHG. 
  {We observe that, despite the presence of the Tikhonov regularisation parameter $\eta$, Algorithm~\ref{al-PD-safe} does not guarantee the uniform boundedness of $\|H^k\|$ in general. Indeed, from the expression of $H^k$ in \eqref{eq:AA-update_filter1}, a standard submultiplicativity argument yields
\[
\|H^k\| \leq \beta\|\widehat{D}\| + \big\|\Delta\mathcal{U}^{k-m_k} + \beta\widehat{D}\,\Delta\mathcal{G}^{k-m_k}\big\|\,\big\|\big((\Delta\mathcal{G}^{k-m_k})^\top\Delta\mathcal{G}^{k-m_k} + \eta I\big)^{-1}(\Delta\mathcal{G}^{k-m_k})^\top\big\|.
\]
The Tikhonov regularisation controls the second factor: since the singular values of $(A^\top A + \eta I)^{-1}A^\top$ are $\sigma_i/(\sigma_i^2 + \eta)$ and $\max_{\sigma \geq 0}\,\sigma/(\sigma^2 + \eta) = 1/(2\sqrt{\eta})$, we obtain
\[
\big\|\big((\Delta\mathcal{G}^{k-m_k})^\top\Delta\mathcal{G}^{k-m_k} + \eta I\big)^{-1}(\Delta\mathcal{G}^{k-m_k})^\top\big\| \leq \frac{1}{2\sqrt{\eta}}.
\]
However, the first factor $\|\Delta\mathcal{U}^{k-m_k} + \beta\widehat{D}\,\Delta\mathcal{G}^{k-m_k}\|$ depends on the norms of the iterate differences, which are not a priori uniformly bounded. Hence, $\eta$ alone does not suffice to ensure $\|H^k\| \leq M$ for all~$k$. Nonetheless, keeping the memory size $m_A$ small mitigates this issue in practice. Since $\Delta\mathcal{U}^{k-m_k}$ has at most $m_k \;(\leq m_A)$ columns, we have
\[
\|\Delta\mathcal{U}^{k-m_k}\| \leq \|\Delta\mathcal{U}^{k-m_k}\|_F = \bigg(\sum_{j=1}^{m_k}\|\Delta u^{k-j}\|^2\bigg)^{1/2} \leq \sqrt{m_k}\,\max_{1 \leq j \leq m_k}\|\Delta u^{k-j}\|,
\]
and similarly for $\Delta\mathcal{G}^{k-m_k}$. With a small value of $m_A$, only a few recent iterate differences contribute to $\|H^k\|$, and as the algorithm converges -- ensured by the safeguard mechanism -- these differences decrease in magnitude, effectively keeping $\|H^k\|$ controlled.}
The rPDHG implementation used in our experiments is based on the publicly available code at \url{https://github.com/google-research/google-research/tree/master/restarting_FOM_for_LP}. To ensure a fair comparison, both AA-PDHG and rPDHG use the same pre-solve, preconditioning, step-size selection, and primal weight update procedures. The only difference lies in the acceleration mechanism: restart for rPDHG versus Anderson Acceleration for AA-PDHG.}

 {The experimental procedure is as follows. We first conduct hyperparameter tuning on 50 randomly selected instances from the dataset, using performance profiles \cite{ASPerprof2016} based on running time (and iteration count for completeness) to determine the best configuration. We then evaluate AA-PDHG with its tuned hyperparameters against rPDHG on both the 50 tuning instances and the entire dataset. Two instances failed during the pre-solve phase, yielding a final comparison on 381 instances.}

 {The hyperparameters of AA-PDHG were selected through a systematic step-by-step tuning procedure on 50 randomly selected instances from the dataset. The tuning covers the memory size $m_A$, the regularization parameter $\eta$, and the primal weight update strategy. The detailed parameter selection experiments and their results are reported in Section~\ref{appen:hyper} of the Supplementary Material. Based on this analysis, the final configuration used in the following 
comparisons is: $m_A = 5$, $\eta = 10^{-10}$, smoothing parameter $\theta = 0.5$, 
and primal weight update every $3000$ iterations.  For the safeguard parameters, 
we fix $D = 1$ and $\varepsilon = 1$  as these choices provide a 
good balance between acceleration frequency and stability across all tested instances.  
    }
\begin{figure}[htb!]
    \centering
    \begin{subfigure}{0.4\linewidth}
        \centering
    \includegraphics[width=\linewidth]{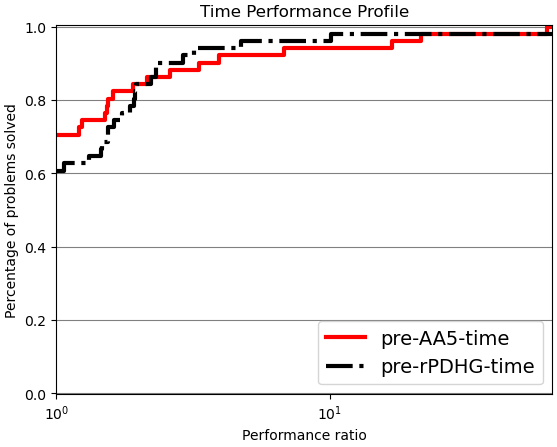}
    \end{subfigure}
    \hfill
    \begin{subfigure}{0.4\linewidth}
    \includegraphics[width=\linewidth]{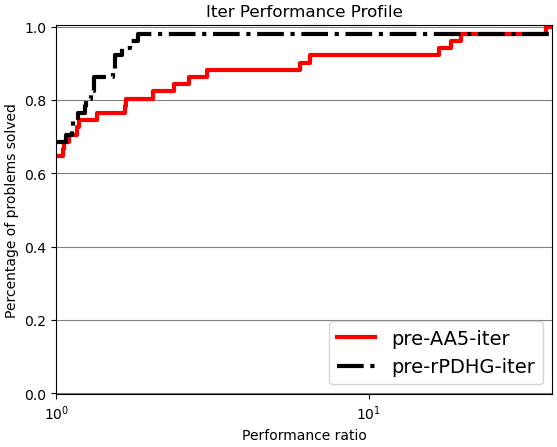}
    \end{subfigure}
    \caption{Time and iteration performance comparison between AA-PDHG and rPDHG on pre-solved randomly selected 50 instances.}
    \label{fig:AA-rPDHG-pre-rm50}
\end{figure}

 {Figure \ref{fig:AA-rPDHG-pre-rm50} shows both, the time and iteration comparison of AA-PDHG and rPDHG on $50$ randomly selected instances.}
 {Note that comparison in Figure \ref{fig:AA-rPDHG-pre-rm50} does not include the primal-weight update for both algorithms. We observe that AA-PDHG is the fastest solver for about $70$\% of the selected instances. }

\begin{figure}[htb!]
\centering
\begin{subfigure}{0.42\linewidth}
\centering
\includegraphics[width=\linewidth]{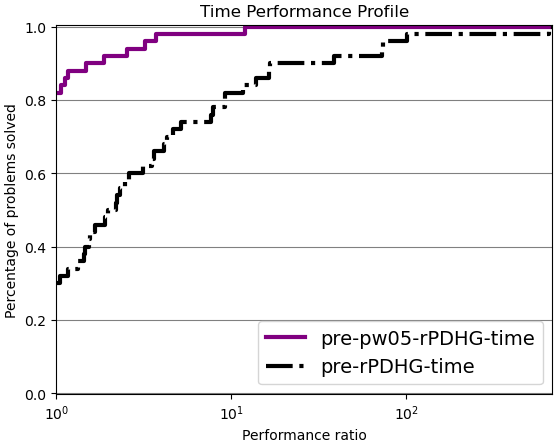}
\end{subfigure}
\hfill
\begin{subfigure}{0.42\linewidth}
\centering
\includegraphics[width=\linewidth]{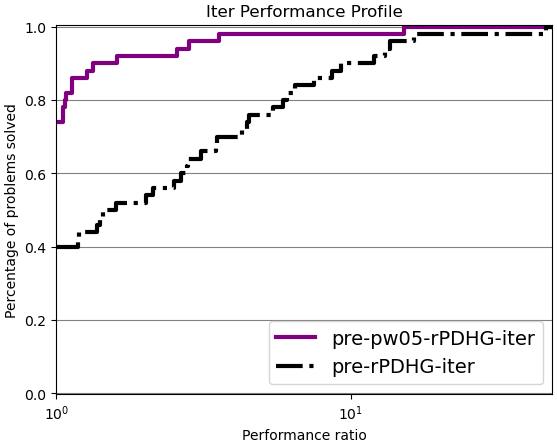}
\end{subfigure}
\caption{Time and iteration performance profiles of rPDHG with and without primal-weight tuning on 50 randomly selected instances. Left: time comparison. Right: iteration comparison.}
\label{fig:pwrPDHG}
\end{figure}

\begin{figure}[htb!]
\centering
\begin{subfigure}{0.42\linewidth}
\centering
\includegraphics[width=\linewidth]{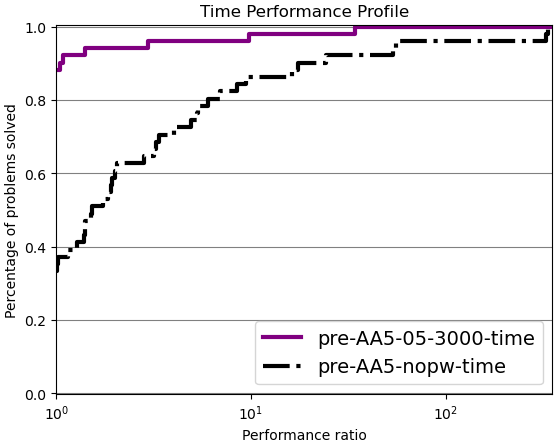}
\end{subfigure}
\hfill
\begin{subfigure}{0.42\linewidth}
\centering
\includegraphics[width=\linewidth]{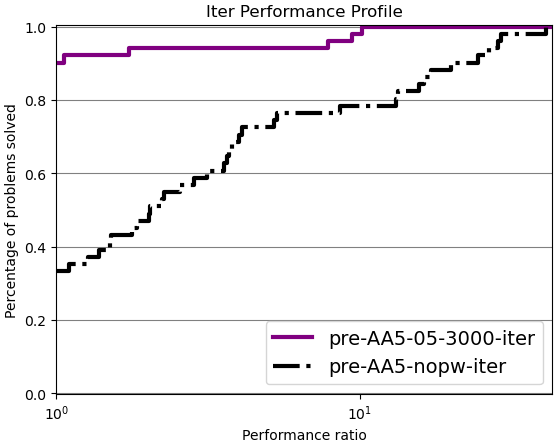}
\end{subfigure}
\caption{Time and iteration performance profiles of AA-PDHG with and without primal-weight tuning on 50 randomly selected instances. Left: time comparison. Right: iteration comparison.}
\label{fig:pwAA}
\end{figure}

 {Next, we present the computational results obtained using the primal weight  update for both rPDHG and AA-PDHG in Figure \ref{fig:pwrPDHG} and Figure \ref{fig:pwAA} (see Algorithm \ref{alg:smooting}). Here we follow the suggestion in \cite{google2022p} to select
the smoothing parameter to be $0.5$ for rPDHG and we use the same smoothing parameter choice in our implementation with the updating period $U_P=3000$. We see from the figures that though the primal weight update in our AA-PDHG framework is slightly different to that of rPDHG, it still provides a computational advantage.}   


 {Figure \ref{fig:AA-rPDHG-pre-whole} presents the performance profile comparison between AA-PDHG and rPDHG on the full dataset of 381 pre-solved instances.}  {Specifically, we compare there the performance of AA-PDHG and rPDHG both with and without primal-weight updates. Note that, when applying primal-weight updates ($U_p=3000$) in AA-PDHG, we also update $\widehat D$ according to the strategy described in Remark~\ref{re:pwform}. We observe that, on the full pre-solved dataset, AA-PDHG outperforms rPDHG overall. Specifically, when neither method uses primal weight updates, AA is the best-performing method on about 70\% of the instances. Although the performance gap becomes smaller once both methods use their respective primal weight update strategies, AA-PDHG remains consistently competitive and performs better on about 60\% of the instances. Moreover, after incorporating primal weight updates, AA shows a noticeable improvement on more challenging instances, as reflected by the tail of the performance profile becoming much closer to that of rPDHG.}
\begin{figure}[htb!]
    \centering
    \begin{subfigure}{0.4\linewidth}
        \centering
    \includegraphics[width=\linewidth]{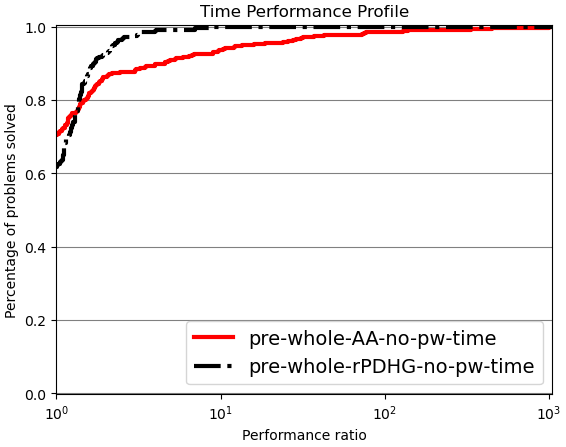}
    \end{subfigure}
    \hfill
    \begin{subfigure}{0.4\linewidth}
        \centering
    \includegraphics[width=\linewidth]{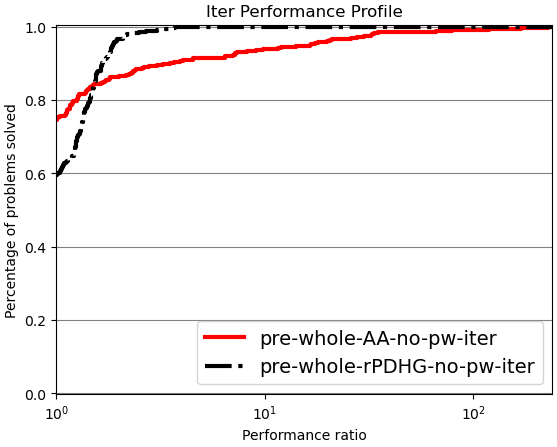}
    \end{subfigure}
    \hfill
    \begin{subfigure}{0.4\linewidth}
        \centering
    \includegraphics[width=\linewidth]{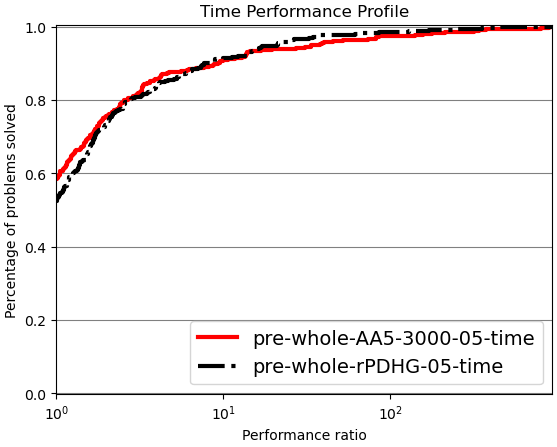}
    \end{subfigure}
    \hfill
    \begin{subfigure}{0.4\linewidth}
    \includegraphics[width=\linewidth]{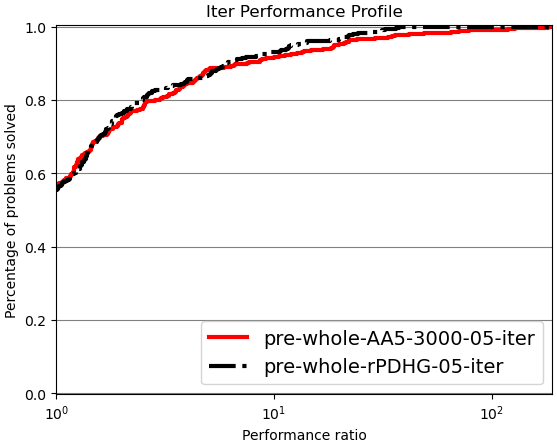}
    \end{subfigure}
        \caption{Time and iteration performance comparison between AA-PDHG and rPDHG on whole pre-solved dataset. Top: neither method uses primal weight updates. Bottom: both methods use their respective primal weight update strategies. Left: Time comparison. Right: Iteration comparison.}
    \label{fig:AA-rPDHG-pre-whole}
\end{figure}


 {To assess the robustness of these findings, we repeat the comparison on the unpresolved dataset. This allows us to assess whether the observed performance differences persist on the original problem instances. For brevity, we report only the final results in Figure~\ref{fig:unpreAA-rPDHG-whole_it}.}

\begin{figure}[htb!]
    \centering
    \begin{subfigure}{0.4\linewidth}
             \includegraphics[width=\linewidth]{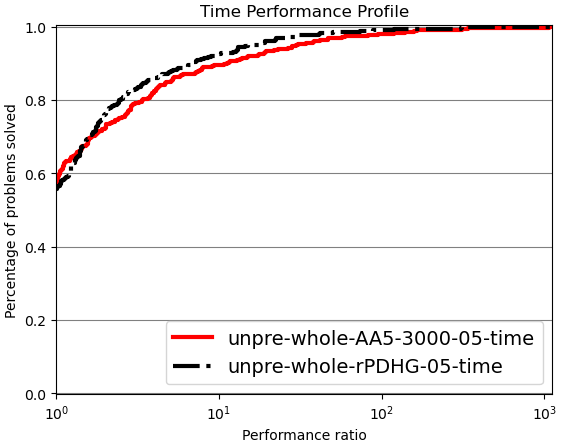}
    \end{subfigure}
    \hfill
    \begin{subfigure}{0.4\linewidth}
             \includegraphics[width=\linewidth]{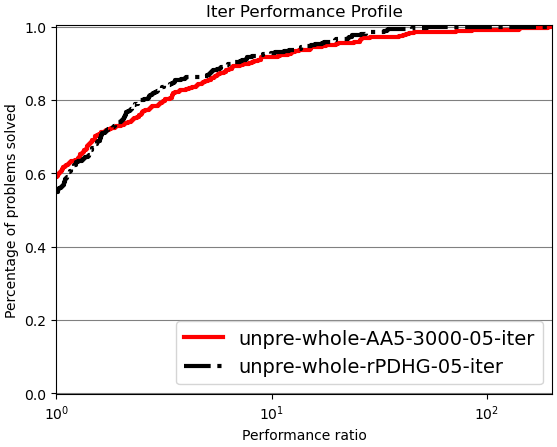}
    \end{subfigure}
      \caption{Performance comparison between AA-PDHG and rPDHG on the whole unpresolved dataset. \\
      Left: time comparison. Right: iteration comparison.}
   \label{fig:unpreAA-rPDHG-whole_it}
\end{figure}

 {
As shown in Figure~\ref{fig:unpreAA-rPDHG-whole_it}, the running-time profiles of the two methods are close on the unpresolved dataset, with AA-PDHG being slightly better on a subset of instances. The iteration-count profiles are also close, with AA-PDHG outperforming rPDHG on about $60$\% of the instances. Overall, this confirms that AA-PDHG remains competitive with rPDHG even without pre-solve.
}

\section{Conclusion}
 {In this work, we investigated whether Anderson Acceleration (AA) can serve as a viable alternative to the restart strategy for accelerating the Primal-Dual Hybrid Gradient (PDHG) method applied to linear programming (LP) problems. To this end, we reformulated PDHG as a fixed-point iteration and integrated AA into this framework, introducing an explicit projection step to maintain feasibility and a safeguard mechanism to guarantee global convergence. We further proposed FAA-PDHG, a filtered variant that enforces the uniform boundedness of the AA coefficient matrix through angle and length filtering, providing a rigorous convergence guarantee. While the filtering overhead makes FAA-PDHG less competitive in wall-clock time, it serves an 
	important theoretical role: it demonstrates that the boundedness assumption required 
	by Theorem~\ref{th:FAA-convergence} can be enforced algorithmically.
	  {This yields a two-tier design: FAA-PDHG provides the rigorous convergence guarantee, while the unfiltered AA-PDHG serves as the practical workhorse. The two variants are connected by the observation that the convergence analysis (Theorem~\ref{th:con}) depends only on $\|H^k\| \leq M$ and not on how this bound is achieved.}
In practice, the conditioning of $H_k$ may deteriorate as the memory parameter $m_A$ increases. For the moderate value used in our experiments, $m_A=5$, the unfiltered AA-PDHG variant was numerically stable and avoided the overhead of explicit filtering. This motivates its use as the practical workhorse in the numerical comparisons. Nevertheless, for unfiltered AA-PDHG the uniform boundedness of $H_k$ remains an assumption in Theorem \ref{th:con}; FAA-PDHG shows that this assumption can be enforced algorithmically when a fully rigorous variant is required.}

 {The numerical experiments on the MIPLIB dataset provide evidence supporting the viability of AA as an acceleration mechanism for PDHG. In comparison with vanilla PDHG, both AA-PDHG and FAA-PDHG deliver substantial speedups on selected medium- to large-scale LP instances.  {
More significantly, the direct comparison with restart PDHG (rPDHG) demonstrates that AA-PDHG achieves the fastest running time on approximately $70$\% of the pre-solved instances when neither method uses primal-weight updates. When both methods incorporate their respective primal-weight update strategies, AA-PDHG remains competitive and achieves better performance on about $60$\% of the instances.}
}
\section*{Acknowledgments}
The authors acknowledge the use of the IRIDIS High Performance Computing Facility and associated support services at the University of Southampton, in the completion of this work.

\bibliographystyle{plainnat}
\bibliography{references} 

@article{MILP,
	author                   = {Gleixner, Ambros and Hendel, Gregor and Gamrath, Gerald and Achterberg, Tobias and Bastubbe, Michael and Berthold, Timo and Christophel, Philipp M. and Jarck, Kati and Koch, Thorsten and Linderoth, Jeff and L\"ubbecke, Marco and Mittelmann, Hans D. and Ozyurt, Derya and Ralphs, Ted K. and Salvagnin, Domenico and Shinano, Yuji},
	title                    = {{MIPLIB 2017: Data-Driven Compilation of the 6th Mixed-Integer Programming Library}},
	journal                  = {Mathematical Programming Computation},
	year                     = {2021},
	doi                      = {10.1007/s12532-020-00194-3},
	url                      = {https://doi.org/10.1007/s12532-020-00194-3}
}

@article {MR4711308,
    AUTHOR = {De Sterck, Hans and He, Yunhui and Krzysik, Oliver A.},
     TITLE = {Anderson acceleration as a {K}rylov method with application to
              convergence analysis},
   JOURNAL = {J. Sci. Comput.},
  FJOURNAL = {Journal of Scientific Computing},
    VOLUME = {99},
      YEAR = {2024},
    NUMBER = {1},
     PAGES = {Paper No. 12, 30},
      ISSN = {0885-7474,1573-7691},
   MRCLASS = {65B05 (65F10 65H10 65K10)},
  MRNUMBER = {4711308},
}

@Misc{amsmath,
  author =	 {{American Mathematical Society}},
  title =	 {User's Guide for the \texttt{amsmath} Package
                  (Version 2.0)},
  url =		 {ftp://ftp.ams.org/pub/tex/doc/amsmath/amsldoc.pdf},
  urldate =	 {2015-07-30},
  year =	 2002}

@article{google2022p,
  title={Practical large-scale linear programming using primal-dual hybrid gradient},
  author={Applegate, David and D{\'\i}az, Mateo and Hinder, Oliver and Lu, Haihao and Lubin, Miles and O'Donoghue, Brendan and Schudy, Warren},
  journal={Proc. NeurIPS},
  volume={34},
  pages={20243--20257},
  year={2021}
}

@article {degenerate2021,
    AUTHOR = {Bredies, Kristian and Chenchene, Enis and Lorenz, Dirk A. and
              Naldi, Emanuele},
     TITLE = {Degenerate preconditioned proximal point algorithms},
   JOURNAL = {SIAM J. Optim.},
  FJOURNAL = {SIAM Journal on Optimization},
    VOLUME = {32},
      YEAR = {2022},
    NUMBER = {3},
     PAGES = {2376--2401},
}

@book {convex-monotone,
    AUTHOR = {Bauschke, Heinz H. and Combettes, Patrick L.},
     TITLE = {Convex analysis and monotone operator theory in {H}ilbert
              spaces},
    SERIES = {CMS Books in Mathematics/Ouvrages de Math\'ematiques de la
              SMC},
      NOTE = {With a foreword by H\'edy Attouch},
 PUBLISHER = {Springer, New York},
      YEAR = {2011},
     PAGES = {xvi+468},
}

@article {ChambollePDHG,
    AUTHOR = {Chambolle, Antonin and Pock, Thomas},
     TITLE = {A first-order primal-dual algorithm for convex problems with
              applications to imaging},
   JOURNAL = {J. Math. Imaging Vision},
  FJOURNAL = {J. Math. Imaging Vision},
    VOLUME = {40},
      YEAR = {2011},
    NUMBER = {1},
     PAGES = {120--145},
 }

@book {first-order-beck,
    AUTHOR = {Beck, Amir},
     TITLE = {First-order methods in optimization},
    SERIES = {MOS-SIAM Series on Optimization},
    VOLUME = {25},
 PUBLISHER = {Society for Industrial and Applied Mathematics (SIAM),
              Philadelphia, PA; Mathematical Optimization Society,
              Philadelphia, PA},
      YEAR = {2017},
     PAGES = {xii+475},
       doi = {10.1137/1.9781611974997.ch1},
       URL = {https://doi.org/10.1137/1.9781611974997.ch1},
}

@article {quasi-Fejer,
    AUTHOR = {Combettes, Patrick L. and V\~u, B\brevgrv ang C.},
     TITLE = {Variable metric quasi-{F}ej\'er monotonicity},
   JOURNAL = {Nonlinear Anal.},
  FJOURNAL = {Nonlinear Analysis. Theory, Methods \& Applications. An
              International Multidisciplinary Journal},
    VOLUME = {78},
      YEAR = {2013},
     PAGES = {17--31},
   }

@article {ANDERSON-DRS-Fu_2020,
    AUTHOR = {Fu, Anqi and Zhang, Junzi and Boyd, Stephen},
     TITLE = {Anderson accelerated {D}ouglas-{R}achford splitting},
   JOURNAL = {SIAM J. Sci. Comput.},
  FJOURNAL = { SIAM Journal on Scientific Computing},
    VOLUME = {42},
      YEAR = {2020},
    NUMBER = {6},
     PAGES = {A3560--A3583},
 }

@article{anderson-accelerated-operatorsplitting,
  author  = {Heng, Qiang and Liu, Xiaoqian and Chi, Eric C.},
  title   = {Anderson Accelerated Operator Splitting Methods for Convex-nonconvex Regularized Problems},
  journal = {arXiv preprint arXiv:2502.14269},
  year    = {2025}
}

@article{anderson-type-I,
    author = {Zhang, Junzi and O'Donoghue, Brendan and Boyd, Stephen},
    title = {Globally Convergent Type-I Anderson Acceleration for Nonsmooth Fixed-Point Iterations},
    journal = {SIAM J. Optim.},
    volume = {30},
    number = {4},
    pages = {3170-3197},
    year = {2020},
  }

@article{PDLPnew2025,
  title={PDLP: A Practical First-Order Method for Large-Scale Linear Programming},
  author={Applegate, David and D{\'\i}az, Mateo and Hinder, Oliver and Lu, Haihao and Lubin, Miles and O'Donoghue, Brendan and Schudy, Warren},
  journal={arXiv preprint arXiv:2501.07018},
  year={2025}
}

@article {AAoriginal,
    AUTHOR = {Anderson, Donald G.},
     TITLE = {Iterative procedures for nonlinear integral equations},
   JOURNAL = {J. Assoc. Comput. Mach.},
  FJOURNAL = {Journal of the Association for Computing Machinery},
    VOLUME = {12},
      YEAR = {1965},
     PAGES = {547--560},
}

@article {AA-FPI-NI,
    AUTHOR = {Walker, Homer F. and Ni, Peng},
     TITLE = {Anderson acceleration for fixed-point iterations},
   JOURNAL = {SIAM J. Numer. Anal.},
  FJOURNAL = { SIAM Journal on Numerical Analysis},
    VOLUME = {49},
      YEAR = {2011},
    NUMBER = {4},
     PAGES = {1715--1735},
}

@article {AA-CON-TAK,
    AUTHOR = {Toth, Alex and Kelley, C. T.},
     TITLE = {Convergence analysis for {A}nderson acceleration},
   JOURNAL = {SIAM J. Numer. Anal.},
  FJOURNAL = {SIAM Journal on Numerical Analysis},
    VOLUME = {53},
      YEAR = {2015},
    NUMBER = {2},
     PAGES = {805--819},
}

@article {AA-IMPROVE-PROOF-ECPRX,
    AUTHOR = {Evans, Claire and Pollock, Sara and Rebholz, Leo G. and Xiao,
              Mengying},
     TITLE = {A proof that {A}nderson acceleration improves the convergence
              rate in linearly converging fixed-point methods (but not in
              those converging quadratically)},
   JOURNAL = {SIAM J. Numer. Anal.},
  FJOURNAL = {SIAM Journal on Numerical Analysis},
    VOLUME = {58},
      YEAR = {2020},
    NUMBER = {1},
     PAGES = {788--810},
       doi = {10.1137/19M1245384},
       URL = {https://doi.org/10.1137/19M1245384},
}

@article {AA-QUASI-NEWTON,
    AUTHOR = {Fang, Haw-ren and Saad, Yousef},
     TITLE = {Two classes of multisecant methods for nonlinear acceleration},
   JOURNAL = {Numer. Linear Algebra Appl.},
  FJOURNAL = {Numerical Linear Algebra with Applications},
    VOLUME = {16},
      YEAR = {2009},
    NUMBER = {3},
     PAGES = {197--221},
       doi = {10.1002/nla.617},
       URL = {https://doi.org/10.1002/nla.617},
}

@book {variational-Rock,
    AUTHOR = {Rockafellar, R. Tyrrell and Wets, Roger J.-B.},
     TITLE = {Variational analysis},
    SERIES = {Grundlehren der mathematischen Wissenschaften [Fundamental
              Principles of Mathematical Sciences]},
    VOLUME = {317},
 PUBLISHER = {Springer-Verlag, Berlin},
      YEAR = {1998},
     PAGES = {xiv+733},
       doi = {10.1007/978-3-642-02431-3},
       URL = {https://doi.org/10.1007/978-3-642-02431-3},
}

@article {QEB-FE2023,
    AUTHOR = {Fercoq, Olivier},
     TITLE = {Quadratic error bound of the smoothed gap and the restarted
              averaged primal-dual hybrid gradient},
   JOURNAL = {Open J. Math. Optim.},
  FJOURNAL = {Open Journal of Mathematical Optimization (OJMO)},
    VOLUME = {4},
      YEAR = {2023},
     PAGES = {Art. No. 6, 34},
       doi = {10.5802/ojmo.26},
       URL = {https://doi.org/10.5802/ojmo.26},
}

@article {fasterPDHG-23,
    AUTHOR = {Applegate, David and Hinder, Oliver and Lu, Haihao and Lubin,
              Miles},
     TITLE = {Faster first-order primal-dual methods for linear programming
              using restarts and sharpness},
   JOURNAL = {Math. Program.},
  FJOURNAL = {Mathematical Programming},
    VOLUME = {201},
      YEAR = {2023},
    NUMBER = {1-2},
     PAGES = {133--184},
       doi = {10.1007/s10107-022-01901-9},
       URL = {https://doi.org/10.1007/s10107-022-01901-9},
}

@article{ADMMLP17,
  title={A new alternating direction method for linear programming},
  author={Wang, Sinong and Shroff, Ness},
  journal={NeurIPS},
  volume={30},
  year={2017}
}

@article {PGLP87,
    AUTHOR = {Calamai, Paul H. and Mor\'e, Jorge J.},
     TITLE = {Projected gradient methods for linearly constrained problems},
   JOURNAL = {Math. Programming},
  FJOURNAL = {Math. Program.},
    VOLUME = {39},
      YEAR = {1987},
    NUMBER = {1},
     PAGES = {93--116},
       doi = {10.1007/BF02592073},
       URL = {https://doi.org/10.1007/BF02592073},
}

@article {SDLP89,
    AUTHOR = {Chang, Soo Y. and Murty, Katta G.},
     TITLE = {The steepest descent gravitational method for linear
              programming},
   JOURNAL = {Discrete Appl. Math.},
  FJOURNAL = {Discrete Applied Mathematics},
    VOLUME = {25},
      YEAR = {1989},
    NUMBER = {3},
     PAGES = {211--239},
       doi = {10.1016/0166-218X(89)90002-4},
       URL = {https://doi.org/10.1016/0166-218X(89)90002-4},
}

@article {AFOM-SPECIAL-LP19,
    AUTHOR = {Renegar, James},
     TITLE = {Accelerated first-order methods for hyperbolic programming},
   JOURNAL = {Math. Program.},
  FJOURNAL = {Mathematical Programming},
    VOLUME = {173},
      YEAR = {2019},
    NUMBER = {1-2},
     PAGES = {1--35},
       doi = {10.1007/s10107-017-1203-y},
       URL = {https://doi.org/10.1007/s10107-017-1203-y},
}

@article {AFOM-SPECIAL-LP11,
    AUTHOR = {Lan, Guanghui and Lu, Zhaosong and Monteiro, Renato D. C.},
     TITLE = {Primal-dual first-order methods with {$\mathcal{O}(1/\epsilon)$}
              iteration-complexity for cone programming},
   JOURNAL = {Math. Program.},
  FJOURNAL = {Mathematical Programming},
    VOLUME = {126},
      YEAR = {2011},
    NUMBER = {1},
     PAGES = {1--29},
       doi = {10.1007/s10107-008-0261-6},
       URL = {https://doi.org/10.1007/s10107-008-0261-6},
}

@article {ACPDHG21,
    AUTHOR = {Liu, Yanli and Xu, Yunbei and Yin, Wotao},
     TITLE = {Acceleration of primal-dual methods by preconditioning and
              simple subproblem procedures},
   JOURNAL = {J. Sci. Comput.},
  FJOURNAL = {Journal of Scientific Computing},
    VOLUME = {86},
      YEAR = {2021},
    NUMBER = {2},
     PAGES = {Paper No. 21, 34},
       doi = {10.1007/s10915-020-01371-1},
       URL = {https://doi.org/10.1007/s10915-020-01371-1},
}

@article {LINE-PDHG18,
    AUTHOR = {Malitsky, Yura and Pock, Thomas},
     TITLE = {A first-order primal-dual algorithm with linesearch},
   JOURNAL = {SIAM J. Optim.},
  FJOURNAL = {SIAM Journal on Optimization},
    VOLUME = {28},
      YEAR = {2018},
    NUMBER = {1},
     PAGES = {411--432},
       doi = {10.1137/16M1092015},
       URL = {https://doi.org/10.1137/16M1092015},
}

@article {IPDHG-16,
    AUTHOR = {Chambolle, Antonin and Pock, Thomas},
     TITLE = {On the ergodic convergence rates of a first-order primal-dual
              algorithm},
   JOURNAL = {Math. Program.},
  FJOURNAL = {Mathematical Programming},
    VOLUME = {159},
      YEAR = {2016},
    NUMBER = {1-2},
     PAGES = {253--287},
       doi = {10.1007/s10107-015-0957-3},
       URL = {https://doi.org/10.1007/s10107-015-0957-3},
}

@article {AACONTRACTIVE-21,
    AUTHOR = {Pollock, Sara and Rebholz, Leo G.},
     TITLE = {Anderson acceleration for contractive and noncontractive
              operators},
   JOURNAL = {IMA J. Numer. Anal.},
  FJOURNAL = {IMA Journal of Numerical Analysis},
    VOLUME = {41},
      YEAR = {2021},
    NUMBER = {4},
     PAGES = {2841--2872},
       doi = {10.1093/imanum/draa095},
       URL = {https://doi.org/10.1093/imanum/draa095},
}

@article {AAMIXING19,
    AUTHOR = {Anderson, Donald G. M.},
     TITLE = {Comments on ``{A}nderson acceleration, mixing and
              extrapolation''},
   JOURNAL = {Numer. Algorithms},
  FJOURNAL = {Numerical Algorithms},
    VOLUME = {80},
      YEAR = {2019},
    NUMBER = {1},
     PAGES = {135--234},
       doi = {10.1007/s11075-018-0549-4},
       URL = {https://doi.org/10.1007/s11075-018-0549-4},
}

@article{dynamicBETAAA-25,
  title={Anderson acceleration of derivative-free projection methods for constrained monotone nonlinear equations},
  author={Jin, Jiachen and Wang, Hongxia and Deng, Kangkang},
  journal={arXiv preprint arXiv:2403.14924},
  year={2024}
}

@article {EDIIS-19,
    AUTHOR = {Chen, Xiaojun and Kelley, C. T.},
     TITLE = {Convergence of the {EDIIS} algorithm for nonlinear equations},
   JOURNAL = {SIAM J. Sci. Comput.},
  FJOURNAL = {SIAM Journal on Scientific Computing},
    VOLUME = {41},
      YEAR = {2019},
    NUMBER = {1},
     PAGES = {A365--A379},
       doi = {10.1137/18M1171084},
       URL = {https://doi.org/10.1137/18M1171084},
}

@article {AA-PG-21,
    AUTHOR = {De Sterck, Hans and He, Yunhui},
     TITLE = {On the asymptotic linear convergence speed of {A}nderson
              acceleration, {N}esterov acceleration, and nonlinear {GMRES}},
   JOURNAL = {SIAM J. Sci. Comput.},
  FJOURNAL = {SIAM Journal on Scientific Computing},
    VOLUME = {43},
      YEAR = {2021},
    NUMBER = {5},
     PAGES = {S21--S46},
       doi = {10.1137/20M1347139},
       URL = {https://doi.org/10.1137/20M1347139},
}

@article{EG-AA-24,
  title={An extra gradient Anderson-accelerated algorithm for pseudomonotone variational inequalities},
  author={Qu, Xin and Bian, Wei and Chen, Xiaojun},
  journal={arXiv e-prints},
  pages={arXiv--2408},
  year={2024}
}

@book {power-iter-96,
    AUTHOR = {Golub, Gene H. and Van Loan, Charles F.},
     TITLE = {Matrix computations},
    SERIES = {Johns Hopkins Studies in the Mathematical Sciences},
   EDITION = {Third},
 PUBLISHER = {Johns Hopkins University Press, Baltimore, MD},
      YEAR = {1996},
     PAGES = {xxx+698},
}

@inproceedings{diagonalPDHG-11,
  title={Diagonal preconditioning for first order primal-dual algorithms in convex optimization},
  author={Pock, Thomas and Chambolle, Antonin},
  booktitle={2011 International Conference on Computer Vision},
  pages={1762--1769},
  year={2011},
  organization={IEEE}
}

@article{AA-NONLINEAR-21,
  author = {Brezinski, Claude and Cipolla, Stefano and Redivo-Zaglia, Michela and Saad, Yousef},
  title = {Shanks and Anderson-type acceleration techniques for systems of nonlinear equations},
  journal = {IMA J. Numer. Anal.},
  volume = {42},
  number = {4},
  pages = {3058-3093},
  year = {2021},
doi = {10.1093/imanum/drab061},
}

@article{GEOMETRYPDHGLP-24,
  title={On the geometry and refined rate of primal--dual hybrid gradient for linear programming},
  author={Lu, Haihao and Yang, Jinwen},
  journal={Math. Program.},
  pages={1--39},
  year={2024},
  publisher={Springer}
}

@article {PRFAA-2023,
    AUTHOR = {Pollock, Sara and Rebholz, Leo G.},
     TITLE = {Filtering for {A}nderson acceleration},
   JOURNAL = {SIAM J. Sci. Comput.},
  FJOURNAL = {SIAM Journal on Scientific Computing},
    VOLUME = {45},
      YEAR = {2023},
    NUMBER = {4},
     PAGES = {A1571--A1590},
      ISSN = {1064-8275,1095-7197},
   MRCLASS = {65B05 (65N30)},
  MRNUMBER = {4611517},
doi = {10.1137/22M1536741},
       URL = {https://doi.org/10.1137/22M1536741},
}

@article {MR3841161,
    AUTHOR = {Brezinski, Claude and Redivo-Zaglia, Michela and Saad, Yousef},
     TITLE = {Shanks sequence transformations and {A}nderson acceleration},
   JOURNAL = {SIAM Rev.},
  FJOURNAL = {SIAM Review},
    VOLUME = {60},
      YEAR = {2018},
    NUMBER = {3},
     PAGES = {646--669},
      ISSN = {0036-1445,1095-7200},
   MRCLASS = {65B05 (65B99 65F10 65H10)},
  MRNUMBER = {3841161},
MRREVIEWER = {Benjamin\ Wi-Lian\ Ong},
doi = {10.1137/17M1120725},
       URL = {https://doi.org/10.1137/17M1120725},
}

@article {MR4926317,
    AUTHOR = {Saad, Yousef},
     TITLE = {Acceleration methods for fixed-point iterations},
   JOURNAL = {Acta Numer.},
  FJOURNAL = {Acta Numerica},
    VOLUME = {34},
      YEAR = {2025},
     PAGES = {805--890},
      ISSN = {0962-4929,1474-0508},
   MRCLASS = {65H10 (65B05 65F10)},
  MRNUMBER = {4926317},
doi = {10.1017/S0962492924000096},
       URL = {https://doi.org/10.1017/S0962492924000096},
}

@article{ASPerprof2016,
  title={Perprof-py: A Python Package for Performance Profile of Mathematical Optimization Software},
  author={Abel Soares Siqueira and Raniere Silva and Luiz-Rafael Santos},
  journal = {J. Open Res. Softw.},
  year={2016},
  volume={4},
  }

@article{OWAAre23,
author = {Ouyang, Wenqing and Tao, Jiong and Milzarek, Andre and Deng, Bailin},
title = {Nonmonotone Globalization for Anderson Acceleration via Adaptive Regularization},
year = {2023},
issue_date = {Jul 2023},
publisher = {Plenum Press},
address = {USA},
volume = {96},
number = {1},
issn = {0885-7474},
journal = {J. Sci. Comput.},
month = may,
numpages = {37},
}

@article{LUGPU25,
author = {Lu, Haihao and Yang, Jinwen},
title = {cuPDLP.jl: A GPU Implementation of Restarted Primal-Dual Hybrid Gradient for Linear Programming in Julia},
journal = {Operations Research},
volume = {73},
number = {6},
pages = {3440-3452},
year = {2025},
}

@article{RBNolinear23,
  author  = {Raghu Bollapragada and Damien Scieur and Alexandre d'Aspremont},
  title   = {Nonlinear acceleration of momentum and primal-dual algorithms},
  journal = {Math. Program.},
  year    = {2023},
  volume  = {198},
  number  = {1},
  pages   = {325--362},
  doi     = {10.1007/s10107-022-01775-x},
  url     = {https://doi.org/10.1007/s10107-022-01775-x}
}

@article{KCHPRLP26,
  author  = {Kaihuang Chen and Defeng Sun and Yancheng Yuan and Guojun Zhang and Xinyuan Zhao},
  title   = {HPR-LP: An implementation of an HPR method for solving linear programming},
  journal = {Math. Program. Comput.},
  year    = {2026},
  volume  = {18},
  number  = {1},
  pages   = {183--210},
}

@inproceedings{LiuCross24,
    title={A New Crossover Algorithm for LP Inspired by the Spiral Dynamic of PDHG}, 
      author={Tianhao Liu and Haihao Lu},
      year={2025},
      eprint={2409.14715},
      archivePrefix={arXiv},
      primaryClass={math.OC},
      url={https://arxiv.org/abs/2409.14715}, 
}

\newpage
\section*{Supplementary Material}
\setcounter{section}{0}
\subsection{ {Explicit expression of {$\Lambda$} and Proof of Remark \ref{remark-lambda}}}  \label{appen:Lambda}

\begin{proof}{Proof}
We first show the expression of $\Lambda$. Let 
$\lambda^+$ and $\lambda^-$ be the corresponding optimal Lagrange multipliers
associated with the lower and upper bound constraints, respectively.
And
$
\lambda_i=\lambda_i^+-\lambda_i^-.$
The possible values of \(\lambda_i\) depend only on whether the lower and upper
bounds are finite.

If \(l_i=-\infty\) and \(u_i=\infty\), then there is no lower or upper bound
constraint. Hence there are no corresponding Lagrange multipliers, and we set $
    \lambda_i^+ = 0,\;\; \lambda_i^- = 0.$
Therefore,$
    \lambda_i = \lambda_i^+ - \lambda_i^- = 0,$
and hence
$$
    \Lambda_i=\{0\}.
$$

If \(l_i=-\infty\) and \(u_i\in\mathbb{R}\), then only the upper bound
constraint \(x_i\le u_i\) is present. Thus there is no lower-bound multiplier,
so \(\lambda_i^+=0\), while \(\lambda_i^-\ge 0\). Therefore,
$
    \lambda_i = \lambda_i^+ - \lambda_i^- = -\lambda_i^- \le 0,
$
and hence
$$
    \Lambda_i=\mathbb{R}_{-}.
$$

If \(l_i\in\mathbb{R}\) and \(u_i=\infty\), then only the lower bound
constraint \(x_i\ge l_i\) is present. Thus there is no upper-bound multiplier,
so \(\lambda_i^-=0\), while \(\lambda_i^+\ge 0\). Therefore,
$
    \lambda_i = \lambda_i^+ - \lambda_i^- = \lambda_i^+ \ge 0,
$
and hence
$$
    \Lambda_i=\mathbb{R}_{+}.
$$

If \(l_i\in\mathbb{R}\) and \(u_i\in\mathbb{R}\), then both lower and
upper bound constraints are present. Hence \(\lambda_i^+\ge 0\) and
\(\lambda_i^-\ge 0\), and
$$
    \lambda_i = \lambda_i^+ - \lambda_i^- \in \mathbb{R}=\Lambda_i.
$$

Consequently, we finally get,
\[
    \Lambda=\Lambda_1\times\cdots\times\Lambda_n,
\]
where
\[
\Lambda_i =
\begin{cases}
\{0\}, & l_i=-\infty,\ u_i=\infty,\\
\mathbb{R}_{-}, & l_i=-\infty,\ u_i\in\mathbb{R},\\
\mathbb{R}_{+}, & l_i\in\mathbb{R},\ u_i=\infty,\\
\mathbb{R}, & l_i\in\mathbb{R},\ u_i\in\mathbb{R}.
\end{cases}
\]

Now let $x$ be an optimal solution of the problem, then the
KKT conditions hold at $x$. In particular, the complementarity conditions give
$$
\lambda_i^+(x_i-l_i)=0\; \text{ if } l_i\in\mathbb{R},
\quad
\lambda_i^-(u_i-x_i)=0 \; \text{ if } u_i\in\mathbb{R}.
$$
Since we have
$
X=\times_{i=1}^n X_i,
$where $
X_i=\{x_i\in\mathbb{R}: l_i\le x_i\le u_i\},
$
 {and the subgradient of the indicator function $I_X(x)$ is its a normal cone, we have,
$
\partial \mathcal{I}_X(x)=\mathcal{N}_{X}(x),
$
where
\begin{align*}
   \mathcal{N}_{X}(x)=\mathcal{N}_{X_1}(x_1)\times \mathcal{N}_{X_2}(x_2)\times \ldots \times \mathcal{N}_{X_n}(x_n)=\partial \mathcal{I}_{X_1}(x_1)\times \ldots \times \partial \mathcal{I}_{X_n}(x_n),
\end{align*} see \cite{variational-Rock}[Proposition 6.41].}
Hence, it suffices to prove that for any optimal point $x_i$, the corresponding $\lambda_i$ would satisfy
$$
-\lambda_i\in \partial \mathcal{I}_{X_i}(x_i)
\quad \text{for each } i.
$$
In the following, we consider 4 scenarios.
\begin{enumerate}
 \item[(1)] If $l_i,u_i\in\mathbb{R}$, then
    \[
    \partial \mathcal{I}_{X_i}(x_i)=
    \begin{cases}
    \{0\}, & l_i<x_i<u_i,\\
    \mathbb{R}_-, & x_i=l_i<u_i,\\
    \mathbb{R}_+, & l_i<x_i=u_i,\\
    \mathbb{R}, & x_i=l_i=u_i.
    \end{cases}
    \]
Consider as follows, 
if \(l_i<x_i<u_i\), then complementarity gives
$\lambda_i^+=0,\;\lambda_i^-=0$, so $-\lambda_i=0\in\{0\}=\partial \mathcal{I}_{X_i}(x_i)$.

If \(x_i=l_i<u_i\), then the upper bound is inactive, hence \(\lambda_i^-=0\), so $ -\lambda_i=-(\lambda_i^+-\lambda_i^-)= -\lambda_i^+ \le 0$,
which implies
$-\lambda_i\in \mathbb{R}_-=\partial \mathcal{I}_{X_i}(x_i).$

If \(l_i<x_i=u_i\), then the lower bound is inactive, hence \(\lambda_i^+=0\), so $-\lambda_i=-(\lambda_i^+-\lambda_i^-)=\lambda_i^- \ge 0,$ which implies
$
-\lambda_i\in \mathbb{R}_+=\partial \mathcal{I}_{X_i}(x_i).
$

If \(x_i=l_i=u_i\), then
$
\partial \mathcal{I}_{X_i}(x_i)=\mathbb{R},
$
so \(-\lambda_i\in \partial \mathcal{I}_{X_i}(x_i)\).

\item[(2)] If $l_i=-\infty$ and $u_i=+\infty$, then $X_i=\mathbb{R}$, so 
$\partial \mathcal{I}_{X_i}(x_i)=\{0\}.$
In this case $\lambda_i=0$, and thus
$-\lambda_i=0\in \partial \mathcal{I}_{X_i}(x_i).$

\item[(3)] If $l_i=-\infty$ and $u_i\in\mathbb{R}$, then
$$\partial \mathcal{I}_{X_i}(x_i)=
\begin{cases}
\{0\}, & x_i<u_i,\\
\mathbb{R}_+, & x_i=u_i.
\end{cases}
$$
Since $\lambda_i=-\lambda_i^-$ with $\lambda_i^-\ge 0$, we have $-\lambda_i\geq 0$. Moreover, complementarity implies $\lambda_i^-=0$ if $x_i<u_i$. Therefore,
$
    -\lambda_i\in \partial \mathcal{I}_{X_i}(x_i).$

\item[(4)] If $l_i\in\mathbb{R}$ and $u_i=+\infty$, then
\[
\partial \mathcal{I}_{X_i}(x_i)=
    \begin{cases}
    \{0\}, & x_i>l_i,\\
    \mathbb{R}_-, & x_i=l_i.
    \end{cases}
    \]
    Since $\lambda_i=\lambda_i^+$ with $\lambda_i^+\ge 0$, we have $-\lambda_i\le 0$. Moreover, complementarity implies $\lambda_i^+=0$ if $x_i>l_i$. Therefore,$
    -\lambda_i\in \partial \mathcal{I}_{X_i}(x_i).$

\end{enumerate}

Hence, for every $i$, we have $-\lambda_i\in \partial \mathcal{I}_{X_i}(x_i)$, and therefore,
$
-\lambda\in \partial \mathcal{I}_X(x).
$

\end{proof}

\subsection{Proof of Eq.~\eqref{lag-mini}, i.e.,
 $\mathcal{I}_{Q}(Kx)=\mathcal{I}^*_{Y}(-Kx+q)$}
\label{appen:IQKX}
\begin{proof}{Proof}
The definition of $\mathcal{I}^*_{Y}(-Kx+q) $, yields,
\begin{align*}
     \mathcal{I}^*_{Y}(-Kx+q)=&\sup_{y\in\mathbb{R}^m
     }\
     \Big\{\langle -Kx+q,y\rangle -\mathcal{I}_{Y}(y)\Big\}=\sup_{y\in Y}\Big\{\langle -Kx+q,y\rangle\Big\}\\
     =&\sup_{y\in Y} \sum_{i=1}^{m_1+m_2} (-Kx_i+q_i)^{\top}y_i.
 \end{align*}
Recall that $Y:=\mathbb{R}_+^{m_1} \times \mathbb{R}^{m_2}$, which means the dual variable \( y \) has  non-negativity constraints on its first \( m_1 \) components, and is unconstrained on the remaining \( m_2 \) components. Analyzing the summation in the $\sup$ component-wise:
\begin{itemize}
    \item[(a)] For index $i\in [1,m_1]$, where $y_i\geq 0$,
    \begin{itemize}
        \item[(1)] If $-Kx_i+q_i<0$, then we would get $\sup_{y\in Y}\sum_{i} (-Kx_i+q_i)^{\top}y_i=0.$
        \item[(2)] If $-Kx_i+q_i=0$, then we would get $\sup_{y\in Y}\sum_{i} (-Kx_i+q_i)^{\top}y_i=0.$
        \item[(3)] If $-Kx_i+q_i>0$, then we would get $\sup_{y\in Y}\sum_{i} (-Kx_i+q_i)^{\top}y_i=+\infty.$
    \end{itemize}
    
    \item[(b)] For index $i\in [m_1+1,m_1+m_2]$, where $y_i$ is free,
    \begin{itemize}
        \item[(1)] If $-Kx_i+q_i<0$, then we would get $\sup_{y\in Y}\sum_{i} (-Kx_i+q_i)^{\top}y_i=+\infty.$
        \item[(2)] If $-Kx_i+q_i=0$, then we would get $\sup_{y\in Y}\sum_{i} (-Kx_i+q_i)^{\top}y_i=0.$
        \item[(3)] If $-Kx_i+q_i>0$, then we would get $\sup_{y\in Y}\sum_{i} (-Kx_i+q_i)^{\top}y_i=+\infty.$
    \end{itemize}
\end{itemize}
In summary, for all $i\in [1,m_1+m_2]$, when $-Kx_i+q_i=0$ or when $-Kx_i+q_i<0$ for $i\in [1:m_1]$, the supremum of $\sum_{i} (-Kx_i+q_i)^{\top}y_i=0$, otherwise,  $\sup_{y\in Y}\sum_{i} (-Kx_i+q_i)^{\top}y_i=+\infty$. We obtained, hence, 
\[
\mathcal{I}^*_{Y}(-Kx + q) =
\begin{cases}
0, & \text{if } (Kx)_i \geq q_i \text{ for } i = 1, \dots, m_1 \ \\&\text{ and } (Kx)_i = q_i \text{ for } i = m_1+1,\dots, m_1+m_2 , \\
+\infty, & \text{otherwise}
\end{cases}
=: \mathcal{I}_Q(Kx),
\]
where $Q$ is the set $$Q:=\{Kx\mid (Kx)_i=q_i \text{ for all } i = m_1+1,\dots,m_1+m_2,\;(Kx)_i \geq  q_i \text{ for } i = 1, \dots, m_1 \}.$$
\end{proof}

\subsection{Proof of Theorem \ref{th:pdhg-residual-0}}\label{appen:fix-res}
\begin{proof}{Proof}
    From \cite[Remark 2.7]{degenerate2021}, we know that if $\ttt$ is $\M$-firmly non-expansive, then the operator $\mathcal{R}:=2\ttt-\mathcal{I}$ is $\M$ non-expansive, and $Fix\mathcal{R}=Fix\ttt$. 
Hence from the iteration $ u^{k+1}=\ttt (u^{k})$, we have \begin{align*}
    u^{k+1}=\ttt (u^{k})=\frac{\mathcal{R}+\mathcal{I}}{2}u^{k}.
\end{align*}
and   {hence}
\begin{align*}
    u^{k+1}=u^{k} + \frac{1}{2}(\mathcal{R}(u^{k})-u^{k}).
\end{align*}
Assume $u^*$ is a fixed-point of $\ttt$, then $u^*\in Fix\ttt=Fix\mathcal{R}$. We have hence,
\begin{align}
    \|u^{k+1}-u^*\|_\M^2&=\|\frac{1}{2}(u^{k}-u^*) + \frac{1}{2}(\mathcal{R}(u^{k})-u^*)\|_\M^2\nonumber
    \\
    &=\frac{1}{2}\|u^{k}-u^*\|_\M^2+\frac{1}{2}\|\mathcal{R}(u^{k})-u^*\|_\M^2-\frac{1}{4}\|u^{k}-\mathcal{R}(u^{k})\|^2_\M\nonumber\\
    &=\frac{1}{2}\|u^{k}-u^*\|_\M^2+\frac{1}{2}\|\mathcal{R}(u^{k})-\mathcal{R}(u^*)\|_\M^2-\frac{1}{4}\|u^{k}-\mathcal{R}(u^{k})\|_\M^2\nonumber\\
    &\leq \|u^{k}-u^*\|^2_\M-\frac{1}{4}\|\mathcal{R}(u^{k})-u^{k}\|^2_\M\label{ineq:summble-p},
\end{align}
where the second equality holds due to the fact that for all $\alpha\in \mathbb{R}$ and for any $u_1,\;u_2\in\mathbb{R}^n$, $$\alpha(1-\alpha)\|u_1-u_2\|_\M^2+\|\alpha u_1+(1-\alpha)u_2\|_\M^2=\alpha \|u_1\|_\M^2+(1-\alpha)\|u_2\|_\M^2,$$ whereas the last inequality holds because of the non-expansive property of $\mathcal{R}$. From \eqref{ineq:summble-p}, we have
\begin{align*}
    \sum_{k=0}^{+\infty}\frac{1}{4}\|\mathcal{R}(u^{k})-u^{k}\|_\M^2\leq \| u^{0}-u^*\|_\M^2,
\end{align*}
and hence we have $\lim_{k\to +\infty} \|\mathcal{R}(u^{k})-u^{k}\|_\M=0.$ From the definition of $\mathcal{R}$, we get $$\lim_{k\to \infty} \|\ttt( u^{k})- u^{k}\|_\M=0,$$ and the thesis follows by using the equivalence of norms in finite dimension and the fact that $\M$ is supposed to be positive definite.
\end{proof}

\subsection{Proof of Lemma \ref{le-AA-use-inifite}}\label{appen:AA-infi}
\begin{proof}{Proof}
     Assume that AA is used only a finite number of times, i.e., there exists a finite integer $\bar{A}$, such that $i^k\leq \bar{A}$ for all $k\in \mathbb{N}$. 
    Then, there exists $\bar{K}\geq \bar{A}$ s.t. for all $k > \bar{K}$, we only use the update at Line~\ref{al-update-PDHG-uk+1} in Algorithm \ref{al-PD-safe}. Applying now Theorem \ref{th:pdhg-residual-0}, 
    we have $\lim_{k\to\infty}\|{g}^k\|= \lim_{k\to\infty}\|\ttt (u^{k}) - u^{k}\|=0$. Hence, $\forall \; \delta>0$, there exists an index $K> \bar{K}$ such that for all $k\geq K$, $\|{g}^k\|\leq\delta$. Letting 
     $$ \delta \leq D \|g^0\| (i^{{k}}+1)^{-(1+\varepsilon)}\leq D \|g^0\| (\bar{A}+1)^{-(1+\varepsilon)},$$
      we have for all ${k}\geq K$, $\|g^k\|\leq\delta\leq D \|g^0\| ({\bar{A}+1})^{-(1+\varepsilon)}$, i.e., that the safeguard condition is satisfied. This contradicts our assumption, as it implies that the safeguard check would be successful one more time.
 \end{proof}

\subsection{Proof of Lemma \ref{le-quasi-fejer}}\label{appen:AA-fejer}

\begin{proof}{Proof}
{{{We divide the discussion into two cases.
Firstly, assume the $k$-th iteration is an AA step, i.e., $k=a_{l}\in K_{aa}$. We have
\begin{equation}
\begin{aligned}\label{eq:Kaa-AA-used}
    \|u^{k+1}-u^*\|_{\M}\leq &\| u^{k} - u^* \|_{\M} + \| \p (u^{k}-H^kg^k)-u^{k}\|_{\M}\\
    \leq &\|u^k - u^*\|_{\M} + 
      {\sqrt{\lambda_{\max}(\M)}\| \p (u^{k}-H^kg^k)-u^{k}\|}\\
     \leq &\|u^k - u^*\|_{\M} +{  { \sqrt{\lambda_{\max}(\M)}}}\|H^kg^k\|\\
     \leq &\|u^k - u^*\|_{\M} +{  { \sqrt{\lambda_{\max}(\M)}}}M\|g^k\|
     \\
    \leq & \|u^k - u^*\|_{\M} + M{  { \sqrt{\lambda_{\max}(\M)}}D}\|g^0\|(i^k+1)^{-(1+\varepsilon)},
\end{aligned}
\end{equation}
where the second inequality follows by the non-expansiveness of the projection, whereas the last inequality holds due to the safeguard strategy in Algorithm \ref{al-PD-safe},  {and $\lambda_{\max}(\M)$ is the biggest eigenvalue of $\M$}.

Analogously, if $k=p_{l}\in K_{pd}$, we have,
\begin{equation}
\begin{aligned}\label{eq:kpd_PDHG-used-1}
     \|u^{k+1}-   u^*\|^2_\M&=\|\ttt (u^{k})-\ttt (u^{*})\|^2_\M \leq \|  u^{k}-u^*\|^2_\M - \|g^{k}\|^2_\M,
\end{aligned}
\end{equation}
where the first inequality holds because \( \mathcal{T} \) is \( \mathcal{M} \)-firmly non-expansive and applying \cite[Proposition 4.25(iii)]{convex-monotone} with \( \alpha = \frac{1}{2} \).  Then we have, 
\begin{equation}
\begin{aligned}\label{pu-bound-1}
&\|u^{k+1}- u^*\|_\M - \|u^{0}- u^*\|_\M = \sum_{j=0}^k\big(\|u^{j+1}- u^*\|_\M - \|u^{j}- u^*\|_\M\big)  \\
&\quad= \sum_{\substack{j \in K_{aa}\\ j \le k}}\!\big(\|u^{j+1}- u^*\|_\M - \|u^{j}- u^*\|_\M\big)
     + \sum_{\substack{j \in K_{ {pd}}\\ j \le k}}\!\big(\|u^{j+1}- u^*\|_\M - \|u^{j}- u^*\|_\M\big) \\
&\quad\le \sum_{\substack{j \in K_{aa}\\ j \le k}}\!\big(\|u^{j+1}- u^*\|_\M - \|u^{j}- u^*\|_\M\big)  \\
&\quad\le M{  { \sqrt{\lambda_{\max}(\M)}}} D\|g^0\|\sum_{j \in K_{aa}}^k (i^{j}+1)^{-(1+\varepsilon)} \\
&\quad\le M{  { \sqrt{\lambda_{\max}(\M)}}} D\|g^0\|\sum_{j=0}^{\infty} (j+1)^{-(1+\varepsilon)} < +\infty .
\end{aligned}
\end{equation}
where the first inequality holds from \eqref{eq:kpd_PDHG-used-1} and the second inequality holds from \eqref{eq:Kaa-AA-used}.
Hence, we get $\{\| u^{k}- u^*\|_\M\}$ is bounded, and we denote this bound by \[
    E:=\| u^{0}- u^*\|_\M + M{{  { \sqrt{\lambda_{\max}(\M)}}} }D\|g^0\|\sum_{j=0}^{\infty}(j+1)^{-(1+\varepsilon)},\]
i.e., $\| u^{k}- u^*\|_\M\leq E$ for all $k \geq 0$. Clearly, this implies that $\{\|u^k-u^*\|\}$ is also bounded. 

Squaring both sides of equation \eqref{eq:Kaa-AA-used}, we have
\begin{equation}
\begin{aligned}\label{uk-sequence-diff-bound}
   & \| u^{{k+1}}- u^*\|^2_\M- \| u^{k}- u^*\|^2_\M \\
   &\leq 2\| u^{k}- u^*\|_\M M{{  { \sqrt{\lambda_{\max}(\M)}}} }D\|g^0\|(i^k+1)^{-(1+\varepsilon)}+M^2D^2{\lambda_{\max}(\M) }\|g^0\|^2(i^k+1)^{-2(1+\varepsilon)}\\
    &\leq 2EM{  { \sqrt{\lambda_{\max}(\M)}}} D\|g^0\|(i^k+1)^{-(1+\varepsilon)}+M^2D^2{\lambda_{\max}(\M) }\|g^0\|^2(i^k+1)^{-2(1+\varepsilon)}.
\end{aligned}
\end{equation}
Let us now   {define} the sequence
\begin{align*}
\varepsilon^{k} =
\begin{cases}
\begin{aligned}
&2EM{  { \sqrt{\lambda_{\max}(\M)}}}D\|g^0\|(i^k+1)^{-(1+\varepsilon)} \\
&\quad + M^2D^2\lambda_{\max}(\mathcal{M}) \|g^0\|^2(i^k+1)^{-2(1+\varepsilon)},
\end{aligned}
& k \in K_{aa},\\
0, & k \notin K_{aa}.
\end{cases}
\end{align*}

Using  \eqref{eq:kpd_PDHG-used-1} and \eqref{uk-sequence-diff-bound}, by the definition of $\varepsilon^k$, it holds $ \| u^{{k+1}}- u^*\|^2_\M\leq \| u^{k}- u^*\|^2_\M+\varepsilon^{k}$ for all $k\geq 0$. Moreover, we have
\begin{equation}
\begin{aligned}
&\sum_{j=0} \varepsilon^{j}=\sum_{j\in   {K_{pd}}} \varepsilon^{j}+\sum_{j\in K_{aa}}\varepsilon^{j}\\
=&\sum_{{j}\in K_{aa}} \Big(2EM{  { \sqrt{\lambda_{\max}(\M)}}} D\|g^0\|(i^j+1)^{-(1+\varepsilon)}+M^2D^2{\lambda_{\max}(\M) }\|g^0\|^2(i^j+1)^{-2(1+\varepsilon)}\Big)\\
\leq &\sum_{j=0}^{\infty}\Big(2EM{  { \sqrt{\lambda_{\max}(\M)}}} D\|g^0\|(j+1)^{-(1+\varepsilon)}+M^2D^2{\lambda_{\max}(\M) }\|g^0\|^2(j+1)^{-2(1+\varepsilon)}\Big)<+\infty,
\end{aligned}
\end{equation}
where the second equality holds due to the definition of $\varepsilon^{k}$. Hence, the $\{\varepsilon^k\}$ is a summable sequence, which proves, looking at Definition~\ref{def:quasi-fejer}, that $\{ u^k\}$ is a $\M$-quasi-Fejér monotone sequence. 
}}}
\end{proof}

\subsection{ {Hyperparameter Selection for AA-PDHG}}\label{appen:hyper}

 {This section describes the systematic hyperparameter selection procedure for AA-PDHG. To avoid overfitting the parameter choices to a specific problem structure, we conduct all tuning experiments on a representative subset of 50 randomly selected pre-solved instances from the MIPLIB 2017 dataset. The selection follows a sequential strategy: each parameter is tuned in turn while keeping the previously selected values fixed. At each step, the candidate configurations are evaluated using performance profiles \cite{ASPerprof2016}, which provide a comprehensive comparison in terms of both efficiency (the fraction of problems for which a solver is fastest) and robustness (the fraction of problems solved within a given factor of the best solver's time/iterations). We report performance profiles for both running time and iteration count.}

\medskip\noindent\textbf{1.  {Memory size $m_A$.}}  {The AA memory size $m_A$ determines how many past iterates are used to construct the acceleration step. A larger $m_A$ provides a richer Quasi-Newton approximations but increases the per-iteration cost of the AA least-squares subproblem and the storage overhead. We compare $m_A \in \{3, 5, 7\}$; the corresponding time and iteration performance profiles are reported in Figures~\ref{fig:mk35} and~\ref{fig:itmk35}, where pre-AA3, pre-AA5, and pre-AA7 denote AA-PDHG with $m_A = 3, 5, 7$, respectively.}

      {The time performance profiles in Figure~\ref{fig:mk35} show that $m_A = 5$ achieves the best trade-off between efficiency and robustness: it is the fastest solver on the largest fraction of instances while maintaining competitive robustness. Although $m_A = 7$ yields fewer iterations overall, as shown in Figure~\ref{fig:itmk35}, the increased per-iteration cost of the larger AA subproblem offsets this gain in terms of wall-clock time. Conversely, $m_A = 3$ provides insufficient historical information to achieve effective acceleration on many instances. Based on these observations, we fix $m_A = 5$ for all subsequent experiments.}

\begin{figure}[htb!]
    \centering
    \begin{subfigure}{0.3\linewidth}
        \centering
        \includegraphics[width=\linewidth]{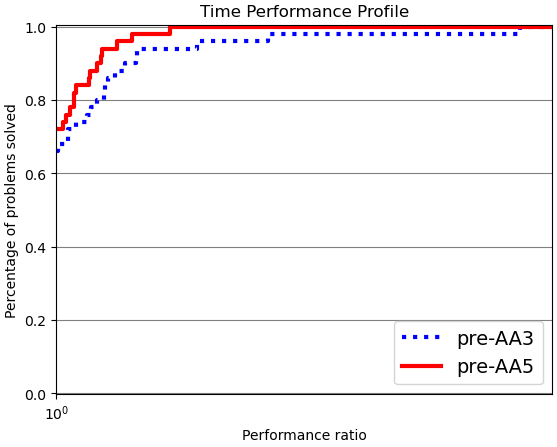}
        \caption{Compare $m_A =3$, $m_A =5$}
    \end{subfigure}
    \hfill
    \begin{subfigure}{0.3\linewidth}
        \centering
        \includegraphics[width=\linewidth]{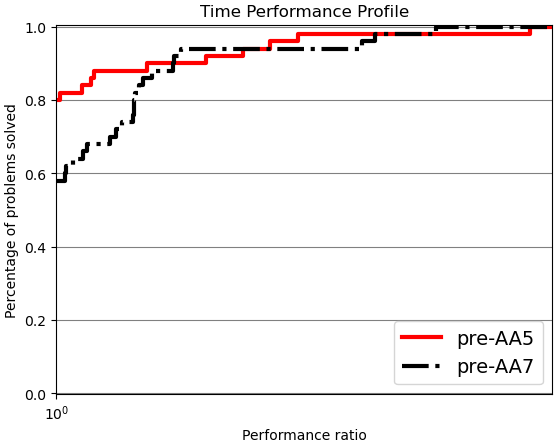}
         \caption{Compare $m_A =5$, $m_A =7$}
    \end{subfigure}
    \hfill
    \begin{subfigure}{0.3\linewidth}
        \centering
        \includegraphics[width=\linewidth]{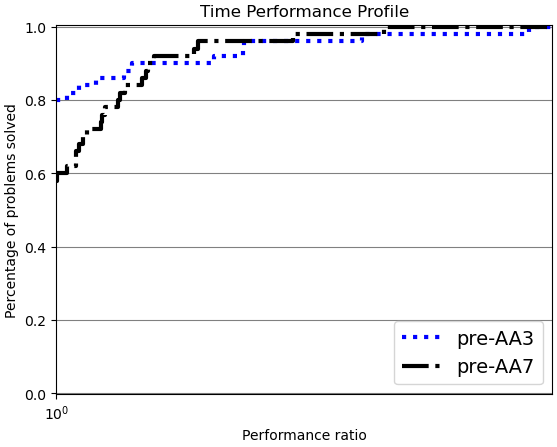}
         \caption{Compare $m_A =3$, $m_A =7$}
    \end{subfigure}
    \caption{ {Time performance profiles for AA-PDHG with different memory sizes $m_A \in \{3, 5, 7\}$.}}
    \label{fig:mk35}
\end{figure}

\begin{figure}[htb!]
    \centering
    \begin{subfigure}{0.3\linewidth}
        \centering
        \includegraphics[width=\linewidth]{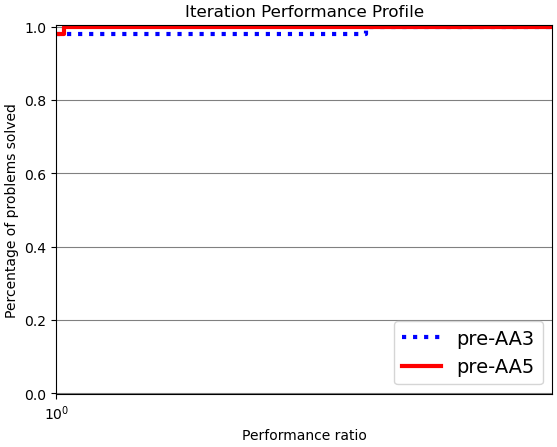}
        \caption{Compare $m_A =3$, $m_A =5$}
    \end{subfigure}
    \hfill
    \begin{subfigure}{0.3\linewidth}
        \centering
    \includegraphics[width=\linewidth]{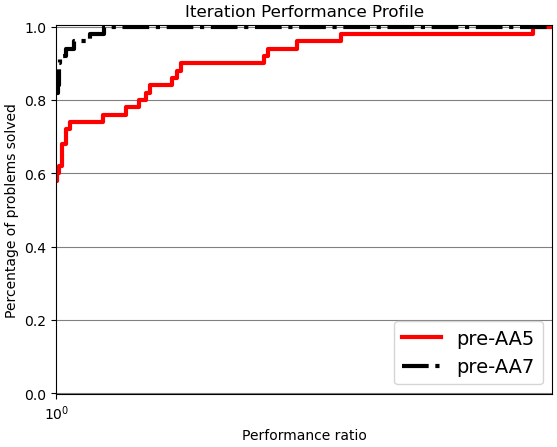}
         \caption{Compare $m_A =5$, $m_A =7$}
    \end{subfigure}
    \hfill
    \begin{subfigure}{0.3\linewidth}
        \centering
           \includegraphics[width=\linewidth]{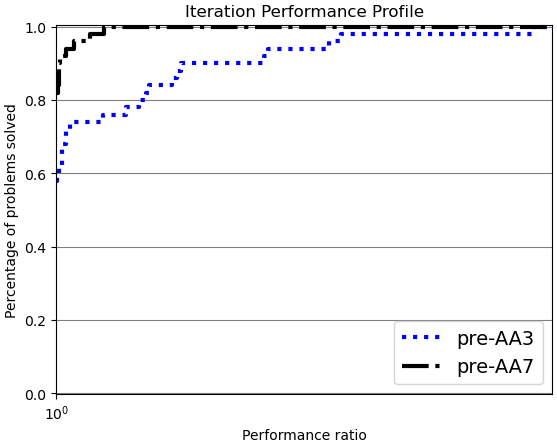}
         \caption{Compare $m_A =3$, $m_A =7$}
    \end{subfigure}
    \caption{ {Iteration performance profiles for AA-PDHG with different memory sizes $m_A \in \{3, 5, 7\}$.}}
    \label{fig:itmk35}
\end{figure}

\medskip\noindent\textbf{2.  {Regularisation parameter $\eta$.}}  {The Tikhonov regularisation parameter $\eta$ in Line~\ref{algline:AA_update} of Algorithm~\ref{al-PD-safe} controls the conditioning of the AA least-squares subproblem. Larger values of $\eta$ improve numerical stability at the cost of reduced acceleration, while smaller values allow more aggressive acceleration but may lead to ill-conditioned updates. We evaluate three candidate values, $\eta \in \{10^{-6}, 10^{-8}, 10^{-10}\}$, with $m_A = 5$ fixed from the previous step. The time and iteration performance profiles are reported in Figures~\ref{fig:lambda_1e810} and~\ref{fig:lambda_1e810_it}.}

\begin{figure}[htb!]
    \centering
    \begin{subfigure}{0.3\linewidth}
        \centering
         \includegraphics[width=\linewidth]{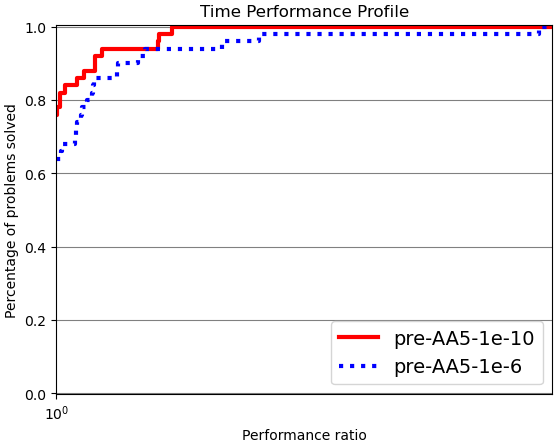}
        \caption{$\eta =10^{-6}$ vs $\eta = 10^{-8}$}
    \end{subfigure}
    \hfill
    \begin{subfigure}{0.3\linewidth}
     \includegraphics[width=\linewidth]{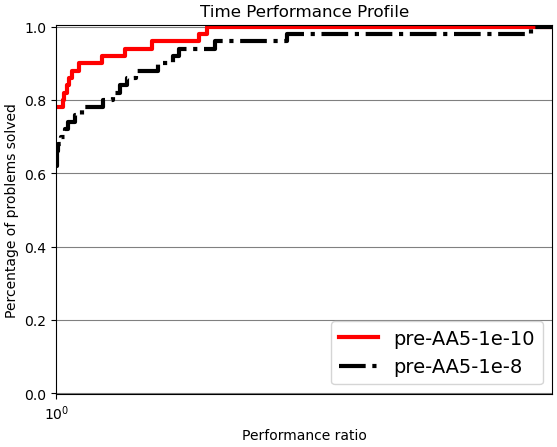}
        \caption{$\eta =10^{-10}$ vs $\eta = 10^{-8}$}
    \end{subfigure}
    \hfill
    \begin{subfigure}{0.3\linewidth}
     \includegraphics[width=\linewidth]{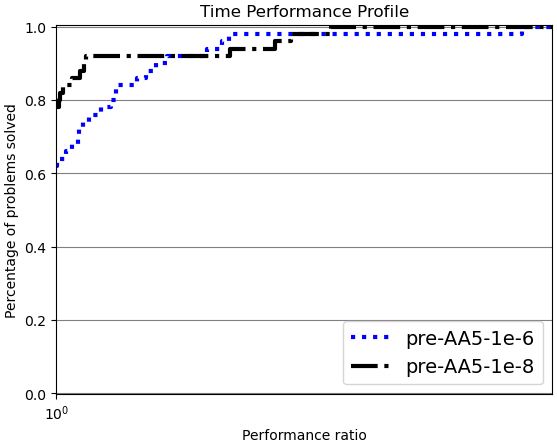}
        \caption{$\eta =10^{-10}$ vs $\eta = 10^{-6}$}
    \end{subfigure}
    \caption{  {Time performance profiles for AA-PDHG ($m_A = 5$) with different regularisation parameters $\eta \in \{10^{-6}, 10^{-8}, 10^{-10}\}$.}}
    \label{fig:lambda_1e810}
\end{figure}

\begin{figure}[htb!]
    \centering
    \begin{subfigure}{0.3\linewidth}
        \centering
         \includegraphics[width=\linewidth]{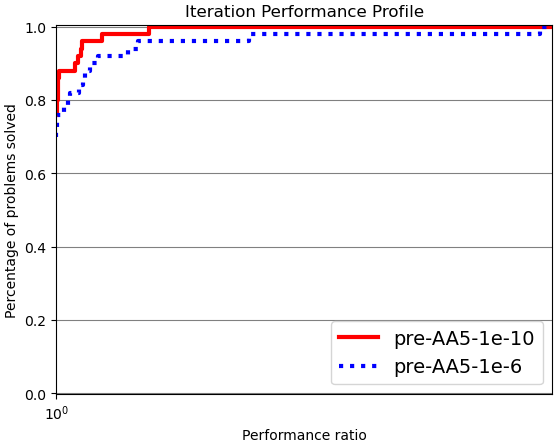}
        \caption{$\eta =10^{-10}$  vs $\eta = 10^{-6}$}
    \end{subfigure}
    \hfill
    \begin{subfigure}{0.3\linewidth}
     \includegraphics[width=\linewidth]{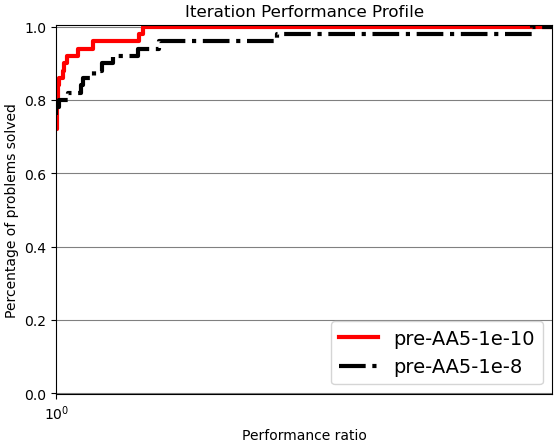}
        \caption{$\eta =10^{-10}$  vs $\eta = 10^{-8}$}
    \end{subfigure}
    \hfill
    \begin{subfigure}{0.3\linewidth}
     \includegraphics[width=\linewidth]{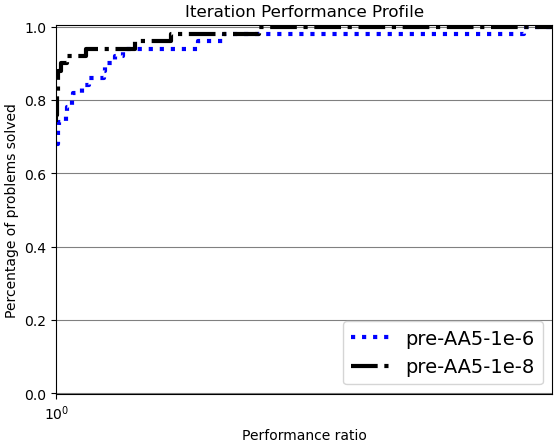}
        \caption{$\eta =10^{-6}$ vs $\eta = 10^{-8}$}
    \end{subfigure}
    \caption{  {Iteration performance profiles for AA-PDHG ($m_A = 5$) with different regularisation parameters $\eta \in \{10^{-6}, 10^{-8}, 10^{-10}\}$.}}
    \label{fig:lambda_1e810_it}
\end{figure}
 {The time performance profiles in Figure~\ref{fig:lambda_1e810} reveal a trade-off between efficiency and robustness: $\eta = 10^{-8}$ achieves the highest efficiency (it is the fastest solver on the largest fraction of instances), whereas $\eta = 10^{-10}$ exhibits superior robustness (it solves the largest fraction of instances within a moderate performance ratio). The iteration performance profiles in Figure~\ref{fig:lambda_1e810_it} provide a complementary perspective: $\eta = 10^{-6}$ requires the fewest iterations on many instances, indicating effective acceleration, but the stronger regularisation leads to higher wall-clock time per iteration.  {Since $\eta=10^{-10}$ exhibit better runtime performance, we therefore adopt this value in the subsequent experiments.}
}

\medskip\noindent\textbf{3.  {Primal weight update strategy.}}  {The primal weight update technique introduced in \cite{google2022p} dynamically rescales the primal and dual step sizes to achieve scale invariance, and has been shown to significantly improve practical convergence. However, in the AA framework, frequent updates to the primal weights alter the underlying fixed-point mapping $\mathcal{T}$, which can destabilise the difference matrices $\Delta\mathcal{U}^{k-m_k}$ and $\Delta\mathcal{G}^{k-m_k}$ and consequently degrade the acceleration effect. To balance the benefits of scale invariance against the stability requirements of AA, we adopt a periodic update strategy and evaluate two update frequencies: every $3{,}000$ and every $5{,}000$ iterations.  {We also compare two smoothing parameter values, $\theta \in \{0.1, 0.5\}$ (where $\theta = 0.5$ is the default in \cite{google2022p}, and we choose $\theta=0.1$ to examine whether this choice is better suited to AA),  and fix $\eta = 10^{-10}$. The experiment is conducted on the full pre-solved dataset. The time and iteration performance profiles are reported in Figures~\ref{fig:pwAA-2050}.}}

 \begin{figure}[htb!]
\centering

\begin{subfigure}[t]{0.4\textwidth}
\centering
\includegraphics[width=\linewidth]{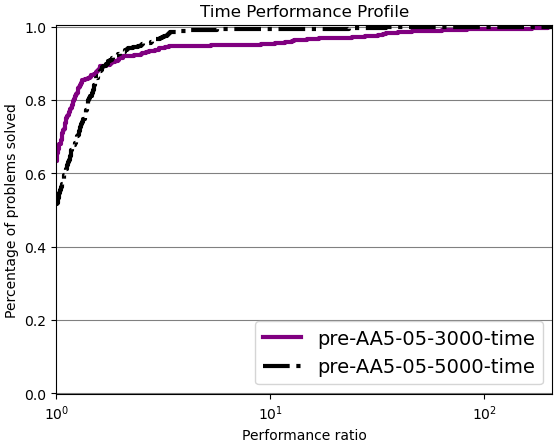}
\caption{Time comparison for different update periods..}
\end{subfigure}
\hspace{0.02\textwidth}
\begin{subfigure}[t]{0.4\textwidth}
\centering
\includegraphics[width=\linewidth]{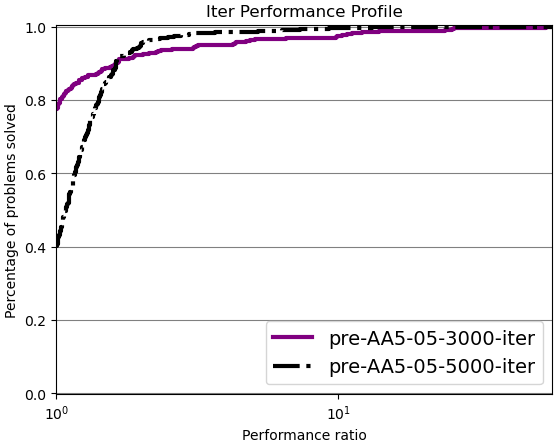}
\caption{Iteration comparison for different update periods.}
\end{subfigure}
\hspace{0.02\textwidth}
\begin{subfigure}[t]{0.4\textwidth}
\centering
\includegraphics[width=\linewidth]{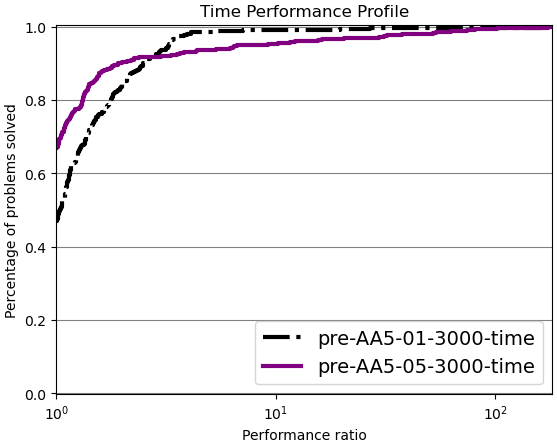}
\caption{Time comparison for different primal-weight smoothing parameters.}
\end{subfigure}
\hspace{0.02\textwidth}
\begin{subfigure}[t]{0.4\textwidth}
\centering
\includegraphics[width=\linewidth]{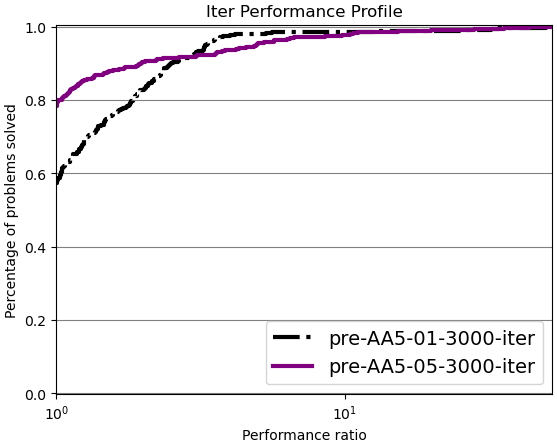}
\caption{Iteration comparison for different primal-weight smoothing parameters.}
\end{subfigure}

\caption{ {Time and iteration performance profiles for AA-PDHG ($m_A = 5$) under different primal weight update strategies. For all methods, we select $\eta=10^{-10}$.}
}
\label{fig:pwAA-2050}

\end{figure}




 {The results in Figure~\ref{fig:pwAA-2050} lead to the following observations. First, Subfigures~(a) and~(b) compare the different update periods under the same smoothing parameter $\theta=0.5$. 
In these cases, updating the primal weights every $3{,}000$ iterations yields superior performance, suggesting that moderately higher update frequency helps capture the evolving problem scaling. Second, Subfigures~(c) and (d) compare the smoothing parameters $\theta = 0.1$ and $\theta = 0.5$. The result shows that  $\theta = 0.5$ in the current AA primal-weight update appears to provide a better balance between the primal and dual progress.
} 
 {Based on the above analysis, the final hyperparameter configuration for AA-PDHG is: memory size $m_A = 5$, regularisation parameter $\eta = 10^{-10}$, smoothing parameter $\theta = 0.5$, and primal weight update every $3{,}000$ iterations. This configuration is used in all comparisons reported in Section~\ref{se:rpdhg_comparison}.}

\medskip\noindent\textbf{4.  {Update $\widehat{D}$.}} 
  {In Remark~\ref{re:pwform}, we also study the role of updating $\widehat{D}$ in the algorithm, carrying out additional experiments. Specifically, when AA-PDHG performs periodic primal-weight updates, we also update $\widehat{D}$ as described in Remark~\ref{re:pwform}. This extra step is inserted after Step~\ref{alstep:resetAA} of Algorithm~\ref{alg:smooting}. The figures below compare the performance on the full pre-solved dataset under different strategy of updating $\widehat{D}$. Here, unless explicitly labeled as \texttt{no-div}, the AA safeguard adopts $\mathrm{Iter}_{\mathrm{AA}}/c$ with $c=3$ to further control the acceptance frequency of AA steps. We see that, with the choice of updating the primal weight every 3000 iterations, together with the safeguard setting $\text{Iter}_\text{AA}/3$ yields the best performance. 
 We also adopt this parameter choice in the final numerical experiments comparing with rPDHG.
 }
 \begin{figure}[htb!]
    \centering
    \begin{subfigure}{0.4\linewidth}
        \centering
         \includegraphics[width=\linewidth]{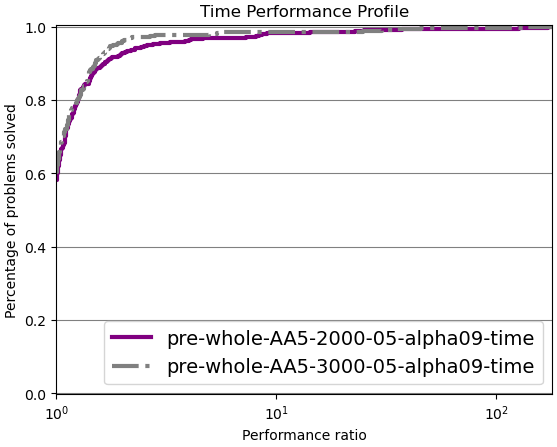}
        \caption{Update $\widehat{D}$ under different primal weight update period.}
    \end{subfigure}
    \hfill
    \begin{subfigure}{0.4\linewidth}
     \includegraphics[width=\linewidth]{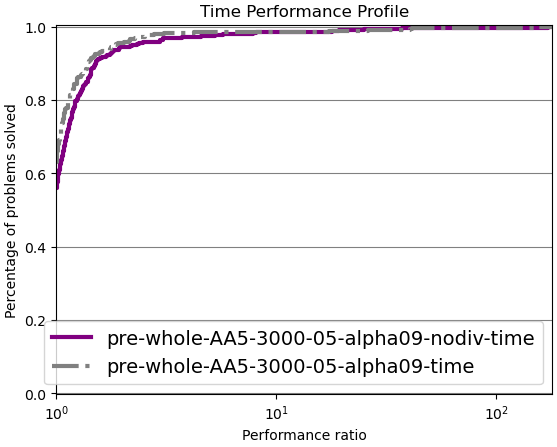}
        \caption{Update $\widehat{D}$ under different safeguard strategy.}
    \end{subfigure}
    \caption{  {Time performance profiles for AA-PDHG ($m_A = 5$) with $\widehat{D}$ update.}}
    \label{fig:pwD}
\end{figure}

\begin{figure}[htb!]
    \centering
    \begin{subfigure}{0.4\linewidth}
        \centering
         \includegraphics[width=\linewidth]{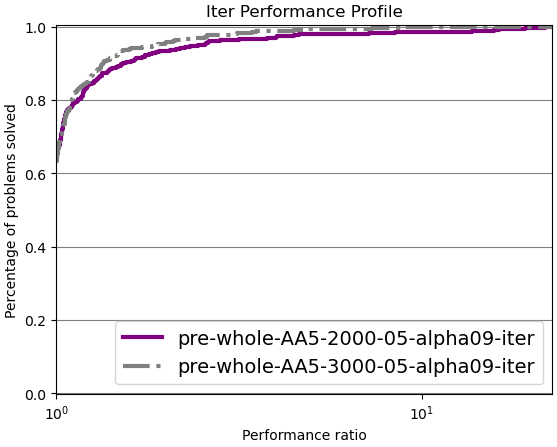}
        \caption{Update $\widehat{D}$ under different primal weight update period.}
    \end{subfigure}
    \hfill
    \begin{subfigure}{0.4\linewidth}
     \includegraphics[width=\linewidth]{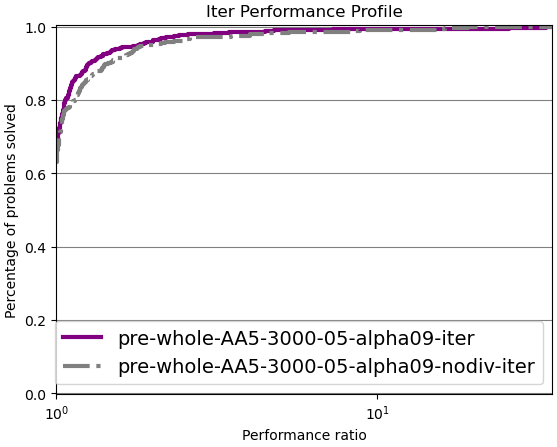}
        \caption{Update $\widehat{D}$ under different safeguard strategy.}
    \end{subfigure}
    \caption{  {Iteration performance profiles for AA-PDHG ($m_A = 5$) with $\widehat{D}$ update.}}
    \label{fig:pwD_iter}
\end{figure}

\end{document}